\documentclass{article}

\usepackage[english]{babel}

\usepackage[letterpaper,top=2cm,bottom=2cm,left=3cm,right=3cm,marginparwidth=1.75cm]{geometry}

\usepackage{amsmath}
\usepackage{graphicx}
\usepackage[colorlinks=true, allcolors=blue]{hyperref}

\usepackage{ dsfont }
\usepackage{amsfonts}
\usepackage{amsthm}
\usepackage{amsmath}
\usepackage{mdframed}
\usepackage{ amssymb, tcolorbox }
\usepackage{xcolor}
\newtheorem*{Maintheorem}{Theorem}
\newtheorem{theorem}{Theorem}[section]
\newtheorem{corollary}{Corollary}[theorem]
\newtheorem{lemma}[theorem]{Lemma}
\newtheorem{remark}{Remark}
\newtheorem{definition}[theorem]{Definition}

\def\R {\mathbb{R}}
\def\p{\partial}
\def\eps{\epsilon}

\def\P{\mathds{P}}

\title{Finite time singularities of smooth solutions for the 2D incompressible porous media (IPM) equation with a smooth source}
\author{ Diego C\'{o}rdoba\footnote{Instituto de Ciencias Matem\'aticas CSIC-UAM-UCM-UC3M, Spain. e-mail: dcg@icmat.es} \; \&  Luis Mart\'{\i}nez-Zoroa \footnote{University of Basel, Switzerland. e-mail: †
luis.martinezzoroa@unibas.ch}}

\begin{document}
\maketitle
\begin{abstract}
    We establish the existence of smooth, finite-energy solutions to the 2D incompressible porous media equation (IPM), with a compactly supported uniformly smooth source, which develop singularities in finite time.
\end{abstract}

\section{Introduction}
We consider smooth solutions in $L^2$ of the Incompressible Porous Media (IPM) system in $\mathbb{R}^2$ with a compactly supported smooth source. The IPM system is an active scalar equation given by:
\begin{equation}\label{eq:IPMsystem}\tag{IPM}
\begin{cases}
    \partial_t \rho + \mathbf{u} \cdot \nabla \rho = F, \\
    \mathbf{u} = - \frac{\mu}{\kappa} \nabla P + \mathbf{g} \rho, \quad \mathbf{g} = (0, -g)^\top, \quad \text{(Darcy's law)} \\
    \nabla \cdot \mathbf{u} = 0,
\end{cases}
\end{equation}
which models the dynamics of a fluid with density $\rho = \rho(x_1, x_2, t): \mathbb{R}^2 \times \mathbb{R}_+ \to \mathbb{R}$ flowing with a divergence free velocity field $\mathbf{u} = (u_1(x_1, x_2, t), u_2(x_1, x_2, t))$ through a porous medium in accordance with Darcy's law with a pressure $P = P(x_1, x_2, t)\in \mathbb{R}$. Here, the external source term $F = F(x,t)$ is in $C_c^\infty(\mathbb{R}^2 \times \mathbb{R}_+)$, and the constants $\mu>0$, $\kappa > 0$ and $g > 0$ represent the viscosity of the fluid, the permeability coefficient and the gravitational acceleration, respectively. For simplicity, we set $\mu=\kappa = g = 1$ in the analysis that follows. For further details on the physical background and applications of this model, we refer the reader to \cite{B} and the references therein.

In this paper we prove the following result:
\begin{mdframed}
\begin{Maintheorem}[\textbf{Finite time blowup for the IPM equation with a smooth source}]
There exists smooth compactly supported initial conditions $\rho_{0}\in C^{\infty}_{c}$ and a smooth compactly supported source term $F(x,t)\in L^{\infty}_{t}C^{\infty}_{x}$ such that the only classical solution to the IPM equation $\rho(x,t)\in C^{\infty}_{c}$ for $t\in[0,1)$ exhibits finite time blowup, that is to say
$$\text{lim}_{t\rightarrow 1}\|\rho\|_{C^1}=\infty.$$

\end{Maintheorem}
\end{mdframed}
\begin{remark}
    The regularity of the source term $F(x,t)\in L^{\infty}_{t}C^{\infty}_{x}$ is chosen to simplify the construction, but one can obtain better regularity in time by making some slight modifications in the proof (specifically, change \eqref{wt} and use a smoother function in time) and doing a more careful analysis of the time dependence of the errors. In particular, $C^{1}_{t}C^{\infty}_{x}$ regularity follows with some extra work, and one should even be able to obtain $C^{\infty}_{t}C^{\infty}_{x}$ regularity for the source term.
\end{remark}

\begin{remark}
    We have a similar behavior for the velocity $u\in C^\infty(\mathbb{R}^2)\cap L^2(\mathbb{R}^2)$ in the same time interval $[0,1)$, $$\lim_{t\rightarrow 1} |\nabla \mathbf{u}(x,t)|_{L^{\infty}} = \infty.$$
    Furthermore, due to blow-up criteria, we have that
    $$\lim_{t\rightarrow 1}\int_0^t |\nabla \mathbf{u}(x,s)|_{L^{\infty}}=\infty.$$
    As a consequence of the construction, at the time of the singularity $t=1$ we have that $\rho(x,1)\in C^{\infty}(\R^2\setminus{0})\cap C^{\alpha}$ for all $\alpha\in[0,1)$.
\end{remark}

The velocity field $\mathbf{u} = (u_1, u_2)$ of the system \eqref{eq:IPMsystem} can be reformulated in terms of a singular integral operator of degree 0 as follows:
\begin{equation}\label{eq:ipm}
\begin{cases}
    \partial_t \rho + \mathbf{u} \cdot \nabla \rho =  F, \\
    \mathbf{u} = (u_1, u_2) = -(\mathcal{R}_2, -\mathcal{R}_1) \mathcal{R}_1 \rho, 
\end{cases}
\end{equation}
where $\mathcal{R}_1$ and $\mathcal{R}_2$ denote the first and second components of the Riesz transform in two space dimensions, defined by
\begin{align}\label{eq:riesz}
    \mathcal{R}_1 = (-\Delta)^{-1/2} \partial_{x_1}, \qquad \mathcal{R}_2 = (-\Delta)^{-1/2} \partial_{x_2}.
\end{align}
The pressure is recovered from $P =(-\Delta)^{-1} \partial_{x_2}\rho$ by taking the divergence of Darcy's law and inverting the laplacian operator.

The incompressibility of the fluid ensures that classical solutions to the \eqref{eq:IPMsystem}  equation satisfy the following energy identity
\[
\frac{d}{dt} \int_{\mathbb{R}^2 } |\rho|^2 \, dx = 2\int_{\mathbb{R}^2 } \rho F \, dx,
\]
which provides a finite time-dependent bound for\(\|\rho \|_{L^2}\) and \(\|u\|_{L^2}\) with a smooth localized source $F$. Similar bounds are also satisfied for the $L^p$ norms of $\rho$ ($1\leq p \leq \infty$) and $u$ ($1<p<\infty$).

\subsection{Motivation and related work}

Active scalar equations that emerge in fluid dynamics pose a complex class of problems within the field of partial differential equations. Perhaps the most prominent two examples are the surface quasi-geostrophic (SQG) equation and the incompressible porous media (IPM) equation. In the case of SQG, the operator linking the velocity and the active scalar is given by $u = (-\mathcal{R}_2\rho, \mathcal{R}_1\rho)$, which is also a zero-order operator ($\mathcal{R}_i$ are the Riesz transforms). These two equations bear a strong analogy to the 3D incompressible Euler equations, sharing a remarkable number of similarities (see \cite{CMT}, \cite{C}, \cite{CGO}, \cite{EIS}, \cite{IV} and \cite{K}).

The problem of whether smooth solutions to the incompressible Euler, SQG, and IPM equations remain globally regular or develop singularities in finite time has been extensively studied by numerous authors over the years and is regarded as a significant open question in the field of partial differential equations.

The main motivation of this work is to provide the first blow-up result of smooth solutions of an incompressible  fluid with finite energy.   In this work we focus in the two-dimensional incompressible porous media (IPM) equation, which models the evolution of density transported by the flow of an incompressible fluid, governed by Darcy's law in the presence of a gravitational field. IPM is an active scalar equation where the scalar is driven by a divergence free velocity field and both are related by a singular integral operator of order zero $u=(u_1, u_2) = -(\mathcal{R}_2, -\mathcal{R}_1) \mathcal{R}_1 \rho$.

\paragraph{Local well-posedness for IPM.} Due to the properties of the Riesz transforms it is relatively direct by known methods to show that the IPM equation is locally in time well-posed (see \cite{CGO} and \cite{C}) in Holder spaces $\rho_0=\rho(x,t=0)\in C^{k,\alpha}(\mathbb{R}^2)\cap L^2 (\mathbb{R}^2)$ (for $k\geq 1$ and $\alpha\in (0,1)$) and in Sobolev spaces $\rho_0\in H^s(\mathbb{R}^2)$ for $s>2$.

\paragraph{Ill-posedness at critical regularity for IPM.}  Recently, the authors of this paper in collaboration with Bianchini  showed in \cite{BCM} that there exists $C^1$ solutions on a time interval $[0,T]$ for some $T>0$ such that at initial time $\rho_0\in H^2(\mathbb{R}^2)$ but for $t>0$ the solution is not in $H^2$ anymore. Furthermore, these solutions are unique in $H^1$ for the time interval $[0,T]$. In previous work \cite{FGSV}, a class of active scalar equations where the transport velocities is more singular by a derivative of order $\beta$ compared to the active scalar itself is studied. It is proved that, for $0 < \beta \leq 2$, these equations are ill-posed in the Lipschitz sense when the initial data is regular.

\paragraph{Stability results for IPM from a stable profile.} Very interesting results were obtain specifically for perturbations of the spectrally stable steady state $$ \bar\rho_{\text{stable}}(x) = -x_2.$$ Elgindi shows in \cite{Elgindi4} global-in-time existence of small perturbations of the stable profile for the IPM equation in $H^{k}(\R^2)$ for $k$ sufficiently large. See also \cite{CCL} for a stability result in the strip case and for improvements on the regularity see \cite{BCP}, \cite{P} and \cite{JK}. It is highly probable that asymptotic stability is maintained in $H^{2+}$, as suggested in \cite{BJP}. Notably, this insight underscores the importance of the findings in \cite{BCM}, where it is shown that reducing the regularity assumption to $H^2$  leads to a complete breakdown of stability near the stable profile, with solutions exhibiting blow-up at any moment in time $t>0$. 

\paragraph{Small scale formation for IPM.} In \cite{KY}, Kiselev and Yao construct smooth solutions to the incompressible porous media (IPM) equation that have unbounded growth of derivatives over time, assuming the solutions stay smooth. As a result, they are able to establish nonlinear instability for a certain class of stratified steady states of the IPM.

\paragraph{Finite time blow-up results of classical solutions for 3D incompressible Euler equations.} Recently, various blow-up scenarios have been developed within the locally well-posed regime $C^{1,\alpha}$ for the 3D incompressible Euler equations. Specifically, there are four known cases of blow-up involving finite energy and no boundary:

\begin{itemize}

    \item The first blow-up result is due to Elgindi in \cite{Elgindi2} and Elgindi, Ghoul and Masmoudi in \cite{Elgindi3}. They construct self similar blow up for axi-symmetric flows and no swirl, focusing on velocity profiles within the $C^{1,\alpha}$ H{\"o}lder space, where $\alpha$ is chosen to be a very small 
number, and which are $C^\infty$ almost everywhere.  

    \item In our recent work in collaboration with Zheng \cite{CMZ}, we presented a novel blowup mechanism for axi-symmetric flows without swirl that does not depend on self-similar profiles. These solutions remain smooth everywhere except at a singular point, where the solution is $C^{1,\alpha}$ and $\alpha$ being a very small value (see also \cite{Che2}). 
    
    \item In the recent study \cite{CM}, we propose a blow-up mechanism for the forced 3D incompressible Euler equations, focusing specifically on non-axisymmetric solutions. We construct solutions in $\mathbb{R}^3$ within the function space $C^{3,\frac{1}{2}} \cap L^2$ over the time interval $[0, T)$, where $T > 0$ is finite, subject to a continuous force in $C^{1,\frac{1}{2} -\epsilon} \cap L^2$. This framework results in a blow-up: as time $t$ approaches the critical value $T$, the integral $\int_0^t |\nabla u| \, ds$ diverges, indicating unbounded growth. Despite this, the solution remains smooth everywhere except at the origin. Importantly, our blow-up construction does not rely on self-similar coordinates and extends to solutions beyond the critical $C^{1,\frac{1}{3}+}$ regularity for global well-posedness in axi-symmetric flows without swirl. Furthermore, following the strategy of \cite{CM} we construct finite time singularities, in \cite{CMZ2} in collaboration with Zheng,  for the hypodissipative 3D Navier-Stokes equations  with an external forcing which is in $L^1_t C_x^{1,\epsilon}\cap L^{\infty}_{t} L^{2}_{x}$.

    \item In a very recent study, Elgindi and Pasqualotto examined the case of axi-symmetric flows with swirl and the 2D Boussinesq system (both systems are strongly related). In \cite{EP}, they prove a self-similar blow-up profile for solutions in $C^{1, \beta}$ (with $\beta$ very small), where the singularity occurs away from the origin. The solutions they considered belong to $H^{2+\delta}$, which places them above the critical regularity for axi-symmetric solutions, in a stronger sense than the usual $C^{1, \beta}$ solutions discussed in previous works on self-similar profiles.
    
\end{itemize}

It is important to highlight that all the previously mentioned blow-ups rely heavily on a $C^{1,\alpha}$ framework, whereas the blow-up we present in this paper for the IPM concerns $C^\infty \cap L^2$ solutions with a uniformly smooth source.

    Recent developments have greatly enhanced our understanding of self-similar singularities in axi-symmetric flows, particularly in the presence of boundaries. These advances have also paved the way for computer-assisted proofs in this area (see \cite{EJ}, \cite{Hou}, \cite{Hou2}, \cite{Hou4} and \cite{WLGB}).

\subsection{Ideas of the proof}
There are several key ideas that we need to apply in our construction in order to produce our smooth blow up. Through this subsection we will try to explain them in some detail.
\subsubsection{The use of different scales}\label{subsecscales}
The first important idea that we will use through our paper consists on studying solutions that are composed by different pieces with very different spatial scales. Say, for example, that we are considering a solution to \eqref{eq:ipm} $\rho(x,t)$ with $F=0$, and that we want to perturb this solution by adding some small perturbation $\rho_{pert}(x,t)$. One can, by direct computation, check that if $\rho(x,t)+\rho_{per}(x,t)$ is a solution to \eqref{eq:ipm} with $F(x,t)=0$, then
$$\p_{t}\rho_{pert}(x,t)+u(\rho_{pert})\cdot\nabla (\rho_{pert}(x,t)+\rho(x,t))+u(\rho(x,t))\cdot\nabla\rho_{pert}(x,t)=0.$$
Let us imagine that we choose $\rho_{pert}(x,t=0),\rho(x,t)$ with the right symmetries (namely, with the symmetries given by Definition \ref{odd}), so that $\rho(x=0,t),u(\rho(\cdot,t))(x=0)=0$. Under these conditions, if we consider now that both our perturbation $\rho_{pert}(x,t)$ and its velocity $u(\rho_{pert})$ are very concentrated around the origin for all times (assumption which is, of course, not necessarily true and in fact almost always false), we would expect the approximations
$$u(\rho_{pert})\cdot\nabla \rho(x,t)\approx u(\rho_{pert})\cdot(\nabla \rho)(0,t)$$
$$u_{i}(\rho(x,t))\p_{x_{i}}\rho_{pert}(x,t)\approx x_{1}(\p_{x_{1}}u_{i}(\rho(\cdot,t))(x=0))\p_{x_{i}}\rho_{pert}(x,t)+x_{2}(\p_{x_{2}}u_{i}(\rho(\cdot,t))(x=0))\p_{x_{i}}\rho_{pert}(x,t) $$
to be reasonable. This suggests studying the simplified system obtained by applying this approximation. Although, even with this approximation, this equation is still too hard to study in general, if we consider $\rho(x,t)\approx x_{2}$ around the origin, after applying our approximation we would obtain the simplified evolution equation
\begin{equation}\label{simple}
    \p_{t}\bar{\rho}_{pert}(x,t)+u(\bar{\rho}_{pert})\cdot\nabla \bar{\rho}_{pert}(x,t)+u_{2}(\bar{\rho}_{pert})+u_{lin}(\rho(x,t))\cdot\nabla\rho_{pert}(x,t)=0.
\end{equation}

where $u_{lin}$ here is just notation for the linear approximation of the velocity around the origin.
In \eqref{simple} the term $u_{2}(\bar{\rho}_{pert})$ is giving us a very clear mechanism for growth, since this is exactly the term that gives exponential growth for the IPM equation with an unstable background. We can use this simplified evolution equation to construct a perturbation that grows very rapidly. If we could, after some period of time, prove that
$$\rho_{pert}(x,t)\approx Mx_{2}$$
for some very big $M$, we could then add an even more concentrated perturbation around the origin, which could grow even faster using the background term, now of size $M u_{2}(\rho_{new pert})$. We could then iterate this process an infinite number of times to obtain a solution of the form
$$\sum_{l=0}^{\infty}\rho_{l}(x,t)$$
where $\rho_{p+1}(x,t)$ grows due the instability generated by $\rho_{p}(x,t)$. We will call these $\rho_{l}$ layers, and solutions formed by $p$ of this layers $p$-layered solutions.
This approach presents several difficulties:
\begin{enumerate}
    \item Can we actually solve \eqref{simple}? Even with the simplifications included, it is not clear at all how much explicit information we can obtain about the solution of the equation.
    \item The solution obtained from \eqref{simple} is not in general a solution to \eqref{eq:ipm} with $F=0$. We can always consider a specific source term that so that the solution to \eqref{simple} solves \eqref{eq:ipm}, but then in general this source term will have bad regularity properties, so the question becomes how smooth can the source term be.
\end{enumerate}
These two points are actually related: In order to obtain any meaningful bounds for the source term required, we would need to have good control over our approximation for the solution. The way to deal with both of them is to actually use more complicated approximations for our solutions that have explicit (or almost explicit) solutions, but that approximate the behavior with a very high degree of accuracy. Let us discuss now some of the ingredients that we use to construct these accurate approximations.

\subsubsection{Approximating the velocity}\label{subsecvel}
The main problem when trying to get explicit solutions for simplified evolution equations similar to \eqref{simple} is the appearance of the term $u(\rho_{pert})$. Even when we have perfect information about $\rho_{pert}$, in general we will not be able to compute $u(\rho_{pert})$ explicitly, and solving PDEs where the time derivative involves $u(\rho_{pert})$ is often hopeless. To circumvent this issue, in Section \ref{secvelocity},  we will obtain local approximations of the velocity $u^{k}(\rho)$, where we can imagine $u^{k}(\rho)$ to be the $k$-th order local approximation of the velocity $u(\rho)$. These operators will give meaningful approximations only under certain circumstances, namely that we are considering $\rho$ of the form
$$\rho(x)=f(x)\sin(N(bx_{1}+ax_{2})+\theta_{0})$$
with $N$ very big, and to be more precise we want $N^{k}$ to be very big compared to $\|f\|_{C^{k}}$. The approximations for the velocity we find will only involve derivatives and multiplication by constants, so the equations we obtain from approximating the real velocity by them will be much easier to study. Furthermore, the velocities obtained this way keep the same structure of the original $\rho$, and in particular

$$u^{k}_{i}(f_{1}(x)\sin(N(bx_{1}+ax_{2}))+f_{2}(x)\cos(N(bx_{1}+ax_{2})))=g_{1}(x)\sin(N(bx_{1}+ax_{2}))+g_{2}(x)\cos(N(bx_{1}+ax_{2})).$$

Another very important property from this approximation is that, again for functions with the right structure
$$\|u(\rho)\cdot\nabla\rho\|_{L^{\infty}}\ll\|\rho\|_{C^1}\|u^{k}(\rho)\|_{L^{\infty}}$$
due to a cancellation in this quadratic term, and similar bounds hold for higher order norms. This will be important since it will mean that, if when we decompose our solution that blows-up in different layers as $\sum_{l=0}^{\infty}\rho_{l}$, the self-interaction terms
$$u(\rho_{l})\cdot\nabla \rho_{l}$$
will be much smaller than what the rough a priori bounds would suggest, and this will allow us to consider these kinds of errors as small perturbative terms.

Note also, since we want the source term to be $C^{\infty}$, in some sense we will need an approximation of the velocity of ¨infinite order¨. More precisely,  we will use higher order approximations as we consider layers of the solution with higher frequency, and this approximation order will tend to infinity as this frequency goes to infinity.

\subsubsection{Dealing with passive transport}\label{subsubsectrunc}

We mentioned earlier in Subsection \ref{subsecscales} how we will use different scales to approximate the behavior of our solutions. Namely, we will consider
$\rho(x,t)=\sum_{l=0}^{\infty}\rho_{l}$
with $\rho_{p+1}$ much more concentrated around the origin than $\rho_{i}$ for $i=0,1,...,p$, and this allows us to approximate 
$$u(\sum_{l=0}^{p}\rho_{l})\cdot\nabla\rho_{p+1}\approx u_{lin}\cdot\nabla \rho_{p+1}$$
with $u_{lin}$ the linear approximation around the origin of $u(\sum_{l=0}^{p}\rho_{l})$. This kind of approximation would be enough for us if we just wanted to prove blow-up with some source term, but this approximation is too inaccurate to give us an error than we can cancel with a smooth source term. Therefore, we would rather approximate this part of the evolution as

$$u(\sum_{l=0}^{p}\rho_{l})\cdot\nabla\rho_{p+1}\approx \P_{m}(u(\sum_{l=0}^{p}\rho_{l}))\cdot\nabla \rho_{p+1}$$
where $\P_{m}(u(\sum_{l=0}^{p}\rho_{l}))$ is the Taylor expansion around the origin of order $m$ of $u(\sum_{l=0}^{p}\rho_{l})$. Taking $m$ very big then allows us to make the error of this approximation as small as we want.

Unfortunately, this introduces a new problem in our construction since we need $\rho_{p+1}$ to be of the form
$$\rho_{p+1}=\sum_{j=0}^{J}f_{1,j}(x)\sin(Nj(bx_{1}+ax_{2}))+f_{2,j}(x)\cos(Nj(bx_{1}+ax_{2}))$$
in order to be able to approximate 
$$u(\rho_{p+1})\approx u^{k}(\rho_{p+1}),$$
but this kind of structure will not be preserved by the evolution of the equation
$$\p_{t}\rho_{p+1}+\P_{m}(u(\sum_{l=0}^{p}\rho_{l}))\cdot\nabla \rho_{p+1}=0.$$

To deal with this, instead we can consider, for example, $X(t)=\P_{n}(\tilde{X}(x,t))$ fulfilling
$$\p_{t}\tilde{X}(x,t)+\P_{m}(u(\sum_{l=0}^{p}\rho_{l}))\cdot\nabla \tilde{X}(x,t)=0,$$
$$\tilde{X}(x,0)=ax_{1}+bx_{2}$$
and then, to approximate the evolution of
$$\p_{t}f+\P_{m}(u(\sum_{l=0}^{p}\rho_{l}))\cdot\nabla f=0,$$
$$f(x,0)=f_{0}(ax_{1}+bx_{2})$$
we can just use $f_{0}(X(x,t))$. We will use this trick to model the effect of the transport term $u(\sum_{l=0}^{p}\rho_{l})\cdot\nabla\rho_{p+1}$ in a way that allows us to preserve the structure of $\rho_{p+1}$, and thus allowing us to still apply the operator $u^{k}$ to $\rho_{p+1}$.

\subsubsection{The first order solution operator}
The main difficulty in our construction will be, given a $p$-layered solution $\sum_{l=0}^{p}\rho_{l}$ solving \eqref{eq:ipm} with some desirable properties (we will later on define p-layered solution carefully) and source term $\sum_{l=0}^{p}F_{l}(x,t)$, to try and construct a $p+1$-layered solution. We would like to use the ideas mentioned in Subsections \ref{subsecvel} and \ref{subsubsectrunc} to construct $\rho_{p+1}$ with the desired properties, but it is not entirely obvious how to do it.

Without going the specifics, in Section \ref{secconst} we will use the approximations explained earlier to construct an operator  $S$ such that
\begin{equation}\label{equationS}
    \p_{t}S(f(x))\approx -u(\sum_{l=0}^{p}\rho_{l})\cdot\nabla S(f(x))-u(S(f(x)))\cdot\nabla[S(f(x))+\sum_{l=0}^{p}\rho_{l}].
\end{equation}

Note that if we had equality instead of $\approx$ in \eqref{equationS}, $\sum_{l=0}^{p}\rho_{l}+S(f(x))$ would solve \eqref{eq:ipm} with source term $\sum_{l=0}^{p}F_{l}(x,t)$. We could then choose appropriate $f(x)$ and define $\rho_{p+1}=S(f(x))$. This, however, will not be the case, so then we would need to obtain estimates for
$$\p_{t}S(f(x))+u(\sum_{l=0}^{p}\rho_{l})\cdot\nabla S(f(x))+u(S(f(x)))\cdot\nabla[S(f(x))+\sum_{l=0}^{p}\rho_{l}].$$
If this term was smooth and small enough, we could just call it $F_{p+1}$, add it to the rest of the source terms and keep adding layers until we have our solution that blows-up $\sum_{l=0}^{\infty}\rho_{l}$. Again, in general this will not happen, so we will define
$$\rho_{p+1,0}(x,t)=S(f(x))$$
and decompose
$$\p_{t}S(f(x))+u(\sum_{l=0}^{p}\rho_{l})\cdot\nabla S(f(x))+u(S(f(x)))\cdot\nabla[S(f(x))+\sum_{l=0}^{p}\rho_{l}]=F_{0,bad}+F_{0,good}$$
where $F_{0,good}$ is a smooth error term that is small enough to be included in our source terms, while $F_{0,bad}$ will be a part of the error that is too big to be included in our source terms without breaking the condition $\sum_{l=0}^{\infty}F_{l}\in C^{\infty}$. 

However, this error term $F_{0,bad}$ is chosen to have enough structure so that $S(F_{0,bad})$ is well defined. We can then define
$$\rho_{p+1,1}(x,t)=S(F_{0,bad})$$
and again decompose
$$\p_{t}S(F_{0,bad})+u(\sum_{l=0}^{p}\rho_{l}+S(f(x)))\cdot\nabla S(F_{0,bad})+u(S(F_{0,bad}(x)))\cdot\nabla[S(F_{0,bad})+S(f(x))+\sum_{l=0}^{p}\rho_{l}]=F_{1,bad}+F_{1,good}$$
and we will keep iterating this process until $F_{j,bad}$ is actually small enough to be included in our source terms, finally defining
$$\rho_{p+1}=\sum_{j=0}^{J}\rho_{p+1,j},\quad F_{p+1}=\sum_{j=0}^{J}F_{j,good}+F_{J,bad}$$
with some $J$ big.

\subsection{Structure of the paper}
In Section \ref{secvelocity} we study the velocity operator $u$ for the IPM equations, obtaining local approximations $u^{k}$ that will be necessary for the final construction. In Section \ref{secconst} we obtain the necessary results to describe and prove the properties of our iterative construction: First in subsection \ref{subsectrunc} we study passive transport when considering polynomial velocities, as well as some modified transport equations where we also assume that the scalar stays a polynomial, and we obtain useful properties for the regularity and behavior of said equations. In Subsection \ref{subsecplayer} we define the $p$-layered solutions that we use for our construction, and obtain useful properties about them. We also include in this section several auxiliary lemmas, and we define the first order solution operator that will be crucial for our iterative construction. In Subsection \ref{subsecfinconst} we prove that given a $p$-layered solution, if the parameter $N_{p}$ is big enough we can actually construct a $p+1$-layered solution. Finally, in Section \ref{secblowup} we prove our main theorem.
\subsection{Notation}
Through our proofs, $C$ will be used to refer to a constant, that might change from line to line.
	If these constants depend on parameters (say $a,b,c$), we will use the notation $C_{a,b,c}$ to make explicit this dependence. We will sometimes omit sub-indexes in the proof to simplify the notation, but we will always specify the dependence on the constants explicitly in the proofs.

\section{Approximations for the velocity}\label{secvelocity}

The main purpose of this section is to construct  a useful approximation of the velocity field. 

The velocity $\mathbf{u}$ can be expressed in terms of convolution kernels in the following equivalent forms, up to an integration by parts (see, for instance, \cite{CGO}):
\begin{align}\label{eq:vel-kernel}
    \mathbf{u}(x, t) &= \text{PV} (H \star \rho)(x, t) = \text{PV} \left( (H_1 \star \rho)(x, t), (H_2 \star \rho)(x, t) \right) - \frac{1}{2}(0, \rho(x, t)),
\end{align}
where
\begin{align}\label{eq:kernels}
    (H_1(x), H_2(x)) &= \frac{1}{2\pi} \left( -2 \frac{x_1 x_2}{|x|^4}, \frac{x_1^2 - x_2^2}{|x|^4} \right).
\end{align}

Alternatively, the velocity field can be written as:
\begin{align}\label{eq:vel-kernel2}
    \mathbf{u}(x, t) &= (K \star \nabla^\perp \rho)(x, t) = \left( -(K \star \partial_{x_2} \rho)(x, t), (K \star \partial_{x_1} \rho)(x, t) \right),
\end{align}
with the kernel
\begin{align}
    K(x) = -\frac{1}{2\pi} \frac{x_1}{|x|^2}.
\end{align}

To make our life easier, we will consider only solutions with certain symmetries that are preserved by the \ref{eq:IPMsystem} equation. Namely, we have the following definition:

\begin{definition}\label{odd}
	We say that a function $f(x,t)$ is odd if $f(x_{1},x_{2},t)=-f(-x_{1},-x_{2},t)$. It is straightforward to check that this property is preserved by \eqref{eq:IPMsystem}: If  both $F(x,t)$ and $\rho(x,t=0)$ are odd, then the solution $\rho(x,t)$ will also be odd.
\end{definition}

\begin{remark}
	If $\rho\in C^{\infty}$ is odd, then we have that, for any $j=0,1,2,...$ and $l=0,1,...,2j$
	$$(\frac{\p^{2j}}{\p^{l}_{x_{1}}\p^{2j-l}x_{2}}\rho)(x=0)=0.$$
	Furthermore, the operators $u_{i}(\rho)$ both produce odd velocities if $\rho$ is odd, so the same property is true for $u_{i}(\rho)$.
\end{remark}

We are interested in computing
\begin{align}
    u_i(f( x) \sin (N(ax_{1}+bx_{2})).
\end{align}
We know from \eqref{eq:ipm} that 
\begin{align*}
     u(\rho (x))&= 
     C_{IPM} \int_{\R^2} \frac{h_1}{|h|^2} \nabla^{\perp}\rho (x+h) \, dh_1 \, dh_2.
\end{align*}
with $C_{IPM}=\frac{1}{2\pi}$. With this in mind we will now study the behavior of the operator
$$V(f(x))=\int_{\R^2} \frac{h_1}{|h|^2} f (x+h) \, dh_1 \, dh_2.$$

Since the value of the constant does not change the qualitative behavior of the system, we will consider $C_{IPM}=1$.
First, we want to distinguish between the "near" part of the operator (when $h$ is small) and the "far" part of the operator (when $h$ is not small).  Therefore, choosing $g(x) \in C^\infty$ such that
\begin{align}
    g(x)=\begin{cases}
         1 \quad |x| \le \frac 1 2,\\
         0 \quad |x| \ge 1,
         \end{cases}
\end{align}
and with $g(|x|)$ monotone decreasing,
we write
\begin{align}
    V(w(x))&= \int_{\R^2} \frac{h_1 }{|h|^2} w (x+h) \, dh_1 \, dh_2\notag \\
    &=\int_{\R^2} \frac{h_1 }{|h|^2} w (x+h) \, (1-g(N^\eps |h|)) \, dh_1 \, dh_2 + \int_{\R^2} \frac{h_1}{|h|^2} w(x+h) \, g(N^\eps |h|) \, dh_1 \, dh_2\\
    &=: V^{\text{far}}(w(x)) + V^{\text{near}}(w(x)).
\end{align}
We would like to prove that $V^{\text{far}}(w(x))$
is a small remainder term. This is the content of the following.

We consider from now on in this section $g$ to be fixed so that the constants do not depend on the specific choice of $g$.
\begin{lemma}\label{vfar}
Let $K >0$, $1>\eps>0$, $a,b,\Theta_{0}\in \mathds{R}$ with $a^2+b^2=1$, and $f(x) \in C_{c}^K(\R^2)$, $\text{supp}(f(x))\subset B_{1}(0)$. There exists a constant $C_{K,\eps}$ such that, if $w(x)=f(x)\sin (N[ax_2+bx_{1}]+\Theta_{0})$, the following remainder estimate holds true:
    \begin{align}\label{remainder1}
        |V^{\text{far}}(w(x))|&=\left| \int_{\R^2} \frac{h_1}{|h|^2} f(x+h) \sin (N[a(x_2+h_2)+b(x_{1}+h_{1})]+\Theta_{0}) \, (1-g(N^\eps |h|)) \, dh_1 \, dh_2 \right|\\ \nonumber
        &\le C_{K,\eps}\sum_{i=0}^{K} N^{i\eps-K}\|f\|_{C^{K-i}}.
    \end{align}
    Furthermore, for any $J\in\mathds{N}$ we have that
    $$\|V^{\text{far}}(w(x))\|_{C^J}\leq C_{K,\eps,J}\sum_{i=0}^{K} N^{i\eps+J-K}\|f\|_{C^{K-i+J}}.$$
\end{lemma}
\begin{proof}
    To simplify the notation, we will show \eqref{remainder1} for $\Theta_{0}=0$, but the general case is proved in exactly the same way. Similarly, we note that either $|a|\geq \frac{1}{2}$ or $|b|\geq \frac{1}{2}$, and we will focus on the case $a\geq \frac12$, the other cases being analogous. We can use integration by parts with respect to $h_{2}$ $K$ times to obtain
    \begin{align*}
        |V^{\text{far}}(w(x))|&=|\int_{\R^2} \frac{h_1}{|h|^2} f(x+h) \sin (N[a(x_2+h_2)+b(x_{1}+h_{1})]) \, (1-g(N^\eps |h|)) \, dh_1 \, dh_2|\\
        &= |\frac{2^{K}}{N^K} \int_{\R^2} \frac{\p^{K}}{\p h^{K}_2} \left(\frac{h_1}{|h|^2}f(x+h)(1-g(N^\eps |h|))\right)  \sin (N[a(x_2+h_2)+b(x_{1}+h_{1})])+\frac{K\pi}{2})  \, dh_1 \, dh_2|\\
        &\leq \frac{C_{K}}{N^{K}}\sum_{i=0}^{K}\|f\|_{C^{K-i}}\int_{N^{-\epsilon}\geq |h|\geq \frac12N^{-\epsilon}}\sum_{j=0}^{i-1} \frac{1}{|h|^{1+j}}N^{(i-j)\epsilon} \, dh_1 \, dh_2\\
        &+\frac{C_{K}}{N^{K}}\sum_{i=0}^{K}\|f\|_{C^{K-i}}\int_{|h|\geq \frac12N^{-\epsilon}\cap \text{supp}(f(x+\cdot))} \frac{1}{|h|^{1+i}} \, dh_1 \, dh_2\\\\
        &\leq \frac{C_{K}}{N^{K}}\sum_{i=0}^{K}\|f\|_{C^{K-i}}N^{i\epsilon},
    \end{align*}
where we used the bounds for the support of $(1-g(N^{\epsilon}|h|))$ and of $\p^{i-j}_{h_{2}}(1-g(N^{\epsilon}|h|))$ with $i>j$.

To show the bounds in $C^{J}$ we just use the fact that $\p_{x_{i}}V^{far}(w)=V^{far}(\p_{x_{i}}w)$ and we apply the $L^{\infty}$ bound to each of the terms obtained from taking derivatives of $\rho$.
\end{proof}
Let us now deal with the localized part of the operator $V^\text{near}$. Exploiting the localization of $V^\text{near}$, which is forced by the cut-off function $g(N^\eps |h|)$, we simply appeal to a Taylor decomposition. More precisely, considering
$w(x)=f(x)\sin (N[ax_2+bx_{1}]+\Theta_{0})$
 we write
\begin{align}\label{TKRKN}
    &V^\text{near}(w(x))\notag\\
    &=\int_{\R^2} \frac{h_1}{|h|^2} f(x+h) \sin (N[a(x_2+h_2)+b(x_{1}+h_{1})]+\Theta_{0}) g(N^\eps |h|) \, dh_1 \, dh_2 \notag\\
    &= \int_{\R^2} \frac{h_1}{|h|^2} \left[\sum_{i=0}^K \sum_{j=0}^i \p_{x_1}^{i-j}\p_{x_2}^{j}(f(x)) \frac{h_1^{i-j} h_2^{j}}{(i-j)! j!}\right] \sin (N[a(x_2+h_2)+b(x_{1}+h_{1})]+\Theta_{0}) g(N^\eps |h|) \, dh_1 \, dh_2 + \mathcal{R}_{K,N}\notag\\
    &=:\mathcal{T}_K + \mathcal{R}_{K,N},
\end{align}
where 
$$\mathcal{R}_{K,N}:=\int_{\R^2} \frac{h_1}{|h|^2} \left[f(x+h)-\sum_{i=0}^K \sum_{j=0}^i \p_{x_1}^{i-j}\p_{x_2}^{j}f(x) \frac{h_1^{i-j} h_2^{j}}{(i-j)! j!}\right] \sin (N[a(x_2+h_2)+b(x_{1}+h_{1})]+\Theta_{0}) g(N^\eps |h|) \, dh_1 \, dh_2$$

We want to prove the following.
\begin{lemma}\label{RKN}
    Given $K\in\mathds{N}$, $1>\eps>0$, let $f( x) \in C^{K+1}(\R^2)$. Then there exists a constant $C_{K,\eps}>0$ such that 
    \begin{align*}
        |\mathcal{R}_{K,N}| \le C_{K,\eps}N^{-(K+2)\eps} \|f\|_{C^{K+1}}.
    \end{align*}
Furthermore, for $J\in\mathds{N}$, if $f( x) \in C^{K+J+1}(\R^2)$ we have that
\begin{align*}
        \|\mathcal{R}_{K,N}\|_{C^{J}} \le C_{K,J,\eps}N^{J-(K+2)\eps} \|f\|_{C^{K+J+1}}.
    \end{align*}
\end{lemma}
\begin{proof}
It is enough to note that
$$|\left[f(x+h)-\sum_{i=0}^K \sum_{j=0}^i \p_{x_1}^{i-j}\p_{x_2}^{j}f(x) \frac{h_1^{i-j} h_2^{j}}{(i-j)! j!}\right]|\leq C_{K}\|f\|_{C^{K+1}} |h|^{K+1}$$
so
$$|\mathcal{R}_{K,N}|\leq C\|f\|_{C^{K+1}}\int_{\R^2} \frac{|h|^{K+1}}{|h|}  g(N^\eps |h|) \, dh_1 \, dh_2\leq C_{K,\eps}\|f\|_{C^{K+1}} N^{-(K+2)\eps}.$$

For the inequality with $J>0$, we note that we can use differentiation under the integral sign to show that, if $D_{x}^{J}$ is a generic derivative of order $J$ with respect to the $x$ variable, and using a slight abuse of notation,

\begin{align*}
&|D^{J}\mathcal{R}_{K,N}|\\
    &\leq \int_{\R^2} \frac{h_1}{|h|^2}\sum_{k=0}^{J}D_{x}^{k}( \left[f(x+h)-\sum_{i=0}^K \sum_{j=0}^i \p_{x_1}^{i-j}\p_{x_2}^{j}f(x) \frac{h_1^{i-j} h_2^{j}}{(i-j)! j!}\right])\\
    &\times D_{x}^{J-k}(\sin (N[a(x_2+h_2)+b(x_{1}+h_{1})]+\Theta_{0})) g(N^\eps |h|) \, dh_1 \, dh_2\\
    & \leq C_{K,J}\int_{\R^2} \frac{|h|^{K}}{|h|}\sum_{k=0}^{J}\|f\|_{C^{K+k+1}} N^{J-k} g(N^\eps |h|) \, dh_1 \, dh_2\leq C_{K,J,\eps}\|f\|_{C^{K+J+1}} N^{J-(K+2)\eps}.
\end{align*}

\end{proof}

We can now use the change of variables $Nh_{1}=\tilde{h}_{1},Nh_{2}=\tilde{h}_{2}$  to get, after relabeling
\begin{align*}
    &V^{near}(w(x))-\mathcal{R}_{K,N}=\sum_{i=0}^K\sum_{j=0}^i N^{-(i+1)}\p_{x_1}^{i-j}\p_{x_2}^{j}f(x)\int_{\R^2} \frac{h_1}{|h|^2}    \frac{h_1^{i-j} h_2^{j}}{(i-j)! j!}\\
    &\times \sin ([a(Nx_2+h_2)+b(Nx_{1}+h_{1})]+\Theta_{0}) g(N^{\eps-1} |h|) \, dh_1 \, dh_2\\
    &=\sum_{i=0}^K\sum_{j=0}^i N^{-(i+1)} \p_{x_1}^{i-j}\p_{x_2}^{j}f(x)(\mathcal{T}^{i,j,s}_{N}+\mathcal{T}^{i,j,c}_{N})
\end{align*}
with
$$\mathcal{T}^{i,j,s}_{N}=\sin(aNx_{2}+bNx_1+\Theta_{0})\int_{\R^2} \frac{h_1}{|h|^2}    \frac{h_1^{i-j} h_2^{j}}{(i-j)! j!}\cos (ah_2+bh_{1}) g(N^{\eps-1} |h|) \, dh_1 \, dh_2,$$
$$\mathcal{T}^{i,j,c}_{N}=\cos(aNx_{2}+bNx_1+\Theta_{0})\int_{\R^2} \frac{h_1}{|h|^2}    \frac{h_1^{i-j} h_2^{j}}{(i-j)! j!}\sin (ah_2+bh_{1}) g(N^{\eps-1} |h|) \, dh_1 \, dh_2.$$
 We are now interested in studying 
 \begin{align}\label{tijs}
 	\mathbf{T^{i,j,s}_{N}}:=\int_{\R^2} \frac{h_1}{|h|^2}    \frac{h_1^{i-j} h_2^{j}}{(i-j)! j!}\cos (ah_2+bh_{1}) g(N^{\eps-1} |h|) \, dh_1 \, dh_2,
 \end{align}
\begin{align}\label{tijc}
	\mathbf{T^{i,j,c}_{N}}:=\int_{\R^2} \frac{h_1}{|h|^2}    \frac{h_1^{i-j} h_2^{j}}{(i-j)! j!}\sin (ah_2+bh_{1}) g(N^{\eps-1} |h|) \, dh_1 \, dh_2
\end{align}
and more specifically, their behavior when $N$ tends to infinity, which we do in the next lemma. It is important to keep in mind that $\mathbf{T}^{i,j,s}_{N},\mathbf{T}^{i,j,c}_{N}$ actually depend on $a$ and $b$.

\begin{lemma}\label{convVK}
    For any $K\geq 1$, $a,b\in \mathds{R}$ with $a^2+b^2=1$, for any $A>1$,  for $N_{2}\geq N_{1}\geq 1$ we have that there is a constant $C_{A,i,j}$ which does not depend on $a,b$ such that
    $$|\mathbf{T^{i,j,s}_{N_{1}}}-\mathbf{T^{i,j,s}_{N_{2}}}|\leq C_{A,i,j}N_{1}^{-A},$$
    $$|\mathbf{T^{i,j,c}_{N_{1}}}-\mathbf{T^{i,j,c}_{N_{2}}}|\leq C_{A,i,j}N_{1}^{-A},$$
    and in particular there exist constants $c^{s}_{i,j}$,$c^{c}_{i,j}$ such that 
    $$|c^{s}_{i,j}-\mathbf{T^{i,j,s}_{N_{1}}}|\leq 2C_{A,i,j}N_{1}^{-A},$$
    $$|c^{c}_{i,j}-\mathbf{T^{i,j,c}_{N_{1}}}|\leq 2C_{A,i,j}N_{1}^{-A}.$$
    Furthermore,
    \begin{equation}\label{cotaT}
        |c^{s}_{i,j}|,|c^{c}_{i,j}|\leq C_{i}
    \end{equation}
    and in particular this constant $C_{i}$ does not depend on $a,b$.
\end{lemma}

\begin{proof}
We will show the proof only for $\mathbf{T^{i,j,s}_{N}}$, since the proof for $\mathbf{T^{i,j,c}_{N}}$ is completely analogous. As in Lemma \ref{vfar}, we will only consider the case $a\geq \frac12$.
First, we note that we can use integration by parts $k_{2}$ times to show that, for any finite $k_{1},k_{2}\in\mathds{N}\cup \{0\}$, $k_{3}\in\mathds{N}$, $c_{0}\in\mathds{R}$
\begin{align}\label{NA}
    &|\int_{\R^2} \frac{\p^{k_{1}}}{\p h_{2}^{k_{1}}}(\frac{h_1}{|h|^2}    \frac{h_1^{i-j} h_2^{j}}{(i-j)! j!})\cos (ah_2+bh_{1}+c_{0}) \frac{\p^{k_{3}}}{\p h_{2}^{k_{3}}}g(N^{\eps-1} |h|) \, dh_1 \, dh_2|\\\nonumber
    &\leq |\int_{\R^2} \frac{\p^{k_{2}}}{\p h_{2}^{k_{2}}}(\frac{\p^{k_{1}}}{\p h_{2}^{k_{1}}}(\frac{h_1}{|h|^2}    \frac{h_1^{i-j} h_2^{j}}{(i-j)! j!})\frac{\p^{k_{3}}}{\p h_{2}^{k_{3}}}g(N^{\eps-1} |h|)) 2^{k_{2}} \, dh_1 \, dh_2| \leq C_{k_{1},k_{2},k_{3}}N^{(-1+\epsilon)(k_{1}+k_{2}+k_{3}-i-1)}
\end{align}
where we used 
$$\text{supp}(\p_{h_{2}} g(N^{-1+\epsilon}|h|))\subset\{h: \frac{1}{2} N^{1-\epsilon}\leq|h|\leq N^{1-\epsilon}\}$$
and
$$|\frac{\p^{k} }{\p h_{2}^{k}} \frac{h_1}{|h|^2}    \frac{h_1^{i-j} h_2^{j}}{(i-j)! j!}|\leq C_{k}|h|^{i-1-k}. $$
But using this, for any $A,I\in\mathds{N}$, $I\geq i+1$

\begin{align*}
    &|\mathbf{T^{i,j,s}_{N_{2}}}-\mathbf{T^{i,j,s}_{N_{1}}}|\leq|\int_{\R^2} \frac{h_1}{|h|^2}    \frac{h_1^{i-j} h_2^{j}}{(i-j)! j!}\cos (ah_2+bh_{1}) [g(N_{1}^{\eps-1} |h|)-g(N_{2}^{\eps-1} |h|)] \, dh_1 \, dh_2|\\
    &\leq C_{I}|\int_{\R^2} \frac{\p^{I}}{\p h_{2}^{I}}\bigg(\frac{h_1}{|h|^2}    \frac{h_1^{i-j} h_2^{j}}{(i-j)! j!}[g(N_{1}^{\eps-1} |h|)-g(N_{2}^{\eps-1} |h|)]\bigg) \frac{1}{a^{I}}\cos (ah_2+bh_{1}+\frac{I\pi}{2}) \, dh_1 \, dh_2|\\
    &\leq C_{I}|\int_{\R^2} \frac{\p^{I}}{\p h_{2}^{I}}(\frac{h_1}{|h|^2}    \frac{h_1^{i-j} h_2^{j}}{(i-j)! j!})[g(N_{1}^{\eps-1} |h|)-g(N_{2}^{\eps-1} |h|)]\frac{1}{a^{I}}\cos (ah_2+bh_{1}+\frac{I\pi}{2})  \, dh_1 \, dh_2|+C_{A,I}N^{-A}\\
    &\leq C_{I}|\int_{|h|\geq N^{1-\epsilon} } |h|^{i-I-1} \, dh_1 \, dh_2|+C_{A,I}N^{-A}\leq C_{I}N^{(1-\eps)(1+i-I)}+C_{A,I}N^{-A}
\end{align*}
and taking $I$ big gives us the desired convergence for $\mathbf{T^{i,j,s}_{N}}$. Finally, we just bound
$$|\mathbf{T^{i,j,s}_{1}}|\leq |\int_{\R^2} \frac{\p^{k_{1}}}{\p h_{2}^{k_{1}}}(\frac{h_1}{|h|^2}    \frac{h_1^{i-j} h_2^{j}}{(i-j)! j!})\cos (ah_2+bh_{1}) \frac{\p}{\p h_{2}}g( |h|) \, dh_1 \, dh_2|\leq C_{i}$$
which combined with the convergence for $\mathbf{T^{i,j,s}_{N}}$ gives \eqref{cotaT}.

\end{proof}

\begin{remark}\label{cijodd}
	By using \eqref{tijc} and \eqref{tijs}, one can check by parity that $c^{s}_{i,j}=0$ if $i$ is even and $c^{c}_{i,j}=0$ if $i$ is odd.
\end{remark}

With these lemmas out of the way, we are now ready to obtain the expression for $V$ that we will use later on.
\begin{definition}
    Given $f(x)\in C^{K+1}$, $a,b,\Theta_{0}\in \mathds{R}$ such that $a^2+b^2=1$, for $w(x)=f(x)\sin (N[ax_2+bx_{1}]+\Theta_{0})$, we define
    $$V^{K}(w)(x)=\sum_{i=0}^{K}\sum_{j=0}^{i}\frac{\p_{x_1}^{i-j}\p_{x_2}^{j}f(x)}{N^{i+1}}(c^{s}_{i,j}\sin(N(bx_{1}+ax_{2})+\Theta_{0})+c^{c}_{i,j}\cos(N(bx_{1}+ax_{2})+\Theta_{0})),$$
    with $c^{s}_{i,j},c^{c}_{i,j}$ as in lemma \ref{convVK}.  
\end{definition}
We consider, as our approximation for $V(w)$ the following expression

\begin{lemma}\label{c00}
    Given $K\in\mathds{N}$, $f(x)\in C^{K+1}$, $\text{supp}(f(x))\subset B_{1}(0)$ $a,b,\Theta_{0}\in \mathds{R}$ fulfilling $a^2+b^2=1$, $N\geq 1$, 
    $$w(x):= f(x)\sin(N(bx_{1}+ax_{2})+\Theta_{0})$$
    we have that for any $J\in\mathds{N}$, if $f(x)\in C^{K+J+2}$,
\begin{align}\label{v1K}
    &\|V(w)(x)-V^{K}(w)(x)\|_{C^{J}}\\
    &\leq C_{K,\eps,J}(\sum_{i=0}^{K} N^{i\eps+J-K}\|f\|_{C^{K-i+J}}+N^{J-(K+2)\eps} \|f\|_{C^{K+J+1}})\nonumber.
\end{align}

Furthermore, we have that
$$c^{s}_{0,0}=0, c^{c}_{0,0}=bC_{0}$$
with $C_{0}>0$.
\end{lemma}

\begin{proof}
    We start by showing \eqref{v1K}. 
    First, we note that, if use the notation as in lemma \ref{vfar} and \eqref{TKRKN}, we have
    $$\|V(w)-V^{K}(w)\|_{C^{J}}=\|V^{far}(w)+\mathcal{T}_K + \mathcal{R}_{K,N}-V^{K}(w)\|_{C^{J}}\leq \|V^{far}(w)\|_{C^{J}}+\|\mathcal{T}_K\|_{C^{J}} + \|\mathcal{R}_{K,N}-V^{K}(w)\|_{C^{J}}$$
    and using lemmas \ref{vfar} and \ref{RKN}, this gives us
    \begin{align}\label{v1v1K}
        &\|V(w)-V^{K}(w)\|_{C^{J}}=\|V^{far}(w)+\mathcal{T}_K + \mathcal{R}_{K,N}-V^{K}(w)\|_{C^{J}}\\ \nonumber
        &\leq \|\mathcal{R}_{K,N}-V^{K}(w)\|_{C^{J}}+C_{K,\eps,J}\sum_{i=0}^{K} N^{i\eps+J-K}\|f\|_{C^{K-i+J}}\\
        &+C_{K,J,\eps}N^{J-(K+2)\eps} \|f\|_{C^{K+J+1}}.
    \end{align}

    But
    \begin{align*}
        &\mathcal{R}_{K,N}-V^{K}(w)=\sum_{i=0}^K\sum_{j=0}^i N^{-i-1} \p_{x_1}^{i-j}\p_{x_2}^{j}f(x)\sin(bNx_{1}+aNx_2+\Theta_{0})[\mathbf{T^{i,j,s}_{N}}-c^{s}_{i,j}]\\
        &+\sum_{i=0}^K\sum_{j=0}^i N^{-i-1} \p_{x_1}^{i-j}\p_{x_2}^{j}f(x)\cos(bNx_{1}+aNx_2+\Theta_{0})[\mathbf{T^{i,j,c}_{N}}-c^{c}_{i,j}]
    \end{align*}
    so we can apply lemma \ref{convVK} to get, for any $A'\in\mathds{N}$
    $$\|\mathcal{R}_{K,N}-V^{K}(w)\|_{C^{J}}\leq C_{A,K}\|f\|_{C^{J+K}}N^{-A'}$$
    which combined with \eqref{v1v1K} gives us \eqref{v1K}.
    
    To show $c^{s}_{0,0}=0$, we note that, using that $\cos(x)=\cos(-x)$,
    $$c^{s}_{0,0}=\text{lim}_{N\rightarrow\infty}\int_{\R^2} \frac{h_1}{|h|^2} \cos (ah_2+bh_{1}) g(N^{\eps-1} |h|) \, dh_1 \, dh_2=0.$$
    Finally, to prove the dependence of $c^{c}_{0,0}$, and writing $c^{c}(0,0)(a,b)$ to specify the value of $a,b$, we have
    $$c^{c}_{0,0}(a,b)=\text{lim}_{N\rightarrow\infty}\int_{\R^2} \frac{h_1}{|h|^2} \sin (ah_2+bh_{1}) g(N^{\eps-1} |h|) \, dh_1 \, dh_2,$$
    but using the change of coordinates $\tilde{h}_{1}=bh_{1}+ah_{2},$ $\tilde{h}_{2}=-ah_{1}+bh_{2}$ we have
     \begin{align*}
    c^{c}_{0,0}(a,b)&=\text{lim}_{N\rightarrow\infty}\int_{\R^2}\int_{\R^2} \frac{b\tilde{h}_1-a\tilde{h}_{2}}{|h|^2} \sin (\tilde{h}_{1}) g(N^{\eps-1} |\tilde{h}|) \, d\tilde{h}_1 \, d\tilde{h}_2\\
        &=\text{lim}_{N\rightarrow\infty}\int_{\R^2}\int_{\R^2} \frac{b\tilde{h}_1}{|h|^2} \sin (\tilde{h}_{1}) g(N^{\eps-1} |\tilde{h}|) \, d\tilde{h}_1 \, d\tilde{h}_2=bc^{c}_{0,0}(0,1),
    \end{align*}
    and thus we only need to show that $c^{c}_{0,0}(a,b)>0$ for some $b>0$. So, considering for example $a=b=(\frac{1}{2})^{\frac12}$ and after re-scaling it is enough to show
    $$\text{lim}_{N\rightarrow \infty}\int_{\R^2} \frac{h_1}{|h|^2} \sin (h_1+h_{2}) g(N^{\eps-1} |h|) \, dh_1 \, dh_2> 0.$$
    But, using integration by parts with respect to $h_{2}$
    \begin{align*}
        &\text{lim}_{N\rightarrow \infty}\int_{\R^2} \frac{h_1}{|h|^2} \sin (h_1+h_{2}) g(N^{\eps-1} |h|) \, dh_1 \, dh_2\\
        &=\text{lim}_{N\rightarrow \infty}\Big(\int_{\R^2} \frac{2h_{1}h_{2}}{|h|^4} \sin (h_1+h_{2}) g(N^{\eps-1} |h|) \, dh_1 \, dh_2-\int_{\R^2} \frac{h_1}{|h|^2} \sin (h_1+h_{2}) \p_{h_{2}}(g(N^{\eps-1} |h|)) \, dh_1 \, dh_2\Big),
    \end{align*}
    and we can use integration by parts as in lemma \ref{convVK}, \eqref{NA} to show that
    $$\text{lim}_{N\rightarrow \infty}|\int_{\R^2} \frac{h_1}{|h|^2} \sin (h_1+h_{2}) \p_{h_{2}}(g(N^{\eps-1} |h|)) \, dh_1 \, dh_2|\leq \text{lim}_{N\rightarrow \infty} CN^{-1}=0.$$
    On the other hand
    \begin{align*}
        &\int_{\R^2}\frac{h_{1}h_{2}}{|h|^4} \cos (h_1+h_{2}) g(N^{\eps-1} |h|) \, dh_1 \, dh_2=4\int_{\R_{+}^2}\frac{h_{1}h_{2}}{|h|^4} \sin (h_1)\sin(h_{2}) g(N^{\eps-1} |h|) \, dh_1 \, dh_2\\
        &=2\int_{\R_{+}^2}\frac{h_{1}h_{2}}{|h|^4} (\cos (h_1+h_{2})-\cos(h_{1}-h_{2})) g(N^{\eps-1} |h|) \, dh_1 \, dh_2\\
        &=2\int_{0}^{\frac{\pi}{2}}\sin(A)\cos(A)\int_{0}^{\infty} \frac{\cos (R(\sin(A)+\cos(A)))-\cos(R(\sin(A)-\cos(A)))}{R} g(N^{\eps-1} R) \, dR \, dA.
    \end{align*}
Now we note that, for $K_{2}\geq K_{1}>0$ we get, by writing $C_{K_{1},K_{2}}:=\int_{0}^{\infty}\frac{\cos (RK_{1})-\cos(RK_{2})}{R} dR$,
\begin{align*}
    &\int_{0}^{\infty} \frac{\cos (RK_{1})-\cos(RK_{2})}{R} g(\lambda  R) \, dR=C_{K_{1},K_{2}}-\text{PV}\int_{0}^{\infty} \frac{(\cos (RK_{1})-\cos(RK_{2}))}{R} (1-g(\lambda R)) \, dR\\
    &=C_{K_{1},K_{2}}-\text{PV}\int_{0}^{\infty} \frac{\cos (R)}{R} [g(\lambda\frac{R}{K_{2}})-g(\lambda \frac{R}{K_{1}})]dR.
\end{align*}
Taking $\lambda\rightarrow \infty$ we get that
\begin{align*}   
0&=C_{K_{1},K_{2}}-\text{lim}_{\lambda\rightarrow\infty}\text{PV}\int_{0}^{\infty} \frac{\cos(R)}{R} [g(\lambda\frac{R}{K_{2}})-g(\lambda \frac{R}{K_{1}})]dR\\&=C_{K_{1},K_{2}}-\text{lim}_{\lambda\rightarrow\infty}\int_{0}^{\infty} \frac{\cos(\frac{R}{\lambda})}{R} [g(\frac{R}{K_{2}})-g( \frac{R}{K_{1}})]dR.
\end{align*}
Then, using the monotonicity of  $g$ we get

$$C_{K_{1},K_{2}}=\int_{0}^{\infty} \frac{1}{R} [g(\frac{R}{K_{2}})-g( \frac{R}{K_{1}})]dR>0.$$
Furthermore we have that, for $K_{2}\geq K_{1}$
\begin{align*}
    &\text{lim}_{\lambda\rightarrow 0}|\text{PV}\int_{0}^{\infty} \frac{\cos(R)}{R} [g(\lambda\frac{R}{K_{2}})-g(\lambda \frac{R}{K_{1}})]dR|\leq \text{lim}_{\lambda\rightarrow 0}\text{PV}\int_{0}^{\infty} |\p_{R}[\frac{g(\lambda\frac{R}{K_{2}})-g(\lambda \frac{R}{K_{1}})}{R}]|dR\\
    &\leq \text{lim}_{\lambda\rightarrow 0}\text{PV}\int_{\frac{K_{1}}{2\lambda}}^{\frac{K_{2}}{\lambda}}|\frac{1}{R^2}+\frac{\lambda}{ K_{1}R}|dR=0,
\end{align*}
so
$$\text{lim}_{\lambda\rightarrow 0}\text{PV}\int_{0}^{\infty}\int_{\R^2}\frac{h_{1}h_{2}}{|h|^4} \cos (h_1+h_{2}) g(N^{\eps-1} |h|) \, dh_1 \, dh_2=C_{K_{1},K_{2}}.$$
But since we have
$$|\int_{0}^{\infty} \frac{\cos(R)}{R} [g(\lambda\frac{R}{K_{2}})-g(\lambda \frac{R}{K_{1}})]dR|\leq \ln(K_{2})+\ln(2)-\ln(K_{1}),$$
we can use the dominated convergence theorem to show that
\begin{align*}
    &\text{lim}_{N\rightarrow\infty}2\int_{0}^{\frac{\pi}{2}}\sin(A)\cos(A)\int_{0}^{\infty} \frac{\cos (R(\sin(A)+\cos(A)))-\cos(R(\sin(A)-\cos(A)))}{R}g(N^{\eps-1} R) \, dR \, dA\\
    &=2\int_{0}^{\frac{\pi}{2}}\sin(A)\cos(A)C_{\sin(A)+\cos(A),\sin(A)-\cos(A)} \, dA>0
\end{align*}
which finishes the proof.

\end{proof}


We can now use Lemma \ref{c00} to give a useful approximation of our velocity when densities of the form $f(x)\sin(N(bx_{1}+ax_{2}))$.

\begin{corollary}\label{defuk}
	Given $K\in\mathds{N}\cup 0$, $f(x)\in C^{K+1}$, $a,b,\Theta_{0}\in \mathds{R}$ fulfilling $a^2+b^2=1$, $N\geq 1$, 
	$$\rho(x):= f(x)\sin(N(bx_{1}+ax_{2})+\Theta_{0})$$
	if we define 
	$$u^{K}_{1}(\rho(x)):=-V^{K}(f(x)\p_{x_{2}}(\sin(N(bx_{1}+ax_{2})+\Theta_{0})))-V^{K-1}(\p_{x_{2}}(f(x))\sin(N(bx_{1}+ax_{2})+\Theta_{0}))$$
	$$u^{K}_{2}(\rho(x)):=V^{K}(f(x)\p_{x_{1}}(\sin(N(bx_{1}+ax_{2})+\Theta_{0})))+V^{K-1}(\p_{x_{1}}(f(x))\sin(N(bx_{1}+ax_{2})+\Theta_{0}))$$
(with $V^{-1}(\rho):=0$) we have that for any $J\in\mathds{N}$, $j=1,2$ if $f(x)\in C^{K+J+2}$,
	\begin{align}\label{u1K}
		&\|u_{j}(\rho)(x)-u^{K}_{j}(\rho)(x)\|_{C^{J}}\\
  &\leq C_{K,\eps,J}(\sum_{i=0}^{K} N^{i\eps+J+1-K}\|f\|_{C^{K-i+J+1}}+N^{J+1-(K+2)\eps} \|f\|_{C^{K+J+2}}).\nonumber
	\end{align}

	Furthermore, we have that
	$$u^{0}_{1}(\rho)(x)=abC_{0}\rho(x),\quad u^{0}_{2}(\rho)(x)=-b^2C_{0}\rho(x)$$
	with $C_{0}>0$ the constant from Lemma \ref{c00}.
\end{corollary}

\begin{proof}
	The only property that is not completely immediate is \eqref{u1K}. For this, we just note that (taking $j=1$ and $\Theta_{0}=0$ for simplicity)
	\begin{align*}
		&\|u_{1}(\rho)(x)-u^{K}_{1}(\rho)(x)\|_{C^{J}}\\
		=&\|V(\partial_{x_{2}}\rho)(x)-V^{K}(f(x)\p_{x_{2}}(\sin(N(bx_{1}+ax_{2}))))-V^{K-1}(\p_{x_{2}}(f(x))\sin(N(bx_{1}+ax_{2})))\|_{C^{J}}\\
		\leq &\|V(f(x)\p_{x_{2}}(\sin(N(bx_{1}+ax_{2}))))-V^{K}(f(x)\p_{x_{2}}(\sin(N(bx_{1}+ax_{2}))))\|_{C^{J}}\\
		&+\|V(\p_{x_{2}}(f(x))\sin(N(bx_{1}+ax_{2})))-V^{K-1}(\p_{x_{2}}(f(x))\sin(N(bx_{1}+ax_{2})))\|_{C^{J}}\\
		\leq& C_{K,\eps,J}(\sum_{i=0}^{K} N^{i\eps+J+1-K}\|f\|_{C^{K-i+J}}+N^{i\eps+J-K}\|f\|_{C^{K-i+J+1}})\\
        + &C_{K,J,\eps}N^{J+1-(K+2)\eps} \|f\|_{C^{K+J+1}}+C_{K,J,\eps}N^{J-(K+2)\eps} \|f\|_{C^{K+J+2}}\\
		\leq&  C_{K,\eps,J}(\sum_{i=0}^{K} N^{i\eps+J+1-K}\|f\|_{C^{K-i+J+1}}+N^{J+1-(K+2)\eps} \|f\|_{C^{K+J+2}}).
	\end{align*}
	\begin{remark}
		It is easy to check using Remark \ref{cijodd} and our definition of $u^{K}$  that if $\rho(x)=f(x)\sin(N(bx_{1}+ax_{2}))$ is odd, then so is  $u^{K}(\rho)$.
	\end{remark}

\end{proof}

\section{Construction of the solution}\label{secconst}



\subsection{The truncated flow}\label{subsectrunc}

When constructing our solutions, solving the full non-linear evolution equation will be too complicated. To circumvent this problem, we will apply certain approximations to our system and use our source term to cancel the differences between our approximated evolution and the real evolution for the equation.

The first (very simplistic) model of our evolution would be just passive transport: we have a solution composed by pieces of very different sizes, and in particular the small pieces barely affect the bigger ones, so the velocity generated by the bigger pieces can be considered given when studying the evolution of the smaller pieces.

This motivate studying an evolution of the form

$$\p_{t}\rho_{small}+u(\rho_{big})\cdot\nabla \rho_{small}=0.$$

Of course, this is neglecting some (very important) parts of the real evolution, but it will be useful as a starting point. However, this model, in general, produces solutions that are not adequate for our construction, since we very quickly lose information about the structure of the solution.

We will therefore make some changes that will allow us to maintain some extra structure of our solution, in particular, we will approximate our velocity by a Taylor series expansion around the origin, which should gives us a reasonable approximation since $\rho_{small}$ will be very concentrated around the origin when compared to $u(\rho_{big})$. We will also, in some sense that will be made precise later, consider Taylor series expansions for $\rho_{small}$, also using the fact that $\rho_{small}$ is very concentrated to justify the approximation.

In order to do all this, we first need some definitions.

\begin{definition}
	Given an odd velocity $u(x,t)=(u_{1}(x,t),u_{2}(x,t))$ in $C^{k}$, we define the operator $\P_{k}(u(x,t))$ as
	$$\P_{k}(u(x,t))=(\P_{k}(u_{1}(x,t)),\P_{k}(u_{2}(x,t)))$$
	$$\P_{k}(u_{l}(x,t))=\sum_{i=0}^{k}\sum_{j=0}^{i}\frac{x_{1}^{j}x_{2}^{i-j}}{j!(i-j)!}\frac{\partial^{i}u_{l}}{\partial x_{1}^{j}\partial x_{2}^{i-j}}(x=0,t).$$
Similarly, given an odd scalar function $g(x,t)\in C^{k}$, we define
$$\P_{k}(g(x,t))=\sum_{i=0}^{k}\sum_{j=0}^{i}\frac{x_{1}^{j}x_{2}^{i-j}}{j!(i-j)!}\frac{\partial^{i}g}{\partial x_{1}^{j}\partial x_{2}^{i-j}}(x=0,t).$$
We call $\P_{k}(u),\P_{k}(g)$ the truncation of order $k$ of $u$ and $g$ respectively.
\end{definition}

Let us imagine we have some function $g(x,t)$ fulfilling
$$g(x,t=t_{0})=\sum_{i=0}^{\infty}\sum_{j=0}^{\infty}\frac{\partial^{i+j} g}{\partial^{i} x_{1} \partial^{j} x_{2}}(x=0,t=t_{0}) \frac{x_{1}^{i}x_{2}^{j}}{i!j!}$$
and such that $g(x,t)$ is transported by some velocity $u(x)$. One way to approximate $g(x,t)$ is to approximate the velocity by its truncation to order $k$, therefore obtaining the equation

$$\p_{t}g(x,t)+\P_{k}(u)\cdot\nabla g(x,t)=0$$
$$g(x,t_{0})=g(x).$$
This equation is formally equivalent to
$$g(x,t)= \sum_{i=0}^{\infty}\sum_{j=0}^{\infty}g_{i,j}(t) \frac{x_{1}^{i}x_{2}^{j}}{i!j!},$$
where $g_{i,j}(t)$ is obtained by solving
\begin{equation}\label{evolcij}
	\p_{t}g_{i,j}=-\sum_{\tilde{i}=0}^{i}\sum_{\tilde{j}=0}^{j}\frac{\p^{\tilde{i}+\tilde{j}}\P_{k}(u_{1})}{\p x_{1}^{\tilde{i}}\p x_{2}^{\tilde{j}}}(x=0,t)\frac{i!j!}{\tilde{i}!\tilde{j}!}\frac{g_{i-\tilde{i}+1,j-\tilde{j}}}{(i-\tilde{i})!(j-\tilde{j})!}-\sum_{\tilde{i}=0}^{i}\sum_{\tilde{j}=0}^{j}\frac{\p^{\tilde{i}+\tilde{j}}\P_{k}(u_{2})}{\p x_{1}^{\tilde{i}}\p x_{2}^{\tilde{j}}}(x=0,t)\frac{i!j!}{\tilde{i}!\tilde{j}!}\frac{g_{i-\tilde{i},j-\tilde{j}+1}}{(i-\tilde{i})!(j-\tilde{j})!}
\end{equation}
$$g_{i,j}(t=t_{0})=(\frac{\p^{i+j}}{\p x_{1}^{i} \p x_{2}^{j}}g)(x=0,t=t_{0}).$$

Unfortunately, even if our $g(x,t)$ is globally analytic at $t=t_{0}$, it is not clear if $g$ will remain globally analytic for $t\neq t_{0}$. However, if we are only interested in the local behaviour of $g(x,t)$ around $x=0$, we can consider
$$\P_{l}(g(x,t))=\sum_{i=0}^{l}\sum_{j=0}^{l-i}\frac{g_{i,j}(t)}{i!j!}x_{1}^ix_{2}^{j}.$$
With this in mind, we have the following definition.

\begin{definition}\label{truncated}
	Given an odd velocity $u(x,t)\in C^{k}$ and an odd scalar $g_{t_{0}}(x)\in C^{l}$, we define the $(k,l)$ truncated system as
	$$\p_{t} (\P_{l} g(x,t))=-\P_{l}[\P_{k}(u)(x,t)\cdot\nabla \P_{l}g(x,t)]$$
	$$\P_{l}g(x,t=t_{0})=\P_{l}g_{t_{0}}(x).$$
	Note that we have
	$$\P_{l}g(x,t)= \sum_{i=0}^{l}\sum_{j=0}^{l-i}g_{i,j}(t) \frac{x_{1}^{i}x_{2}^{j}}{i!j!},$$
	with
	\begin{equation}
		\p_{t}g_{i,j}=-\sum_{\tilde{i}=0}^{i}\sum_{\tilde{j}=0}^{j}\frac{\p^{\tilde{i}+\tilde{j}}\P_{k}(u_{1})}{\p x_{1}^{\tilde{i}}\p x_{2}^{\tilde{j}}}(x=0,t)\frac{i!j!}{\tilde{i}!\tilde{j}!}\frac{g_{i-\tilde{i}+1,j-\tilde{j}}}{(i-\tilde{i})!(j-\tilde{j})!}-\sum_{\tilde{i}=0}^{i}\sum_{\tilde{j}=0}^{j}\frac{\p^{\tilde{i}+\tilde{j}}\P_{k}(u_{2})}{\p x_{1}^{\tilde{i}}\p x_{2}^{\tilde{j}}}(x=0,t)\frac{i!j!}{\tilde{i}!\tilde{j}!}\frac{g_{i-\tilde{i},j-\tilde{j}+1}}{(i-\tilde{i})!(j-\tilde{j})!},
	\end{equation}
	$$g_{i,j}(t=t_{0})=(\frac{\p^{i+j}}{\p x_{1}^{i} \p x_{2}^{j}}g)(x=0,t=t_{0}).$$
\end{definition}
 Of special interest for us is the evolution for the $(1,1)$ truncated system, i.e., we are interested in studying the evolution of linear functions when considering $\P_{1}(u)$ as our velocity.

For this, given a linear velocity, i.e.
\begin{equation}\label{vlin}
	\P_{1}u=\P_{1}(u_{1}(x,t),u_{2}(x,t)),\quad\P_{1}u_{1}(x,t)=a(t)x_{1}+b(t)x_{2},\quad\P_{1}u_{2}(x,t)=c(t)x_{1}+d(t)x_{2}
\end{equation}
a straightforward computation shows that, if $\theta_{t_{0}}(x)=f(t_{0})x_{1}+g(t_{0})x_{2}$ and we study the passive transport equation
\begin{equation}\label{passtrans}
	\p_{t} \theta+\P_{1}(u)\cdot\nabla \theta=0, \quad\theta(x,t=t_{0})=\theta_{t_{0}}(x)
\end{equation}
then $\theta(x,t)=f(t)x_{1}+g(t)x_{2}$ with $f(t),g(t)$ given by
$$\p_{t}f(t)+a(t)f(t)+c(t)g(t)=0,\quad\p_{t}g(t)+b(t)f(t)+d(t)g(t)=0$$
is a solution to \eqref{passtrans}.
\begin{definition}\label{escvec}
	Given a linear function $\theta(x,t)=f(t)x_{1}+g(t)x_{2}$, we will use $\vec{\theta}(x,t)$ to represent it in vector form as $\vec{\theta}=\begin{pmatrix}
		f(t)\\
		g(t)
	\end{pmatrix}$.
	Similarly, given $\vec{\theta}=\begin{pmatrix}
		f(t)\\
		g(t)
	\end{pmatrix}$ we define its scalar representation $\theta(x,t)$ as $\theta(x,t)=f(t)x_{1}+g(t)x_{2}$.

\end{definition}
If we consider the  Jacobian of a velocity as in \eqref{vlin}, we will define
$$D\P_{1} u(x,t):=\begin{pmatrix}
	\p_{x_{1}}\P_{1} u_{1} & \p_{x_{2}}\P_{1} u_{1}\\
	\p_{x_{1}}\P_{1} u_{2} & \p_{x_{2}}\P_{1} u_{2}
\end{pmatrix}=\begin{pmatrix}
	a(t) & b(t)\\
	c(t) & d(t)
\end{pmatrix}.$$
 Note that, since $(\p_{x_{j}}\P_{1} u_{i})(x,t)=(\p_{x_{j}} u_{i})(x=0,t)$, we do not need to specify the point where we evaluate the Jacobian. To obtain a more compact notation, we will use the notation 
$$Du(x,t):=D\P_{1}u(x,t).$$
This gives us the following simple lemma.
\begin{lemma}\label{equvec}
	Given a linear function $\theta(x,t)$, with vectorial representation $\vec{\theta}(x,t)$ as in Definition \ref{escvec}, and a velocity as in \eqref{vlin}, we have that $\theta(x,t)$ fulfills
	$$\p_{t}\theta(x,t)+\P_{1}u(x,t)\cdot\nabla\theta(x,t)=0$$
	if and only if
	$$\p_{t}\vec{\theta}(x,t)+(Du(x,t))^{t}\vec{\theta}(x,t)=0.$$
\end{lemma}
\begin{proof}
	The lemma follows trivially by direct computation.
\end{proof}
\begin{definition}\label{angle&mod}
	We say that a linear function $\vec{\theta}(t)$ has modulus $r(t)$ and angle $\alpha(t)$ if
	
	$$\vec{\theta}(x,t)=\begin{pmatrix}
		r(t)\cos(\alpha(t))\\
		r(t)\sin(\alpha(t))
	\end{pmatrix}.$$
	
	Note that, given some linear velocity $\P_{1}(u)$ with Jacobian $Du$, if
	$$\p_{t}\vec{\theta}(x,t)+(Du(x,t))^{t}\vec{\theta}(x,t)=0,$$
	we have

	$$\p_{t}r(t)+((Du(x,t))^{t}\vec{\theta}(x,t))\cdot\frac{\vec{\theta}(x,t)}{|\vec{\theta}(x,t)|}=0,$$
	\begin{equation*}
		\p_{t}\alpha(t)+((Du(x,t))^{t}\frac{\vec{\theta}(x,t)}{|\vec{\theta}(x,t)|})\cdot\frac{\vec{\theta}^{\perp}(x,t)}{|\vec{\theta}^{\perp}(x,t)|}=0.
	\end{equation*}
 Note that the evolution for $\alpha(t)$ does not depend on $r(t)$, and we can write it as
	\begin{equation}\label{angleeq}
		\p_{t}\alpha(t)+((Du(x,t))^{t}\begin{pmatrix}
		\cos(\alpha(t))\\
		\sin(\alpha(t))
	\end{pmatrix})\cdot\begin{pmatrix}
    -\sin(\alpha(t))\\
		\cos(\alpha(t))
	\end{pmatrix}=0.
	\end{equation}
	
	Furthermore, if $\alpha(t)$ fulfills equation \eqref{angleeq} for some $Du$, we say that $\alpha(t)$ moves with velocity $u(x,t)$.
	
\end{definition}
\begin{remark}\label{alphapoly}
    We can also define the angle $\alpha(t)$ of a polynomial, by considering only the lineal part of the polynomial. Given a polynomial $X(x,t)$ we say $\alpha(t)$ is its angle if it is the angle corresponding to the lineal part of $X(x,t)$.
\end{remark}

We start first by studying the behavior of the coefficients when the linear part of the velocity is equal to zero.
\begin{lemma}\label{vcuadcontrol}
	Let $u(x,t)$ be a $C^{k}$ odd velocity with 
	$$\P_{1}(u)=0$$
	and  consider $g(x,t=T)=\lambda_{1}x_{1}+\lambda_{2}x_{2}$ with $\lambda_{1},\lambda_{2}\in[0,1]$. If we consider
	$$\p_{t} (\P_{l} g(x,t))=-\P_{l}(\P_{k}(u)\cdot\nabla (\P_{l}g(x,t))$$
	with $\P_{k}$, $\P_{l}$ as in Definition \ref{truncated}, then, for $g_{i,j}(x,t)$ as in Definition \ref{truncated} we have that, for $i+j\geq 1$
	$$|g_{i,j}(t)|\leq C_{i+j}[1+|T-t|]^{i+j-1}\|u(x,t)\|^{i+j-1}_{C^{k}}.$$
	
\end{lemma}
\begin{proof}
	We start by noting that the properties for the derivative of $u$ give us
	$$g_{1,0}(t)=g_{1,0}(t=T),g_{0,1}(t)=g_{0,1}(t=T).$$
	Now, we can just argue by induction, if the bound is true for $i+j=1,2,...,n_{0}$ we have that, for $i=0,1,...,n_{0}+1$, $j=n_{0}+1-i$,
	\begin{align*}
		&|\p_{t}g_{i,j}|\\
		&\leq|\sum_{\tilde{i}=0}^{i}\sum_{\tilde{j}=0}^{j}\frac{\p^{\tilde{i}+\tilde{j}}\P_{k}(u_{1})}{\p x_{1}^{\tilde{i}}\p x_{2}^{\tilde{j}}}(x=0,t)\frac{i!j!}{\tilde{i}!\tilde{j}!}\frac{g_{i-\tilde{i}+1,j-\tilde{j}}}{(i-\tilde{i})!(j-\tilde{j})!}-\sum_{\tilde{i}=0}^{i}\sum_{\tilde{j}=0}^{j}\frac{\p^{\tilde{i}+\tilde{j}}\P_{k}(u_{2})}{\p x_{1}^{\tilde{i}}\p x_{2}^{\tilde{j}}}(x=0,t)\frac{i!j!}{\tilde{i}!\tilde{j}!}\frac{g_{i-\tilde{i},j-\tilde{j}+1}}{(i-\tilde{i})!(j-\tilde{j})!}|\\
		&\leq \sum_{\tilde{i}=0}^{i}\sum_{\tilde{j}=0,\tilde{j}+\tilde{i}\geq 2}^{j}C_{n_{0}}\|u\|_{C^{k}}[1+|T-t|]^{n_{0}-1}\|u(x,t)\|^{n_{0}-1}_{C^{k}}\leq C_{n_{0}+1}[1+|T-t|]^{n_{0}-1}\|u(x,t)\|^{n_{0}}_{C^{k}}
	\end{align*}
	and integrating in time gives the desired result.

\end{proof}

Unfortunately, in general we cannot hope that the linear part of our velocity is zero. With this in mind, we have the following more general lemma for the evolution of $\P_{l}(g(x,t))$.

\begin{lemma}\label{controlpolinomio}
	Let $u(x,t)$ be a $C^{k}$ odd velocity and $g(x,t=T)=ax_{1}+bx_{2}$ with $a,b\in[0,1]$. Let us consider
	$$\p_{t} (\P_{l} g(x,t))=-\P_{l}(\P_{k}(u)\cdot\nabla (\P_{l}g(x,t)))$$
	with $\P_{k}$, $\P_{l}$ as in Definition \ref{truncated}.
	If we define
	$$\p_{t}\phi(x,t)=\P_{1}(u)(\phi(x,t)),$$
	$$\phi(x,t=T)=x$$
	and we have, for $i,j=1,2$
	$$|\p_{x_{j}}\phi_{i}(x,t)|,|\p_{x_{j}}(\phi^{-1})_{i}(x,t)|\leq M,$$
	then we have that,
	$$|g_{i,j}(t)|\leq C_{i+j}M^{i+j+2}[1+|T-t|]^{i+j-1}\|u(x,t)\|^{i+j-1}_{C^{k}},$$
	with $g_{i,j}$ as in Definition \ref{truncated}.

\end{lemma}

\begin{proof}
We start by considering the function $f(x,t)=g(\phi(x,t),t).$ $\P_{l}f(x,t)$ fulfills the evolution equation
\begin{align*}
	&\frac{d}{dt}\P_{l}(f(x,t))=\p_{t}\P_{l}g(\phi(x,t),t)+\frac{\p \P_{l}g(\phi(x,t),t)}{\p \phi_{1}(x,t)}\p_{t}\phi_{1}(x,t)+\frac{\p \P_{l}g(\phi(x,t),t)}{\p \phi_{2}(x,t)}\p_{t}\phi_{2}(x,t))\\
	&=-\P_{l}((\P_{k}(u)\cdot \nabla \P_{l}g)(\phi(x,t),t)+\frac{\p \P_{l}g(\phi(x,t),t)}{\p \phi_{1}(x,t)}\p_{t}\phi_{1}(x,t)+\frac{\p \P_{l}g(\phi(x,t),t)}{\p \phi_{2}(x,t)}\p_{t}\phi_{2}(x,t))\\
	&=-\P_{l}([\P_{k}(u)-\P_{1}(u)]\cdot \nabla \P_{l} g)(\phi(x,t),t).
\end{align*}
To study this expression, we start by taking a look at $[\P_{k}(u)-\P_{1}(u)](\phi(x,t),t)$.
First, we have that
$$(\P_{k}(u)-\P_{1}(u))(x,t)=\sum_{j=2}^{k}\sum_{i=0}^{j}\frac{x_{1}^{i}x_{2}^{j-i}}{i!(j-i)!}\frac{\p^{j}u}{\p x_{1}^{i}\p x_{2}^{j-i}}(x,t)$$
and thus
$$(\P_{k}(u)-\P_{1}(u))(\phi(x,t),t)=\sum_{j=2}^{k}\sum_{i=0}^{j}\frac{\phi_{1}(x,t)^{i}\phi_{2}(x,t)^{j-i}}{i!(j-i)!}\frac{\p^{j}u}{\p x_{1}^{i}\p x_{2}^{j-i}}(\phi(x,t),t).$$
But we also have that, since $\phi_{1},\phi_{2}$ are linear
$$\phi_{1}(x,t)=x_{1}\p_{x_{1}}\phi_{1}(0,t)+x_{2}\p_{x_{2}}\phi_{1}(0,t),\quad\phi_{2}(x,t)=x_{1}\p_{x_{1}}\phi_{2}(0,t)+x_{2}\p_{x_{2}}\phi_{2}(0,t)$$
so
\begin{align*}
    &(\P_{k}(u)-\P_{1}(u))(\phi(x,t),t)\\
    &=\sum_{j=2}^{k}\sum_{i=0}^{j}\frac{(x_{1}\p_{x_{1}}\phi_{1}(0,t)+x_{2}\p_{x_{2}}\phi_{1}(0,t))^{i}(x_{1}\p_{x_{1}}\phi_{2}(0,t)+x_{2}\p_{x_{2}}\phi_{2}(0,t))^{j-i}}{i!(j-i)!}\frac{\p^{j}u}{\p x_{1}^{i}\p x_{2}^{j-i}}(\phi(x,t),t).
\end{align*}
On the other hand, we note that

$$\begin{pmatrix}
	\p_{x_{1}}(\P_{l}g(\phi(x,t),t))\\
	\p_{x_{2}}(\P_{l}g(\phi(x,t),t))
\end{pmatrix}=\begin{pmatrix}
	\p_{x_{1}}\phi_{1} & \p_{x_{1}}\phi_{2}\\
	\p_{x_{2}}\phi_{1} & \p_{x_{2}}\phi_{2}
\end{pmatrix}\begin{pmatrix}
 (\p_{x_{1}}\P_{l}g)(\phi(x,t),t)\\
 (\p_{x_{2}}\P_{l}g)(\phi(x,t),t)
\end{pmatrix}=(D\phi)^{t}\begin{pmatrix}
(\p_{x_{1}}\P_{l}g)(\phi(x,t),t)\\
(\p_{x_{2}}\P_{l}g)(\phi(x,t),t)
\end{pmatrix}$$
and therefore,
\begin{align*}
	\begin{pmatrix}
		(\p_{x_{1}}\P_{l}g)(\phi(x,t),t)\\
		(\p_{x_{2}}\P_{l}g)(\phi(x,t),t)
	\end{pmatrix}
	&=((D\phi)^{t})^{-1}
	\begin{pmatrix}
		\p_{x_{1}}(\P_{l}g(\phi(x,t),t))\\
		\p_{x_{2}}(\P_{l}g(\phi(x,t),t))
	\end{pmatrix}=
	(D\phi^{-1})^{t}
	\begin{pmatrix}
		\p_{x_{1}}(\P_{l}g(\phi(x,t),t))\\
		\p_{x_{2}}(\P_{l}g(\phi(x,t),t))
	\end{pmatrix}\\
	&=
	(D\phi^{-1})^{t}
	\begin{pmatrix}
		\p_{x_{1}}(\P_{l}f(x,t))\\
		\p_{x_{2}}(\P_{l}f(x,t))
	\end{pmatrix}
\end{align*}
which we can use in the evolution equation to get
\begin{align*}
	&\frac{d}{dt}\P_{l}(f(x,t))=-\P_{l}([\P_{k}(u)-\P_{1}(u)(\phi(x,t))]\cdot [(D\phi^{-1})^{t}\nabla \P_{l}f(x,t)])\\
	&=-\P_{l}((D\phi^{-1})[\P_{k}(u)-\P_{1}(u)(\phi(x,t))]\cdot \nabla \P_{l}f(x,t)).
\end{align*}
We can now use Lemma \ref{vcuadcontrol} with velocity $(D\phi^{-1})[\P_{k}(u)-\P_{1}(u)(\phi(x,t))]$ and the bounds for $\phi,\phi^{-1}$ to obtain
$$|f_{i,j}(t)|\leq C_{i+j}M^{i+j+1}[1+|T-t|]^{i+j-1}\|u(x,t)\|^{i+j-1}_{C^{k}},$$
where
$$f_{i,j}(t)=(\frac{\p^{i+j}}{\p x_{1}^{i} \p x_{2}^{j}}f)(x=0,t).$$
Finally, to obtain the bounds for $g(x,t)$ we just use that 
$g(x,t)=f(\phi^{-1}(x,t),t),\P_{l}g(x,t)=\P_{l}f(\phi^{-1}(x,t),t)$
combined with 
$$(\phi^{-1})_{1}(x,t)=x_{1}\p_{x_{1}}(\phi^{-1})_{1}(0,t)+x_{2}\p_{x_{2}}(\phi^{-1})_{1}(0,t),(\phi^{-1})_{2}(x,t)=x_{1}\p_{x_{1}}(\phi^{-1})_{2}(0,t)+x_{2}\p_{x_{2}}(\phi^{-1})_{2}(0,t)$$
and the bounds for the derivatives of $\phi^{-1}$ to get
$$|g_{i,j}(t)|\leq C_{i+j}M^{i+j+2}[1+|T-t|]^{i+j-1}\|u(x,t)\|^{i+j-1}_{C^{k}}.$$
\end{proof}

\subsection{p-layered solutions and the first order solution operator}\label{subsecplayer}

The approximation obtained for passive transport will be useful when constructing our solutions since, as mentioned before, when perturbing a known solution to IPM $\rho_{big}(x,t)$ with a much more concentrated perturbation $\rho_{small}(x,0)$, it gives, in some sense, the main order of the evolution. However, it is still completely ignoring some terms in our evolution: The self interaction of $\rho_{small}(x,t)$ with itself and the effect of the velocity generated by $\rho_{small}$ moving $\rho_{big}(x,t)$. 
Now, we need further structure of the functions $\rho_{big}$ and $\rho_{small}$ to give a more accurate description of the behavior. If, for example, we assume that $\rho_{small}=f(x)\sin(N(bx_{1}+ax_{2}))$, we can use the results from Section \ref{secvelocity} to approximate the evolution of $\rho_{small}$ as

\begin{equation}\label{attempt1c00}
	\p_{t}\rho_{small}+u(\rho_{big})\cdot\nabla\rho_{small}+u^{0}(\rho_{small})\cdot\nabla\rho_{big}=0,
\end{equation}
where $u^{0}$ is just defined using Corollary \ref{defuk} with $K=0$.

This model already would give us a  way more realistic approximation of the behaviour of $\rho_{small}$, but it has one very important problem: A function of the form $f(x)\sin(N(bx_{1}+ax_{2}))$ will stop having this form when it gets transported by some arbitrary flow. To circumvent this, we will instead modify  the transport part of our velocity, in particular we will instead of using normal transport use the results obtained in Subsection \ref{subsectrunc} to approximate the transport in a way that allows to apply the results from Section \ref{secvelocity}.

Furthermore, we still have not given a precise definition of the kind of solutions we want to study, or how exactly we will perturb a known solution with a very concentrated perturbation. In this section we will give the key definitions for the construction of our solutions, defining in particular $p$-layered solutions, which will be the building blocks for our solution that blows-up. We will also prove some important properties for $p$-layered solutions, and we will define precisely the approximated evolution equation that we will apply in our construction, reminiscent of \eqref{attempt1c00} but with some carefully done changes.

\begin{definition}\label{player}
	We say that $\sum_{i=0}^{p}\rho_{i}(x,t)$, with $\rho_{i}$ odd functions, is a p-layered solution if there exist odd functions $F_{i}(x,t)$, $i=0,1,...,p$ and parameters $(N_{i},\beta_{i},k_{i},K_{i})$, with $k_{i},K_{i}\in\mathds{N}, N_{i}\in\mathds{R}$ $k_{i},K_{i},N_{i}>2$, $\beta_{i}\in (0,\frac{1}{4})$ such that:
	\begin{enumerate}
		\item For $t\in [0,1]$ , $i=0,1,...,p$ $\rho_{i}(x,t),F_{i}(x,t)\in C^{\infty}$, $\text{supp}(\rho_{i}(x,t)),\text{supp}(F_{i}(x,t))\subset B_{1}(0)$. 
		\item For $i=1,2,...,p$, $N_{i}\geq e^{K_{-}\beta_{i-1}N_{i-1}^{\frac{\beta_{i}}{8}}}$ with $K_{-}>0$ a universal constant.
		\item For $t\in[0,1]$, $i=0,1,...,p$, we have
		$$\p_{t}(\sum_{j=0}^{i}\rho_{j}(x,t))+u(\sum_{j=0}^{i}\rho_{j}(x,t))\cdot\nabla(\sum_{j=0}^{i}\rho_{j}(x,t))=\sum_{j=0}^{i}F_{j}(x,t).$$
		\item For  $i=1,...,p$, $t\in[0,1-\tilde{C}N_{i-1}^{-\frac{3}{4}\beta_{i-1}}]$ with $\tilde{C}=\frac{16\pi^3}{C_{0}}$, $C_{0}$ as in Lemma \ref{c00}, we have $F_{i},\rho_{i}=0$.
		\item For $t\in[0,1]$, $i=0,1,...,p$, $j=0,1,...,k_{i}$
		$$\|\rho_{i}(x,t)\|_{C^{j}}\leq N_{i}^{j}.$$
		\item For $t\in[0,1]$, $i=0,1,...,p$, if $\beta_{i}\leq \frac{1}{40}$,
		$$\|F_{i}(x,t)\|_{C^{\lfloor\frac{1}{4\beta_{i}}\rfloor-10}}\leq N_{i}^{-\frac{1}{4}}.$$
		\item $K_{i}\geq \lceil\frac{4}{\beta_{i}^2}\rceil+1$ and for $t\in[0,1]$, $i=0,1,...,p$, $j=1,...,K_{i}$
		$$\|\sum_{l=0}^{i}\rho_{l}(x,t)\|_{C^{j}}\leq 2N_{i}^{j}.$$
		\item For $i=0,1,...,p-1$
  \begin{align}\label{cotasnb}
      &2\beta_{i+1}\geq \beta_{i},N_{i+1}\geq 4N_{i}, 2N_{i+1}^{-\frac{3\beta_{i+1}}{4}}\leq N_{i}^{-\frac{3\beta_{i}}{4}},\nonumber\\ 
       & N_{i}^{-1}\geq N_{i+1}^{-\frac{\beta_{i+1}}{8}},N_{i+1}^{\beta_{i+1}}\geq 10N_{i+1}^{\frac{\beta_{i+1}}{8}}.
  \end{align}
		
	\end{enumerate}  
	Furthermore, we have that, for $i=1,...,p$, $t\in[1-\tilde{C}N_{i}^{-\frac{3\beta_{i}}{4}},1]$, $\tilde{C}$ as in Lemma \ref{girot0}
	\begin{equation}\label{lowfrapriori}
		\|\sum_{l=0}^{i-1}\rho_{l}\|_{C^1}\leq N_{i}^{\frac{\beta_{i}}{8}},\|u(\sum_{l=0}^{i-1}\rho_{l})\|_{C^1}\leq N_{i}^{\frac{\beta_{i}}{8}},
	\end{equation}
	and for $i=0,1,...,p$ there are $A_{i}(t),B_{i}(t)\alpha_{i}(t)$, fulfilling
	$A_{i}(t),B_{i}(t)\in[\frac{N_{i}^{\beta_{i}}}{2},2N_{i}^{\beta_{i}}]$, such that 
	\begin{equation}\label{derivadasrhoi}
		|\frac{\p}{\p x_{1}}\rho_{i}-A_{i}(t)\cos(\alpha_{i}(t))|,|\frac{\p}{\p x_{2}}\rho_{i}-A_{i}(t)\sin(\alpha_{i}(t))|\leq N_{i}^{\frac{\beta_{i}}{8}}.
	\end{equation}
	\begin{align}\label{anguloup}
		&|\frac{\p}{\p x_{1}}u_{1}(\rho_{i})(x=0)-B_{i}(t)C_{0}\sin(\alpha_{i}(t))\cos(\alpha_{i}(t))^2|\leq N_{i}^{\frac{\beta_{i}}{8}},\\ 
		&|\frac{\p}{\p x_{2}}u_{1}(\rho_{i})(x=0)-B_{i}(t)C_{0}\sin(\alpha_{i}(t))^2\cos(\alpha_{i}(t))|\leq  N_{i}^{\frac{\beta_{i}}{8}},\\
        &|\frac{\p}{\p x_{1}}u_{2}(\rho_{i})(x=0)+B_{i}(t)C_{0}\cos(\alpha_{i}(t))^3|\leq  N_{i}^{\frac{\beta_{i}}{8}},
        \end{align}
        \begin{align}\label{anguloup2}
		&|\frac{\p}{\p x_{2}}u_{2}(\rho_{i})(x=0)+B_{i}(t)C_{0}\sin(\alpha_{i}(t))\cos(\alpha_{i}(t))^2|\leq  N_{i}^{\frac{\beta_{i}}{8}},
	\end{align}
	with $C_{0}$ as in Lemma \ref{c00}, and  $\alpha_{i}(t)$ fulfilling $\alpha_{i}(t=1)=N_{i}^{-\frac{\beta_{i}}{8}}+\frac{\pi}{2}$ where $\alpha_{i}(t)$ is the angle moving with velocity $u(\sum_{l=0}^{i-1}\rho_{l})$.

\end{definition}

\begin{remark}
    Although from the properties listed above it is not apparent what the exact meaning of each parameter is, each of them tells us a specific property of the $p$-layered solution. $\beta_{i}$ gives us the regularity of $\rho_{i}$, and in particular we expect $\|\rho_{i}\|_{C^{1-\beta_{i}}}$ to be roughly of order 1. $N_{i}$ gives us the frequency of $\rho_{i}$, i.e. $\rho_{i}$ should look roughly like $\frac{\sin(N_{i}x_{2})}{N_{i}^{1-\beta_{i}}}$ at the time of blow-up. Finally $k_{i}$ tells us how many derivatives of $\rho_{i}$ we can estimate and $K_{i}$ how many derivatives of $\sum_{l=0}^{i}\rho_{l}$ we can estimate.
\end{remark}
\begin{remark}
    Later on, in several of our lemmas, given a $p$-layered solution, the result will be valid given a choice of $\beta_{p+1}$. Whenever this happens, we always assume that this choice of $\beta_{p+1}$ is consistent with the conditions in Definition \ref{player}, and in particular $2\beta_{p+1}\geq \beta_{p}.$
\end{remark}

With this definition in mind, we are now ready to prove some properties for the angles $\alpha_{i}$ from our $p$-layered solutions.

\begin{lemma}\label{controlangulocorto}
	Let $\sum_{i=0}^{p}\rho_{i}(x,t)$ be a $p$-layered solution with parameters $(N_{i},\beta_{i},k_{i},K_{i})$ for $i=0,1,...,p$. Then, for $p\geq 1$, $\lambda>0$, $t\in[1-\lambda N_{p}^{-\frac{3\beta_{p}}{4}},1]$, we have that, if $N_{p}$ is big enough (depending only on $\beta_{p},\lambda$), then
	$$\alpha_{p}(t)\in[N_{p}^{-\frac{\beta_{p}}{8}}+\frac{\pi}{2}-N_{p}^{-\frac{\beta_{p}}{2}}, N_{p}^{-\frac{\beta_{p}}{8}}+\frac{\pi}{2}+N_{p}^{-\frac{\beta_{p}}{2}}].$$

\end{lemma}

\begin{proof}
	Using \eqref{angleeq}, and defining, for $i=1,2$, 
	\begin{equation}\label{vki}
		v_{p,i}:=u_{i}(\sum_{l=0}^{p-1}\rho_{l})
	\end{equation}
 and using \eqref{lowfrapriori} and \eqref{angleeq}  to bound $v_{p}$ we get

	\begin{align*}
		|\p_{t}\alpha_{p}(t)|=\\
		&|-(\cos(\alpha_{p}(t))\p_{1}v_{p,1}(x=0,t)+\sin(\alpha_{p}(t))\p_{1}v_{p,2}(x=0,t))(-\sin(\alpha_{p}(t)))\\
		&-(\cos(\alpha_{p}(t))\p_{2}v_{p,1}(x=0,t)+\sin(\alpha_{p}(t))\p_{2}v_{p,2}(x=0,t))\cos(\alpha_{p}(t))|\\
		&\leq 2\sum_{i=1,2}\|v_{p,i}\|_{C^1}\leq 2 N_{p}^{\frac{\beta_{p}}{8}}
	\end{align*}
and integrating in time and using that  $N_{p}^{\beta_{p}}$ is big we get, for $t\in[1-\lambda N_{p}^{\frac{-3\beta_{p}}{4}},1]$
\begin{align*}
	&|\alpha_{p}(1)-\alpha_{p}(t)|\leq 2\lambda N_{p}^{-\frac{3\beta_{p}}{4}}\|v_{p,i}\|_{C^1}\leq 4\lambda N_{p}^{\frac{\beta_{p}}{8}}N_{p}^{-\frac{3\beta_{p}}{4}}\leq N_{p}^{\frac{-\beta_{p}}{2}}.
\end{align*}

\end{proof}

This gives us information about the specific behaviour of $\alpha_{p}$ for a $p$-layered solution. We can use this to obtain some useful bounds for the flow generated by $\P_{1}(u(\sum_{l=0}^{p}\rho_{l}))$.

\begin{lemma}\label{p1phi}
	Let $\sum_{l=0}^{p}\rho_{l}(x,t)$ be a $p-$layered solution with parameters $(N_{i},\beta_{i},k_{i},K_{i})$ for $i=0,1,...,p$. Let $\lambda>0$ and  $\phi(x,\tilde{t},t)$ for some $\tilde{t}\in[1-\lambda N_{p}^{-\frac{3\beta_{p}}{4}},1]$ be defined as
	$$\p_{t}\phi(x,\tilde{t},t)=\P_{1}(u(\sum_{l=0}^{p}\rho_{l}))$$
	$$\phi(x,\tilde{t},\tilde{t})=x.$$
	Then, if $N_{p}$ is big enough (depending only on $\beta_{p}$), for $i,j=1,2$, $t\in [1-\lambda N_{p}^{-\frac{3\beta_{p}}{4}},1]$ we have
	$$|\p_{x_{j}}\phi_{i}(x,\tilde{t},t)|\leq C_{\lambda}N_{p}^{\frac{\beta_{p}}{4}},$$
	and, in particular, if $f(x,t)$ fulfills
	$$\p_{t}f(x,t)+\P_{1}(u(\sum_{l=0}^{p}\rho_{l}))\cdot\nabla f(x,t)=0$$
	and if $\text{supp}(f(x,\tilde{t}))\subset B_{a}(0)$ for $\tilde{t}\in [1-\lambda N_{p}^{-\frac{3\beta_{p}}{4}},1]$, then
	$$\text{supp}(f(x,t))\subset B_{2C_{\lambda } N_{p}^{\frac{\beta_{p}}{4}}a}(0)$$
	for $t\in[1-\lambda N_{p}^{-\frac{3\beta_{p}}{4}},1]$
	since the support is transported by the velocity $\P_{1}(u(\sum_{l=0}^{p}\rho_{l}))$ and $u(\sum_{l=0}^{p}\rho_{l})(x=0)=0$.
	
\end{lemma}

\begin{proof}
	We will consider $\lambda=1$ in the proof for simplicity, but note that the proof is exactly the same with other values of $\lambda$. According to \eqref{anguloup} and the equations afterwards we have that
	\begin{align*}
		&\bigg|\begin{pmatrix}
			&\p_{1}u_{1}(\rho_{p})(x=0,t)\ &\p_{2}u_{1}(\rho_{p})(x=0,t)\\
			&\p_{1}u_{2}(\rho_{p})(x=0,t)\ &\p_{2}u_{2}(\rho_{p})(x=0,t)
		\end{pmatrix}\\
		&-C_{0}B_{p}(t)\cos(\alpha_{p}(t))\begin{pmatrix}
			&\sin(\alpha_{p}(t))\cos(\alpha_{p}(t))\ &\sin(\alpha_{p}(t))^2\\
			&-\cos(\alpha_{p}(t))^2\ &-\sin(\alpha_{p}(t))\cos(\alpha_{p}(t))
		\end{pmatrix}\bigg|\\
		&\leq \begin{pmatrix}
			&N_{p}^{\frac{\beta_{p}}{8}}\ &N_{p}^{\frac{\beta_{p}}{8}}\\
			&N_{p}^{\frac{\beta_{p}}{8}}\ &N_{p}^{\frac{\beta_{p}}{8}}
		\end{pmatrix}.
	\end{align*}
	Furthermore, by Lemma \ref{controlangulocorto} we have that, for $N_{p}$ big enough
	\begin{align*}
		&\bigg|C_{0}B_{p}(t)\cos(\alpha_{p}(t=1))\begin{pmatrix}
			&\sin(\alpha_{p}(t=1))\cos(\alpha_{p}(t=1))\ &\sin(\alpha_{p}(t=1))^2\\
			&-\cos(\alpha_{p}(t=1))^2\ &-\sin(\alpha_{p}(t=1))\cos(\alpha_{p}(t=1))
		\end{pmatrix}\\
		&-C_{0}B_{p}(t)\cos(\alpha_{p}(t))\begin{pmatrix}
			&\sin(\alpha_{p}(t))\cos(\alpha_{p}(t))\ &\sin(\alpha_{p}(t))^2\\
			&-\cos(\alpha_{p}(t))^2\ &-\sin(\alpha_{p}(t))\cos(\alpha_{p}(t))
		\end{pmatrix}\bigg|\\
		&\leq C\begin{pmatrix}
			&N_{p}^{\frac{\beta_{p}}{2}}\ &N_{p}^{\frac{\beta_{p}}{2}}\\
			&N_{p}^{\frac{\beta_{p}}{2}}\ &N_{p}^{\frac{\beta_{p}}{2}}
		\end{pmatrix}.
	\end{align*}
	Combining these two bounds with \eqref{lowfrapriori} we get, for $N_{p}$ big enough
	\begin{align*}
		&\bigg|\begin{pmatrix}
			&\p_{1}u_{1}(\sum_{l=0}^{p}(\rho_{l}))(x=0,t)\ &\p_{2}u_{1}(\sum_{l=0}^{p}\rho_{l})(x=0,t)\\
			&\p_{1}u_{2}(\sum_{l=0}^{p}\rho_{l})(x=0,t)\ &\p_{2}u_{2}(\sum_{l=0}^{p}\rho_{l})(x=0,t)
		\end{pmatrix}\\
		&-C_{0}B_{p}(t)\cos(\alpha_{p}(t))\begin{pmatrix}
			&\sin(\alpha_{p}(t))\cos(\alpha_{p}(t))\ &\sin(\alpha_{p}(t))^2\\
			&-\cos(\alpha_{p}(t))^2\ &-\sin(\alpha_{p}(t))\cos(\alpha_{p}(t))
		\end{pmatrix}\bigg|\\
		&\leq C\begin{pmatrix}
			&N_{p}^{\frac{\beta_{p}}{2}}\ &N_{p}^{\frac{\beta_{p}}{2}}\\
			&N_{p}^{\frac{\beta_{p}}{2}}\ &N_{p}^{\frac{\beta_{p}}{2}}
		\end{pmatrix}.
	\end{align*}
	If we now write $\phi_{i}(x,\tilde{t},t)=z_{1,i}(t)x_{1}+z_{2,i}(t)x_{2}$, (where we leave the dependence on $\tilde{t}$ implicit to simplify the notation), then $z_{1,i},z_{2,i}$ fulfill
	$$\p_{t}z_{1,i}(t)+z_{1,i}(t)\p_{x_{1}}\P_{1}(u_{1}(\sum_{l=0}^{p}\rho_{l}))+z_{2,i}(t)\p_{x_{1}}\P_{1}(u_{2}(\sum_{l=0}^{p}\rho_{l}))=0$$
	$$\p_{t}z_{2,i}(t)+z_{1,i}(t)\p_{x_{2}}\P_{1}(u_{1}(\sum_{l=0}^{p}\rho_{l}))+z_{2,i}(t)\p_{x_{2}}\P_{1}(u_{2}(\sum_{l=0}^{p}\rho_{l}))=0$$
	$$z_{1,i}(t=\tilde{t})x_{1}+z_{2,i}(t=\tilde{t})x_{2}=x_{i}.$$
	If we now consider 
	$$\tilde{z}_{1,i}(t)=z_{1,i}(t)\cos(\alpha_{p}(t=1))+z_{2,i}(t)\sin(\alpha_{p}(t=1))$$
	$$\tilde{z}_{2,i}(t)=-z_{1,i}(t)\sin(\alpha_{p}(t=1))+z_{2,i}(t)\cos(\alpha_{p}(t=1))$$
	we get
	$$\p_{t}\tilde{z}_{1,i}=C_{0}B_{p}(t)\cos(\alpha_{p}(t))\tilde{z}_{2,i}+a(t)\tilde{z}_{1,i}+b(t)\tilde{z}_{2,i}$$
	$$\p_{t}\tilde{z}_{2,i}=c(t)\tilde{z}_{1,i}+d(t)\tilde{z}_{2,i}$$
	with $|a(t)|,|b(t)|,|c(t)|,|d(t)|\leq CN^{\frac{\beta_{p}}{2}}$.
	If we define $\tilde{Z}_{1,i}(t)=\text{sup}_{s\in[\tilde{t},1]}\tilde{z}_{1,i}(t)$ (we can make the same argument for $s\in[t_{0},\tilde{t}]$ for times smaller than $\tilde{t}$) we have that, by a Gronwall type estimate, using the bounds for $t-\tilde{t}$,
	$$|\tilde{z}_{2,i}(t)|\leq e^{\int_{\tilde{t}}^{t}d(s)ds}(\tilde{z}_{2,i}(\tilde{t})+\int_{\tilde{t}}^{t}b(s)\tilde{z}_{1,i}(s)ds)\leq e^{N_{p}^{-C\frac{\beta_{p}}{4}}}(1+\tilde{Z}_{1,i}(t)N_{p}^{-\frac{\beta_{p}}{4}})\leq C(1+\tilde{Z}_{1,i}(t)N_{p}^{-\frac{\beta_{p}}{4}}).$$
	Thus
	$$\p_{t}\tilde{Z}_{1,i}(t)\leq C_{0}B_{p}(t)\cos(\alpha_{p}(t))\tilde{z}_{2,i}+a(t)\tilde{z}_{1,i}+b(t)\tilde{z}_{2,i}\leq CN^{\frac{3\beta_{p}}{4}}\tilde{Z}_{1,i}+CN^{\beta_{p}}$$
	so, again by a Gronwall type estimate, using the bounds for $t-\tilde{t}$,
	$$\tilde{Z}_{1,i}(t)\leq e^{\int_{\tilde{t}}^{t}CN_{p}^{\frac{3\beta_{p}}{4}}}(\tilde{Z}_{1,i}(\tilde{t})+CN_{p}^{\frac{\beta_{p}}{4}})\leq CN_{p}^{\frac{\beta_{p}}{4}}.$$
	Finally, using the bounds for $\tilde{z}_{1,i}$, $\tilde{z}_{2,i}$ we obtain
	$$|z_{1,i}|,|z_{2,i}|\leq CN_{p}^{\frac{\beta_{p}}{4}},$$
	and since $\phi_{i}(x,\tilde{t},t)=z_{1,i}(t)x_{1}+z_{2,i}(t)x_{2}$ we have
	$$|\p_{x_{j}}\phi_{i}(x,\tilde{t},t)|\leq CN_{p}^{\frac{\beta_{p}}{4}}$$
	as we wanted to prove.

\end{proof}

 We can also obtain properties about how the $\alpha_{p+1}$ should behave in time when we construct our $p+1$-layered solution.
\begin{lemma}\label{girot0}
	Let $\rho(x,t)=\sum_{l=0}^{p}\rho_{l}(x,t)$ be a $p$-layered solution with parameters $(N_{i},\beta_{i},k_{i},K_{i})$ for $i=0,1,...,p$. If we consider $\alpha_{p+1}$ an angle moving with velocity $u(\rho(x,t))$ with
	$$\alpha_{p+1}(t=1)\in[\frac{\pi}{2},\frac{\pi}{2}+\frac{N_{p}^{-\frac{\beta_{p}}{8}}}{4}]$$
	 we have that, if $N_{p}$ is big enough (depending only on $\beta_{p}$), then if we define
 $\tilde{C}=\frac{16\pi^3}{C_{0}}$ (with $C_{0}$ the constant from Lemma \ref{c00}) there exists
 $t_{0}\in[1-\tilde{C}N_{p}^{-\frac{3\beta_{p}}{4}},1]$ such that
	$$\alpha_{p+1}(t_{0})-\alpha_{p}(t_{0})=-\frac{\pi}{2}$$
	and, for $t\in[t_{0},1]$, we have
	\begin{equation}\label{cotagiro}
		\p_{t}(\alpha_{p+1}(t)-\alpha_{p}(t))\leq 3C_{0}{N_{p}^{\frac{7\beta_{p}}{8}}}(\alpha_{p+1}-\alpha_{p})^2 ,
	\end{equation}
	$$\p_{t}(\alpha_{p+1}(t)-\alpha_{p}(t))\geq C_{0}\frac{N_{p}^{\frac{7\beta_{p}}{8}}}{4}\frac{(\alpha_{p+1}-\alpha_{p})^2}{\pi^2}.$$
\end{lemma}

\begin{proof}
	Using the notation 
 $$v_{p}:=u(\sum_{l=0}^{p-1}\rho_{l}),v_{p+1}:=u(\rho_{p})$$ 
 and the equation for the evolution of $\alpha_{i}(t)$ \eqref{angleeq}, we have that
	\begin{align*}
		&\p_{t}(\alpha_{p+1}(t)-\alpha_{p}(t))=\\
		&-(Dv_{p+1}-Dv_{p})^{t}
			\begin{pmatrix}
				\cos(\alpha_{p+1}) \\           
				\sin(\alpha_{p+1})
			\end{pmatrix}\cdot
			\begin{pmatrix}
				-\sin(\alpha_{p+1}) \\           
				\cos(\alpha_{p+1})
			\end{pmatrix}\\
		&-(Dv_{p}(x,t))^{t}\begin{pmatrix}
			\cos(\alpha_{p+1})-\cos(\alpha_{p}) \\           
			\sin(\alpha_{p+1})-\sin(\alpha_{p})
		\end{pmatrix}\cdot
		\begin{pmatrix}
			-\sin(\alpha_{p+1}) \\           
			\cos(\alpha_{p+1})
		\end{pmatrix}\\
		&+(Dv_{p}(x,t))^{t}\begin{pmatrix}
			\cos(\alpha_{p}) \\           
			\sin(\alpha_{p})
		\end{pmatrix}\cdot
		\begin{pmatrix}
			-\sin(\alpha_{p})+\sin(\alpha_{p+1}) \\           
			\cos(\alpha_{p})-\cos(\alpha_{p+1})
		\end{pmatrix}.
			\end{align*}
			Furthermore, using \eqref{lowfrapriori} we obtain
		
			\begin{align*}
				\Big| (Dv_{p}(x,t))^{t}\begin{pmatrix}
					\cos(\alpha_{p+1})-\cos(\alpha_{p}) \\           
					\sin(\alpha_{p+1})-\sin(\alpha_{p})
				\end{pmatrix}\cdot
				\begin{pmatrix}
					-\sin(\alpha_{p+1}) \\           
					\cos(\alpha_{p+1})
				\end{pmatrix}\Big|\leq 8|\alpha_{p+1}-\alpha_{p}|N_{p}^{\frac{\beta_{p}}{8}},\\
				\Big|(Dv_{p}(x,t))^{t}\begin{pmatrix}
					\cos(\alpha_{p}) \\           
					\sin(\alpha_{p})
				\end{pmatrix}\cdot
				\begin{pmatrix}
					-\sin(\alpha_{p})+\sin(\alpha_{p+1}) \\           
					\cos(\alpha_{p})-\cos(\alpha_{p+1})
				\end{pmatrix}\Big|\leq 8|\alpha_{p+1}-\alpha_{p}|N_{p}^{\frac{\beta_{p}}{8}}.
			\end{align*}
But we also have, using the properties \eqref{anguloup}-\eqref{anguloup2} for the velocity of a $p$-layered solution, that
$$|(Dv_{p+1}-Dv_{p})^{t}-\begin{pmatrix}
	B_{i}(t)C_{0}\sin(\alpha_{p}(t))\cos(\alpha_{p}(t))^2 & -B_{i}(t)C_{0}\cos(\alpha_{p}(t))^3\\
	B_{i}(t)C_{0}\sin(\alpha_{p}(t))^2\cos(\alpha_{p}(t)) & -B_{i}(t)C_{0}\sin(\alpha_{p}(t))\cos(\alpha_{p}(t))^2
\end{pmatrix}|\leq \begin{pmatrix}
N_{p}^{\frac{\beta_{p}}{8}} & N_{p}^{\frac{\beta_{p}}{8}}\\
N_{p}^{\frac{\beta_{p}}{8}} & N_{p}^{\frac{\beta_{p}}{8}}
\end{pmatrix}.$$	
We now compute that last term, of the evolution equation, which is
\begin{align*}
	-&B_{i}(t)C_{0}\cos(\alpha_{p}(t))\begin{pmatrix}
		\sin(\alpha_{p}(t))\cos(\alpha_{p}(t)) & -\cos(\alpha_{p}(t))^2\\
		\sin(\alpha_{p}(t))^2 & -\sin(\alpha_{p}(t))\cos(\alpha_{p}(t))
	\end{pmatrix}\begin{pmatrix}
	\cos(\alpha_{p+1}(t))\\ \sin(\alpha_{p+1}(t))
	\end{pmatrix}\cdot\begin{pmatrix}
	-\sin(\alpha_{p+1}(t))\\ \cos(\alpha_{p+1}(t))
	\end{pmatrix}\\
	&=-B_{i}(t)C_{0}\cos(\alpha_{p}(t))\sin(\alpha_{p}(t)-\alpha_{p+1}(t))\begin{pmatrix}
		\cos(\alpha_{p}(t))\\ \sin(\alpha_{p}(t))
	\end{pmatrix}\cdot\begin{pmatrix}
		-\sin(\alpha_{p+1}(t))\\ \cos(\alpha_{p+1}(t))
	\end{pmatrix}\\
	&=-B_{i}(t)C_{0}\cos(\alpha_{p}(t))\sin(\alpha_{p}(t)-\alpha_{p+1}(t))^2.
	\end{align*}	
We can now use that, for $x\in[-\frac{\pi}{2},\frac{\pi}{2}]$, 
$$ \frac{x^2}{\pi^2}\leq \sin(x)^2\leq x^2$$	
plus the bounds for $\alpha_{p}(t)$ obtained in Lemma \ref{controlangulocorto} and the bounds for $B(t)$  
$$\frac{N_{p}^{\beta_{p}}}{2}\leq B(t)\leq 2N_{p}^{\beta_{p}},\  \cos(\alpha_{p}(t))\in [-2N_{p}^{-\frac{\beta_{p}}{8}},-\frac{N_{p}^{-\frac{\beta_{p}}{8}}}{2\pi}]$$
to get

$$\p_{t}(\alpha_{p+1}(t)-\alpha_{p}(t))\leq2C_{0}{N_{p}^{\frac{7\beta_{p}}{8}}}(\alpha_{p+1}-\alpha_{p})^2 +|\alpha_{p+1}-\alpha_{p}|16N_{p}^{\frac{\beta_{p}}{8}}+4N_{p}^{\frac{\beta_{p}}{8}},$$
$$\p_{t}(\alpha_{p+1}(t)-\alpha_{p}(t))\geq C_{0}\frac{N_{p}^{\frac{7\beta_{p}}{8}}}{2\pi}\frac{(\alpha_{p+1}-\alpha_{p})^2}{\pi^2} -|\alpha_{p+1}-\alpha_{p}|16N_{p}^{\frac{\beta_{p}}{8}}-4N_{p}^{\frac{\beta_{p}}{8}}.$$
Then, we have that, if $\alpha_{p+1}(t)-\alpha_{p}(t)\in[-\frac{\pi}{2},-\frac{-N_{p}^{-\frac{\beta_{p}}{8}}}{4}]$

$$\p_{t}(\alpha_{p+1}(t)-\alpha_{p}(t))\leq 2C_{0}{N_{p}^{\frac{7\beta_{p}}{8}}}(\alpha_{p+1}-\alpha_{p})^2+(\alpha_{p+1}-\alpha_{p})^2 4N^{\frac{\beta_{p}}{8}}(16N_{p}^{\frac{\beta_{p}}{8}}+16N_{p}^{\frac{\beta_{p}}{4}}),$$
$$\p_{t}(\alpha_{p+1}(t)-\alpha_{p}(t))\geq C_{0}\frac{N_{p}^{\frac{7\beta_{p}}{8}}}{2\pi}\frac{(\alpha_{p+1}-\alpha_{p})^2}{\pi^2} -(\alpha_{p+1}-\alpha_{p})^{2}4N_{p}^{\frac{\beta_{p}}{8}}(16N_{p}^{\frac{\beta_{p}}{8}}+16N_{p}^{\frac{\beta_{p}}{4}}).$$
Finally, using that $N_{p}$ is big, we obtain

$$\p_{t}(\alpha_{p+1}(t)-\alpha_{p}(t))\leq  3C_{0}{N_{p}^{\frac{7\beta_{p}}{8}}}(\alpha_{p+1}-\alpha_{p})^2 ,$$
$$\p_{t}(\alpha_{p+1}(t)-\alpha_{p}(t))\geq C_{0}\frac{N_{p}^{\frac{7\beta_{p}}{8}}}{4\pi}\frac{(\alpha_{p+1}-\alpha_{p})^2}{\pi^2}.$$
Note that, since the time derivative is positive and the bounds obtained are correct as long as $\alpha_{p+1}(t)-\alpha_{p}(t)\in[-\frac{\pi}{2},-\frac{-N_{p}^{-\frac{\beta_{p}}{8}}}{2}]$ and $t\in[1-\tilde{C} N_{p}^{-\frac{3\beta_{p}}{4}},1]$, the bounds will be correct for $t\in[t_{0},1]$, with $t_{0}$ the smallest time in $[1-\tilde{C}N_{p}^{-\frac{3\beta_{p}}{4}},1]$  such that $\alpha_{p}(t_{0})-\alpha_{p+1}(t_{0})\geq -\frac{\pi}{2}$.

But now, we can integrate backwards in time the upper bound for the derivative of $(\alpha_{p+1}(t)-\alpha_{p}(t))$ to get
$$-\frac{1}{c_{1}(\alpha_{p+1}(t_{0})-\alpha_{p}(t_{0}))}+\frac{1}{c_{1}(\alpha_{p+1}(t=1)-\alpha_{p}(t=1))}\leq(t_{0}-1)$$
$$1-t_{0}\leq \frac{1}{c_{1}\frac{N_{p}^{\frac{-\beta_{p}}{8}}}{4}}$$
where $c_{1}=C_{0}\frac{N_{p}^{\frac{7\beta_{p}}{8}}}{4\pi^3}$.
This gives us $t_{0}\in [1-N_{p}^{-\frac{3}{4}}\frac{16\pi^3}{C_{0}},1]$ which finishes the proof.

\end{proof}

The point of the previous lemmas is that, when we add a new layer $\rho_{p+1}(x,t)$ to our $p$-layered solution, if it is very concentrated around the origin, we expect the angle $\alpha_{p+1}(t)$ to give us some useful insight regarding the evolution of the new layer. More precisely, this captures (part of) the change on the new layer $\rho_{p+1}(x,t)$ coming from the velocity generated by the $\sum_{i=1}^{p}\rho_{i}(x,t)$.

If we, for simplicity, consider now that $$\rho_{p+1}(x,t=1)=\sin(N_{p+1}(x_{1}\cos(\frac{\pi}{2}+N_{p+1}^{-\frac{\beta_{p+1}}{8}})+x_{2}\sin(\frac{\pi}{2}+N_{p+1}^{-\frac{\beta_{p+1}}{8}}))),$$ We can model the effect of the velocity generated by $\rho_{p}$ using the results from Subsection \ref{subsectrunc}, namely we consider
$$\rho_{p+1}(x,t)\approx \sin(N_{p+1}X(x,t))$$
where $X=\P_{M}\tilde{X}(x,t)$
$$\p_{t}\P_{M}\tilde{X}+\P_{M}([\P_{M}u(\sum_{l=0}^{p}\rho_{l})]\cdot\nabla \P_{M}\tilde X)=0$$
$$\tilde{X}(x,t=1)=x_{1}\cos(\frac{\pi}{2}+N_{p+1}^{-\frac{\beta_{p+1}}{8}})+x_{2}\sin(\frac{\pi}{2}+N_{p+1}^{-\frac{\beta_{p+1}}{8}}),$$
and since $X(x,t)$ is a polynomial in $x_{1}$ and $x_{2}$, we can decompose
$$\sin(N_{p+1}X(x,t))=\sin(N_{p+1}\P_{1}X(x,t))\cos(N_{p+1}(X-\P_{1}X))+\cos(N_{p+1}\P_{1}X(x,t))\sin(N_{p+1}(X-\P_{1}X))$$
and each of the summands can be written as $g(x)\sin(N_{p+1}(bx_{1}+ax_{2})+\theta_{0})$, so that we can apply the results from Section \ref{secvelocity}.
Then, if we assume that $\rho_{p+1}$ is very localized around the origin (which, in this simplified case is not true, but it will be a correct assumption when we do the proper construction later on), we can use the approximation
$$u(\rho_{p+1})(x,t)\cdot\nabla \rho_{p}(x,t)\approx u^{0}(\rho_{p+1})(x,t)\cdot(\nabla \rho_{p})(x=0,t).$$
The interesting thing is, this term can actually be approximated very nicely using the angles of $\rho_{p+1}$ and $\rho_{p}$. Namely, if $X(x,t)$ has angle $\alpha_{p+1}$, we then have
$$(u^{0}_{1}(\rho_{p+1}),u^{0}_{2}(\rho_{p+1}))=C_{0}\cos(\alpha_{p+1})(\sin(\alpha_{p+1},-\cos(\alpha_{p+1})))\rho_{p+1}$$
and we have, from \eqref{derivadasrhoi}, that
$$((\p_{x_{1}}\rho_{p})(x=0,t),(\p_{x_{2}}\rho_{p})(x=0,t))\approx A_{p}(t)(\cos(\alpha_{p+1}),\sin(\alpha_{p+1})).$$

We can use this to obtain that

\begin{align*}
	&u^{0}(\rho_{p+1})\cdot\nabla (\sum_{i=1}^{p}\rho_{i})\\
	&\approx \rho_{p+1}A_{p}(t)C_{0}\cos(\alpha_{p+1}(t))[\sin(\alpha_{p+1}(t))\cos(\alpha_{p+1}(t))-\cos(\alpha_{p+1}(t))\sin(\alpha_{p}(t))]\\
	&=\rho_{p+1}A_{p}(t)C_{0}\cos(\alpha_{p+1}(t))\sin(\alpha_{p+1}(t)-\alpha_{p}(t)).
\end{align*}

This suggest approximating the solution by 
$$A_{p+1}(t)\sin(NX(t))$$
with
$$\p_{t}A_{p+1}(t)=A_{p+1}(t)A_{p}(t)C_{0}\cos(\alpha_{p+1}(t))\sin(\alpha_{p+1}(t)-\alpha_{p}(t))+G_{error}(\alpha_{p+1},t)$$
where $G_{error}$ includes the term we ignored when making the approximation $$((\p_{x_{1}}\rho_{p})(x=0,t),(\p_{x_{2}}\rho_{p})(x=0,t))\approx A_{p}(t)(\cos(\alpha_{p+1}),\sin(\alpha_{p+1})).$$

The next lemma obtains some properties of the evolution equation we just obtained for the amplitude of $\rho_{p+1}$.
\begin{lemma}\label{crecimientoAp}
	Let $\rho(x,t)=\sum_{l=0}^{p}\rho_{l}$ be a $p-$layered solution and $\mathcal{A}(t)$ be a function fulfilling
	$$\p_{t}\mathcal{A}(t)=\mathcal{A}(t)[A_{p}(t)C_{0}\cos(\alpha_{p+1}(t))\sin(\alpha_{p}(t)-\alpha_{p+1}(t))+G_{error}(\alpha_{p+1},t)]$$
	$$G_{error}=A_{p,err}(t)C_{0}\cos(\alpha_{p+1}(t))\sin(\alpha_{p,err}(t)-\alpha_{p+1}(t))$$
	with
	$$(A_{p,err}(t)\cos(\alpha_{p,err}),A_{p,err}(t)\sin(\alpha_{p,err}))=(\p_{x_{1}}\sum_{l=0}^{p}\rho_{l}(x=0)-A_{p}\cos(\alpha_{p}),\p_{x_{2}}\sum_{l=0}^{p}\rho_{l}(x=0)-A_{p}\sin(\alpha_{p}))$$
 and $\alpha_{p+1}(t)$ the angle moving with velocity $u(\rho)$ according to Definition \ref{angle&mod}
 with
	$$\alpha_{p+1}(t=1)\in[\frac{\pi}{2},\frac{\pi}{2}+N_{p}^{-1}]$$
	and let $t_{0}\in[1-\tilde{C}N_{p}^{-\frac{3\beta_{p}}{4}},1]$ be the value given by Lemma \ref{girot0} such that
	$$\alpha_{p+1}(t_{0})-\alpha_{p}(t_{0})=-\frac{\pi}{2}.$$
	
	Then we have that, if $N_{p}$ is big enough (depending on $\beta_{p}$)
	$$ e^{K_{+}N_{p}^{\frac{\beta_{p}}{4}}}\geq \frac{\mathcal{A}(t=1)}{\mathcal{A}(t=t_{0})}\geq e^{K_{-}N_{p}^{\frac{\beta_{p}}{8}}}.$$
	Furthermore, for $1\geq s_{2}\geq s_{1}\geq t_{0}$ we have
	\begin{equation}\label{amplitudmonotona}
		\frac{\mathcal{A}(s_{2})}{\mathcal{A}(s_{1})}\geq e^{-2}.
	\end{equation}

\end{lemma}
\begin{proof}
	To obtain the bounds, we note that, by integrating in time the evolution equation for $\mathcal{A}(t)$ we get
	$$\frac{\mathcal{A}(t=1)}{\mathcal{A}(t=t_{0})}=e^{\int_{t_{0}}^{1}[A_{p}(s)\cos(\alpha_{p+1}(s))\sin(\alpha_{p}(s)-\alpha_{p+1}(s))+G_{error}(\alpha_{p+1},s)]ds}.$$
	For the upper bound we just use that 
	$$|A_{p}(t)\cos(\alpha_{p+1}(t))\sin(\alpha_{p}(t)-\alpha_{p+1}(t))+G_{error}(\alpha_{p+1},t)|\leq 4N_{p}^{\beta_{p}}$$
	which combined with the bounds for $|1-t_{0}|$ gives the desired bound after integrating in time.
	
	For the lower bound, we note that, during the times considered, $A_{p}(t)\cos(\alpha_{p+1}(t))\sin(\alpha_{p}(t)-\alpha_{p+1}(t))$ has the same sign as $\cos(\alpha_{p+1}(t))$. Note also that, during the times considered, $\alpha_{p+1}(t)-\alpha_{p}(t)$ is monotonous increasing according to Lemma \ref{girot0}. In particular we can define
	$$t_{1}:=\{t: t\in[t_{0},1], \alpha_{p+1}(t)-\alpha_{p}(t)=-\frac{\pi}{2}+\frac{\pi}{10}\},$$
	$$t_{2}:=\{t: t\in[t_{0},1], \alpha_{p+1}(t)-\alpha_{p}(t)=-2N^{-\frac{\beta_{p}}{8}}_{p}\}.$$
	By Lemma \ref{controlangulocorto}, we have that, for $t\in[t_{0},t_{2}]$ 
	$$\alpha_{p+1}(t)=\alpha_{p}(t)+(\alpha_{p+1}(t)-\alpha_{p}(t))\in [0,\frac{\pi}{2}]$$
	so $\cos(\alpha_{p+1}(t))>0$ and
	\begin{align*}
		&\int_{t_{0}}^{1}A_{p}(s)\cos(\alpha_{p+1}(s))\sin(\alpha_{p}(s)-\alpha_{p+1}(s))ds\\
		&\geq \int_{t_{0}}^{t_{1}}A_{p}(s)\cos(\alpha_{p+1}(s))\sin(\alpha_{p}(s)-\alpha_{p+1}(s))ds+\int_{t_{2}}^{1}A_{p}(s)\cos(\alpha_{p+1}(s))\sin(\alpha_{p}(s)-\alpha_{p+1}(s))ds
	\end{align*}
	and, using that, by Lemma \ref{controlangulocorto} for $t\in[t_{2},1]$ we have $|\cos(\alpha_{p+1}(t))|\leq 2N_{p}^{-\frac12}$, we get
	$$|\int_{t_{2}}^{1}A_{p}(s)\cos(\alpha_{p+1}(s))\sin(\alpha_{p}(s)-\alpha_{p+1}(s))ds|\leq 2N_{p}^{-\frac12}|\int_{t_{2}}^{1}A_{p}(s)ds|\leq 1.$$
	For the other contribution to the integral, we have
	$$\int_{t_{0}}^{t_{1}}A_{p}(s)\cos(\alpha_{p+1}(s))\sin(\alpha_{p}(s)-\alpha_{p+1}(s))ds\geq \frac{1}{2} \int_{t_{0}}^{t_{1}}A_{p}(s)ds$$
	and using \eqref{cotagiro} we have
	$$t_{1}-t_{0}\geq CN_{p}^{\frac{-7\beta_{p}}{8}}$$
	so in particular, 
	$$\int_{t_{0}}^{t_{1}}A_{p}(s)\cos(\alpha_{p+1}(s))\sin(\alpha_{p}(s)-\alpha_{p+1}(s))ds\geq CN_{p}^{\frac{\beta_{p}}{8}}.$$
	But, since, for $N_{p}$ big enough
	\begin{align*}
		&(A_{p,err})^2=(\p_{x_{1}}\sum_{l=0}^{p}\rho_{l}(x=0)-A_{p}\cos(\alpha_{p}))^2+(\p_{x_{2}}\sum_{l=0}^{p}\rho_{l}(x=0)-A_{p}\sin(\alpha_{p})))^2\\
		&\leq 2(2N_{p}^{\frac{\beta_{p}}{8}})^2\leq CN_{p}^{\frac{\beta_{p}}{4}}
	\end{align*}
	so, for big $N_{p}$
	$$|\int_{t_{0}}^{1}G_{error}(\alpha_{p+1},s)ds|\leq CN_{p}^{\frac{\beta_{p}}{8}} N_{p}^{-\frac{3\beta_{p}}{4}}\leq 1$$
	and thus
	$$\frac{\mathcal{A}(t=1)}{\mathcal{A}(t=t_{0})}\geq e^{K_{-}N_{p}^{\frac{\beta_{p}}{8}}}.$$
	Similarly, for \eqref{amplitudmonotona} we just use
	\begin{align*}
		&\frac{\mathcal{A}(s_{2})}{\mathcal{A}(s_{1})}=e^{\int_{s_{1}}^{s_{2}}A_{p}(s)\cos(\alpha_{p+1}(s))\sin(\alpha_{p}(s)-\alpha_{p+1}(s))+G_{error}(\alpha_{p+1},s)ds}\\
		&\geq e^{\int_{t_{2}}^{1}A_{p}(s)\cos(\alpha_{p+1}(s))\sin(\alpha_{p}(s)-\alpha_{p+1}(s))ds+\int_{t_{0}}^{1}G_{error}(\alpha_{p+1},s)ds}\\
		&\geq e^{-2}.
	\end{align*}

\end{proof}
\begin{remark}\label{remarknp1}
	In the definition of the $p$-layered solutions there is some leeway for the relationship between $N_{i}$ and $N_{i+1}$, so when constructing our $p+1$-layered solution from a given $p$-layer solution, we need to make a choice of $N_{p+1}$, and a priori it is not necessarily clear what the choice should be. Given some value of $\beta_{p+1}$ and
	$$\alpha_{p+1}(t=1)\in[\frac{\pi}{2},\frac{\pi}{2}+N_{p}^{-1}]$$
	this implicitly makes a choice of $N_{p+1}$ since we also have
	$$\alpha_{p+1}(t=1)=N_{p+1}^{-\frac{\beta_{p+1}}{8}}+\frac{\pi}{2}.$$
	Given $\beta_{p+1}$, we will choose $N_{p+1}$ such that
	\begin{equation}\label{choicetp}
		\frac{\mathcal{A}(t=1)}{\mathcal{A}(t=t_{0})}=N_{p+1}^{\frac{1}{\beta_{p+1}}}
	\end{equation}
	where $t_{0}$ is given by Lemma \ref{crecimientoAp}.
 Remembering than $t_{0}$ actually depends on $N_{p+1}$ (through the choice of the angle at $t=1$), we need to show that there exists $N_{p+1}$ such that
 $$\frac{\mathcal{A}(t=1)}{\mathcal{A}(t=t_{0})}-N_{p+1}^{\frac{1}{\beta_{p+1}}}=0$$
But then we can use the bounds for
	$\frac{\mathcal{A}(t=1)}{\mathcal{A}(t=t_{0})}$
	from Lemma \ref{crecimientoAp}  and the continuity of $t_{0}$, $\frac{\mathcal{A}(t=1)}{\mathcal{A}(t=t_{0})}$ with respect $\alpha_{p+1}(t=1)$ (which comes from the fact that all the coefficients in the evolution equation are smooth in time and the continuity theory for ODEs), and therefore, continuity with respect to the value of $N_{p+1}$, to obtain that there exists a valid $N_{p+1}$ fulfilling
	$$ e^{K_{+}\beta_{p+1}N_{p}^{\frac{\beta_{p}}{4}}}\geq N_{p+1}\geq e^{K_{-}\beta_{p+1}N_{p}^{\frac{\beta_{p}}{8}}}.$$

\end{remark}

With all this we are ready to give the first order solution operator.

\begin{definition}\label{soloperator}
	Given a p-layered solution and a choice of $\beta_{p+1}, N_{p+1}$, and  $M:=\lceil\frac{1}{\beta_{p+1}^2}\rceil$, we define
	$X(x,t)=\P_{M}\tilde{X}(x,t)$
	$$\p_{t}\P_{M}\tilde{X}+\P_{M}([\P_{M}u(\sum_{l=0}^{p}\rho_{l})]\cdot\nabla \P_{M}\tilde X)=0$$
	$$\tilde{X}(x,t=1)=x_{1}\cos(\frac{\pi}{2}+N_{p+1}^{-\frac{\beta_{p+1}}{8}})+x_{2}\sin(\frac{\pi}{2}+N_{p+1}^{-\frac{\beta_{p+1}}{8}}),$$
	and given a function $f(x)$, we define 
	$$S^{M}_{t_{1},t_{2}}[f(x)\sin(NX(x,t_{1})+\theta_{0})]:=\mathcal{A}_{t_{1}}(t_{2})f(x,t_{2})\sin(NX(x,t_{2})+\theta_{0})$$
	with
	$$\p_{t}\mathcal{A}_{t_{1}}(t)=\mathcal{A}_{t_{1}}(t)[A_{p}(t)C_{0}\cos(\alpha(t))\sin(\alpha_{p}(t)-\alpha(t))+G_{error}(\alpha,t)],$$
	$$G_{error}(t,\alpha)=A_{p,err}(t)C_{0}\cos(\alpha(t))\sin(\alpha_{p,err}(t)-\alpha(t)),$$
	with 
	$$(A_{p,err}(t)\cos(\alpha_{p,err}),A_{p,err}(t)\sin(\alpha_{p,err}))=(\p_{x_{1}}\sum_{l=0}^{p}\rho_{l}(x=0)-A_{p}\cos(\alpha_{p}),\p_{x_{2}}\sum_{l=0}^{p}\rho_{l}(x=0)-A_{p}\sin(\alpha_{p}))$$
	$$\mathcal{A}_{t_{1}}(t_{1})=1$$
	$\alpha(t)$ the angle corresponding to $X(x,t)$ (as in Remark \ref{alphapoly} ) and
	$$\p_{t}f(x,t)=-\P_{1}(u(\sum_{l=0}^{p}\rho_{l}))\cdot\nabla f(x,t)$$
	$$f(x,t=t_{1})=f(x).$$
	We will refer to this operator as the first order solution operator.
	
\end{definition}

\begin{remark}\label{remarkunfsol}
    If we take a look at 
    $$S^{M}_{t_{1},t}[f(x)\sin(NX(x,t_{1})+\theta_{0})]=\mathcal{A}_{t_{1}}(t)f(x,t)\sin(NX(x,t)+\theta_{0}):=h(x,t)\sin(NX(x,t)+\theta_{0})$$ 
    we can compute the time derivative as
    \begin{align*}
       &\frac{d}{dt}S^{M}_{t_{1},t}[f(x)\sin(NX(x,t_{1})+\theta_{0})] =\frac{d}{dt}(\mathcal{A}_{t_{1}}(t)f(x,t)\sin(NX(x,t)+\theta_{0}))\\
       &=(\frac{d}{dt}\mathcal{A}_{t_{1}}(t))f(x,t)\sin(NX(x,t)+\theta_{0})\\
       &-\P_{M}(u(\sum_{l=0}^{p}\rho_{l})) h(x,t)\cdot\nabla \sin(NX(x,t)+\theta_{0})- \P_{1}(u(\sum_{l=0}^{p}\rho_{l})) \sin(NX(x,t)+\theta_{0}) \cdot\nabla h(x,t)
    \end{align*}
and furthermore
\begin{align}
&(\frac{d}{dt}\mathcal{A}_{t_{1}}(t))f(x,t)\sin(NX(x,t)+\theta_{0})\\
&=\mathcal{A}_{t_{1}}(t)f(x,t)\sin(NX(x,t)+\theta_{0})
    [A_{p}(t)C_{0}\cos(\alpha(t))\sin(\alpha_{p}(t)-\alpha(t))+G_{error}(\alpha,t)]\\
    &=-u_{0}(\mathcal{A}_{t_{1}}(t)f(x,t)\sin(NX(x,t)+\theta_{0}))\cdot\nabla \sum_{l=0}^{p}\rho_{l}(x=0,t)
\end{align}
so
\begin{align*}
       &\frac{d}{dt}S^{M}_{t_{1},t}[f(x)\sin(NX(x,t_{1})+\theta_{0})] \\
       &=-u_{0}(h(x,t)\sin(NX(x,t)+\theta_{0}))\cdot\nabla \sum_{l=0}^{p}\rho_{l}(x=0,t)\\
       &-\P_{M}(u(\sum_{l=0}^{p}\rho_{l})) h(x,t)\cdot\nabla \sin(NX(x,t))- \P_{1}(u(\sum_{l=0}^{p}\rho_{l})) \sin(NX(x,t)) \cdot\nabla h(x,t).
    \end{align*}

\end{remark}
\begin{remark}\label{evolsol0}
	If we consider a source term $F(x,s)$ fulfilling
 $$S^{M}_{s,t}[F(x,s)]=h(x,t,s)\sin(NX(x,t))$$ 
 and then we define
	$$w(x,t)\sin(NX(x,t)):=\int_{\tilde{t}}^{t}S^{M}_{s,t}[F(x,s)]ds=\sin(NX(x,t))\int_{\tilde{t}}^{t}h(x,t,s)ds$$
	we have that
	\begin{align*}
		\p_{t}w(x,t)&=h(x,t,t)\\&-\int_{\tilde{t}}^{t}[\P_{1}(u(\sum_{l=0}^{p}\rho_{l})(x,t))\cdot\nabla  h(x,t,s)+\frac{u^{0}(h(x,t,s)\sin(NX(x,t)))}{\sin(NX(x,t))}\cdot\nabla (\sum_{l=0}\rho_{l}(x=0,t))]ds
	\end{align*}
	and noting that we can write
	$$\frac{u^{0}(h(x,t,s)\sin(NX(x,t))}{\sin(NX(x,t))}\cdot\nabla(\sum_{l=0}\rho_{l}(x=0,t))=H(x,t)h(x,t,s)$$
 and
    $$\frac{u^{0}(w(x,t)\sin(NX(x,t))}{\sin(NX(x,t))}\cdot\nabla(\sum_{l=0}\rho_{l}(x=0,t))=H(x,t)w(x,t)$$
	we get
	\begin{align*}
		\p_{t}w(x,t)&=\frac{F(x,t)}{\sin(NX(x,t))}-\int_{\tilde{t}}^{t}[\P_{1}(u(\sum_{l=0}^{p}\rho_{l}))\cdot\nabla h(x,t,s)+h(x,t,s)H(x,t)]ds\\
		&=\frac{F(x,t)}{\sin(NX(x,t))}-\P_{1}(u(\sum_{l=0}^{p}\rho_{l}))\cdot\nabla \int_{\tilde{t}}^{t} h(x,t,s)ds-H(x,t)\int_{\tilde{t}}^{t}h(x,t,s)ds\\
		&=\frac{F(x,t)}{\sin(NX(x,t))}-\P_{1}(u(\sum_{l=0}^{p}\rho_{l}))\cdot\nabla w(x,t)-H(x,t)w(x,t)
	\end{align*}
	so that
	\begin{align*}
		\p_{t}&(w(x,t)\sin(NX(x,t)))=\p_{t}\int_{\tilde{t}}^{t}S^{M}_{s,t}[F(x,s)]ds\\
		=&F(x,t)-\sin(NX(x,t))\P_{1}(u(\sum_{l=0}^{p}\rho_{l}))\cdot\nabla w(x,t)-u^{0}(\sin(NX(x,t))w(x,t))\cdot\nabla(\sum_{l=0}^{p}\rho_{l}(x=0,t))\\
		&-\P_{M}(u(\sum_{l=0}^{p}\rho_{l})) w(x,t)\cdot\nabla \sin(NX(t)).
	\end{align*}

Note that the same result holds if we use $\sin(NlX(x,t)+\theta_{0})$ instead of $\sin(NX(x,t))$.

\end{remark}

\begin{remark}
	The purpose of this operator is that, if we add a perturbation to a $p-$layered solution at time $t_{1}$, it will give us an approximation for how this perturbation will look at time $t_{2}$. For this approximation to be somewhat reasonable there are several implicit assumptions: $\sin(NX(x,t))$ is expected to be the higher frequency term, which is why we approximate the transport over this term to a relatively high degree of accuracy. $f(x,t)$, on the other hand, is supposed to be some lower frequency cut-off function, which justifies approximating the transport of this part by a first order approximation (we approximate the velocity by a linear function, instead of a higher order polynomial). Finally, the evolution in $\mathcal{A}_{t_{1}}(t)$, the amplitude, is obtained from approximating $u(\mathcal{A}_{t_{1}}(t)f(x,t)\sin(NX(x,t)))\approx u^{0}(\mathcal{A}_{t_{1}}(t)f(x,t)\sin(NX(x,t)))$.
	
	We will use this solution operator to construct our $p+1$-layered solution, by fixing the value of the new layer at a fixed time, using the solution operator to study its evolution, and then we compute the error we are committing when we use this solution operator instead of actually solving IPM. We can then apply the solution operator to (most of) these errors, and iterating this process allow us to obtain the behavior of the $p+1$-layered solution with an error that can be as regular as we want.
\end{remark}
Note that we can combine our definition of $X(x,t)$ with the properties obtained in subsection \ref{subsectrunc} to obtain useful bounds for the coefficients of $X(x,t)$. More precisely, we have the following lemma.

\begin{lemma}
	Given a p-layered solution and a choice of $\beta_{p+1},N_{p+1}$,  if we define $M:=\lceil\frac{1}{\beta_{p+1}^2}\rceil$ and $X(x,t)$ defined as
	 $X:=\P_{M}\tilde{X}(x,t)$ with
	$$\p_{t}\P_{M}\tilde{X}+\P_{M}([\P_{M}u(\sum_{l=0}^{p}\rho_{l})]\cdot\nabla \P_{M}\tilde X)=0$$
	$$\tilde{X}(x,t=1)=x_{1}\cos(\frac{\pi}{2}+N_{p+1}^{-\frac{\beta_{p+1}}{8}})+x_{2}\sin(\frac{\pi}{2}+N_{p+1}^{-\frac{\beta_{p+1}}{8}}),$$
 then we have that, if $N_{p}$ is big enough (depending only on $\beta_{p}$) for $t\in [t_{0},1]$, 
	$$X(x,t)=\sum_{i+j=1,...,M,\ i,j\geq 0}c_{i,j}(t)x_{1}^ix_{2}^{j}$$
	with
	$$|c_{i,j}|\leq C_{\beta_{p}} N_{p}^{C_{\beta_{p}}}.$$
\end{lemma}

\begin{proof}
	We can apply Lemmas \ref{controlpolinomio} and  \ref{p1phi} to get, for $0\leq i+j\leq M$ (using $K_{p}\geq M+1$)
	$$|c_{i,j}|\leq C_{M}N_{p}^{\frac{\beta_{p}}{4}(i+j+2)}\|u(x,t)\|^{i+j-1}_{C^{M}}\leq C_{M}N^{\frac{\beta_{p}}{4}(M+2)}N_{p}^{M}\ln(N_{p})$$
	as we wanted to prove.
\end{proof}
\begin{corollary}\label{cotasx1np1}
		Given a p-layered solution and a choice of $\beta_{p+1},N_{p+1}$,  if we define $M:=\lceil\frac{1}{\beta_{p+1}^2}\rceil$ and $X(x,t)$ defined as
	$X:=\P_{M}\tilde{X}(x,t)$ with
	$$\p_{t}\P_{M}\tilde{X}+\P_{M}([\P_{M}u(\sum_{l=0}^{p}\rho_{l})]\cdot\nabla \tilde X)=0$$
	$$\tilde{X}(x,t=1)=x_{1}\cos(\frac{\pi}{2}+N_{p+1}^{-\frac{\beta_{p+1}}{8}})+x_{2}\sin(\frac{\pi}{2}+N_{p+1}^{-\frac{\beta_{p+1}}{8}}),$$
	with $t_{0}$ as in Lemma \ref{girot0}, 
	 then we have that, for $t\in [t_{0},1]$, 
	$$X(x,t)=\sum_{j=1}^{M}\sum_{i=0}^{j}c_{i,j}(t)x_{1}^ix_{2}^{j-i}$$
	with
	$$|c_{i,j}|\leq C_{\beta_{p}} \ln(N_{p+1})^{C_{\beta_{p},\beta_{p+1}}}.$$

\end{corollary}
\begin{remark}
    Note that we have relabeled the $c_{i,j}$, so that the coefficient $i$ actually gives us the total order of derivatives. We will from now on use this notation, since it is more convenient in the next lemmas we need to prove.
\end{remark}
Before we can finally start the inductive process that will allow us to construct our $p+1-$layered solution given a $p$-layered solution, there is a couple of technical hurdles that we want to deal with.
First, since we will be dealing with functions of the form $f(x)\sin(P(x))$, with $P(x)$ a polynomial, we will obtain some important properties regarding the regularity of this kind of functions.

 Second, even though we have obtained useful properties about $u^{k}$ (our approximations for the velocity) when applied to something of the form $f(x)\sin(Nax_{1}+Nbx_{2}+C)$, we want to work with the more general functions
$$f(x)\sin(NlX(x,t)+\theta_{0}).$$
Therefore, we need to define precisely what the operator $u^{k}$ means when applied to this kind of functions, and we would also want to obtain useful bounds for $u^k$ in that case.

\begin{lemma}\label{boundsfinalproduct}
	Given $c_{1},c_{2}>0$, a smooth function $f(x)$ with support in $B_{1}(0)$, $K\in\mathds{N}$ and 
	$$g_{N}(x)=\sum_{j=0}^{K}\sum_{i=0}^{j}c_{i,j}x_{1}^{i}x_{2}^{j-i}$$
	with
	$$|c_{i,j}|\leq c_{1}\ln(N)^{c_{2}}$$
	for $N\geq 2$, $i\geq 1$.
	Then we have that
	$$\|f(x)\sin(Ng_{N}(x))\|_{C^{j}}\leq C_{K,c_{1},j}\sum_{i=0}^{j}\|f(x)\|_{C^{i}}N^{j}\ln(N)^{C_{c_{2},j}}.$$
\end{lemma}
\begin{proof}
	To simplify notation, we will show 
	$$\|\p^{j}_{x_{1}}(f(x)\sin(Ng_{N}(x)))\|_{L^{\infty}}\leq C_{K,c_{1},j}\sum_{i=0}^{j}\|f(x)\|_{C^{i}}N^{j}\ln(N_{p})^{C_{c_{2},j}},$$
	but note that the exact same proof works with a generic derivative of order $j$ $D^{j}$.
	By applying the Leibniz rule we have
	\begin{align*}
		&|\p^{j}_{x_{1}}f(x)\sin(Ng_{N}(x))|\leq C_{j}|\sum_{i=0}^{j}(\p^{i}_{x_{1}}f(x))(\p_{x_{1}}^{j-i}\sin(Ng_{N}(x)))|\\
		&\leq C_{j}\sum_{i=0}^{j}\|f\|_{C^{i}}\|\p_{x_{1}}^{j-i}\sin(Ng_{N}(x))\|_{L^{\infty}(B_{1}(0))}\leq  C_{K,c_{1},j}\sum_{i=0}^{j}\|f\|_{C^{i}}N^{j-i}\ln(N)^{c_{2}j}
	\end{align*}
	as we wanted to prove.
\end{proof}

\begin{lemma}\label{auxuk}
	Given $\beta\in (0,\frac{1}{2}),(a_{i})_{i\in\mathds{N}},(b_{i})_{i\in\mathds{N}}$, $c,d>0$ $K\in\mathds{N}$, for $N>2$ if we have a smooth function $f_{N}(x)$ fulfilling $\text{supp}(f_{N})\subset B_{N^{-\beta}\ln(N)^{4}}(0)\subset B_{1}(0)$,
	$$\|f_{N}\|_{C^{j}}\leq a_{j}N^{j\beta}\ln(N)^{b_{j}}$$
	and $g_{N}(x)$ fulfilling
	$$g_{N}(x)=\sum_{j=2}^{K}\sum_{i=0}^{j}c_{i,j}x_{1}^{i}x_{2}^{j-i}$$
	with $|c_{i,j}|\leq c\ln(N)^{d}$, then we have that, for $k,i\in\mathds{N}$, $k\geq i$, there exists $g^{1}_{k,i},g^{2}_{k,i}\in C^{\infty}$ such that $\text{supp}(g^{1}_{k,i}),\text{supp}(g^{2}_{k,i})\subset B_{N^{-\beta}\ln(N)^{4}}(0)$ and
	$$\frac{\partial^{k}}{\partial x_{1}^{k-i}\partial x^{i}_{2}}f_{N}(x)\sin(Ng_{N}(x)+\theta_{0})=g^{1}_{k,i}(x)\sin(Ng_{N}(x)+\theta_{0})+g^{2}_{k,i}(x)\cos(Ng_{N}(x)+\theta_{0})$$
	and there exists $(\bar{a}^{k}_{j})_{j\in\mathds{N}}$, $(\bar{b}^{k}_{j})_{j\in\mathds{N}}$ (depending on  $\beta,(a_{i})_{i\in\mathds{N}},(b_{i})_{i\in\mathds{N}}$, $c,d,k$ and $K$) such that, 
	$$\|g^{1}_{k,i}\|_{C^{j}},\|g^{2}_{k,i}\|_{C^{j}}\leq \bar{a}^{k}_{j}N^{k(1-\beta)+j\beta}\ln(N)^{\bar{b}^{k}_{j}}.$$
\end{lemma}

\begin{proof}
	We will show the result by induction. Note that the result holds trivially true for $k=0$ with
	$$g^{1}_{0,0}=f_{N},\ \  g^{2}_{0,0}=0.$$
	 We will prove the result, for simplicity, in the case $i=0$ , but note that the proof is exactly the same (but with a more cumbersome notation) for generic $i$. If the result holds true for $k=k_{0}$, we have that
	\begin{align*}
		&\frac{\partial^{k_{0}+1}}{\partial x_{1}^{k_{0}+1}}f_{N}(x)\sin(Ng_{N}(x)+\theta_{0})=\frac{\partial}{\partial x_{1}}[g^{1}_{k_{0},0}(x)\sin(Ng_{N}(x)+\theta_{0})+g^{2}_{k_{0},0}(x)\cos(Ng_{N}(x)+\theta_{0})]\\
		&=(\frac{\partial}{\partial x_{1}}g^{1}_{k_{0},0}(x))\sin(Ng_{N}(x)+\theta_{0})-N(\frac{\partial}{\partial x_{1}}g_{N}(x))g^{2}_{k_{0},0}(x)\sin(Ng_{N}(x)+\theta_{0})\\
		&+N(\frac{\partial}{\partial x_{1}}g_{N}(x))g^{1}_{k_{0},0}(x)\cos(Ng_{N}(x)+\theta_{0})+(\frac{\partial}{\partial x_{1}}g^{2}_{k_{0},0}(x))\cos(Ng_{N}(x)+\theta_{0})
	\end{align*}
and thus we can define
$$g^{1}_{k_{0}+1,0}(x)=\frac{\partial}{\partial x_{1}}g^{1}_{k_{0},0}(x)-(N\frac{\partial }{\partial x_{1}}g_{N}(x))g^{2}_{k,0}(x),$$
$$g^{2}_{k_{0}+1,0}(x)=\frac{\partial}{\partial x_{1}}g^{2}_{k_{0},0}(x)+(N\frac{\partial }{\partial x_{1}}g_{N}(x))g^{1}_{k,0}(x),$$
which immediately fulfill the bounds for the support, and thus we only need to show the bounds for the norms of $g^{1}_{k+1,0}(x),g^{2}_{k+1,0}(x)$. Let us focus on $g^{1}_{k+1,0}(x)$, the other case being analogous. First, we note that, by hypothesis
$$\|\frac{\partial}{\partial x_{1}}g^{1}_{k_{0},0}(x)\|_{C^{j}}\leq \bar{a}^{k_{0}}_{j}N^{k_{0}(1-\beta)+(j+1)\beta}\ln(N)^{\bar{b}^{k}_{j}}\leq\bar{a}^{k_{0}}_{j}N^{(k_{0}+1)(1-\beta)+j\beta}\ln(N)^{\bar{b}^{k}_{j}}$$
since $\beta\leq 1-\beta$, $N>1$, so we only need to worry about the norms of $(N\frac{\partial }{\partial x_{1}}g_{N}(x))g^{2}_{k,0}(x)$.
But we have that, since $\frac{\partial }{\partial x_{1}}g_{N}(x=0)=0$, using the bounds for the support of $f_{N}$ and the bounds for $g_{N}$
$$\|\frac{\partial}{\partial x_{1}}g_{N}\|_{L^{\infty}(\text{supp}(f_{N}))}\leq C\ln(N)^4N^{-\beta}\ln(N)^{d}$$
$$\|\frac{\partial}{\partial x_{1}}g_{N}\|_{C^{j}(\text{supp}(f_{N}))}\leq C_{j}c\ln(N)^{d}$$
so
\begin{align*}
	&\|(N\frac{\partial }{\partial x_{1}}g_{N}(x))g^{2}_{k_{0},0}(x)\|_{C^{j}}=\|(N\frac{\partial }{\partial x_{1}}g_{N}(x))g^{2}_{k_{0},0}(x)\|_{C^{j}(\text{supp}(f_{N}))}\\
	&\leq N\sum_{i=0}^{j}\|g^{2}_{k_{0},0}\|_{C^{j-i}}\| \p_{x_{1}}g_{N}\|_{C^{i}(\text{supp}(f_{N}))}\leq N^{1-\beta}\sum_{i=0}^{j}C_{i}\|g^{2}_{k_{0},0}\|_{C^{j-i}}\ln(N)^{d}N^{i \beta}\\
	&\leq C_{j}N^{1-\beta}\sum_{i=0}^{j}\bar{a}^{k_{0}}_{j}N^{k_{0}(1-\beta)+(j-i)\beta}\ln(N)^{\bar{b}^{k_{0}}_{j-i}}\ln(N)^{d}N^{i \beta}\leq C_{j}N^{(k_{0}+1)(1-\beta)+j\beta}\ln^{d+c_{k_{0}}}
\end{align*}
as we wanted to prove.

\end{proof}

\begin{lemma}\label{boundsukgen}
	Given a $p$-layered solution and $k,L\in\mathds{N}$,
	  $X(x,t)$ as in Definition \ref{soloperator}, $\frac14>\beta_{p+1}>0$, $\lambda>0$ with $N_{p+1}$ given by Remark \ref{remarknp1}, $(a_{i})_{i\in\mathds{N}}, (b_{i})_{i\in\mathds{N}}$  and $f_{N}(x)\in C^{\infty}$ with $\text{supp}(f_{N_{p+1}}(x))\subset B_{N_{p+1}^{-2\beta_{p+1}}[\ln(N_{p+1})]^4}$ and fulfilling
	  	$$\|f_{N_{p+1}}\|_{C^{j}}\leq a_{j}N_{p+1}^{2j\beta_{p+1}}\ln(N_{p+1})^{b_{j}}$$
	  if we define, for $L\geq l\geq 1$
	\begin{align*}
		&u^{k}(\lambda f_{N_{p+1}}(x)\sin(N_{p+1}lX(x,t))+\theta_{0}):=\\
		&u^{k}(\lambda f_{N_{p+1}}(x)\sin(N_{p+1}l\P_{1}(x,t)+\theta_{0})\cos(N_{p+1}lX(x,t)-N_{p+1}l\P_{1}X(x,t)))\\
		&+u^{k}(\lambda f_{N_{p+1}}(x)\cos(N_{p+1}l\P_{1}(x,t)+\theta_{0})\sin(N_{p+1}lX(x,t)-N_{p+1}l\P_{1}X(x,t)))
	\end{align*}
then we have that, for $N_{p}$ big enough (depending only on $\beta_{p}$ and $\beta_{p+1}$), there exists $\bar{g}_{1,i,j} (x),\bar{g}_{2,i,j}(x)$ (which depend on $N_{p},f_{N_{p}}$ and $l$, but not on $k$)for $i=0,1,...,k$, $j=1,2$ such that
\begin{align}\label{decompuk}
	&u_{j}^{k}(\lambda f_{N_{p+1}}(x)\sin(N_{p+1}lX(x,t)+\theta_{0})\\
	&=\sum_{i=0}^{k}\lambda \bar{g}_{1,i,j}(x)\sin(N_{p+1}lX(x,t)+\theta_{0})+\sum_{i=1}^{k}\lambda\nonumber \bar{g}_{2,i,j}(x)\cos(N_{p+1}lX(x,t)+\theta_{0})
\end{align}
 with $\text{supp }\bar{g}_{1,i,j},\bar{g}_{2,i,j}\subset B_{N_{p+1}^{-2\beta_{p+1}}[\ln(N_{p+1})]^4}$ and there exists $(\bar{a}^{i}_{m})_{m\in\mathds{N}}$, $(\bar{b}^{i}_{m})_{m\in\mathds{N}}$ (depending on  $\beta_{p},\beta_{p+1},(a_{j})_{j\in\mathds{N}},$ $(b_{j})_{j\in\mathds{N}}$, $k,i$ and $L$) such that, for $i=0,1,...,k$, $j=1,2$, 
 \begin{align}\label{decompukbd}
 	\|\bar{g}_{1,i,j}\|_{C^{m}},\|\bar{g}_{2,i,j}\|_{C^{m}}\leq \bar{a}^{i}_{m}N_{p+1}^{2(m-i)\beta_{p+1}}\ln(N_{p+1})^{\bar{b}^{i}_{m}}.
 \end{align}
Furthermore, there exists $(\tilde{a}^{k}_{m})_{j\in\mathds{N}}$, $(\tilde{b}^{k}_{m})_{j\in\mathds{N}}$ (depending on  $\beta_{p},\beta_{p+1},$ $(a_{i})_{i\in\mathds{N}},(b_{i})_{i\in\mathds{N}}$, $k$ and $L$) such that
\begin{align}\label{cotaukx1}
	&\|u(\lambda f_{N_{p+1}}(x)\sin(N_{p+1}lX(x,t))+\theta_{0})-u^{k}(\lambda f_{N_{p+1}}(x)\sin(N_{p+1}lX(x,t))+\theta_{0})\|_{C^{m}}\\
	&\leq \lambda \tilde{a}^{k}_{m}N_{p+1}^{2(m+2)-k\beta_{p+1}}\ln(N)^{\tilde{b}^{k}_{m}}.\nonumber
\end{align}

\end{lemma}

\begin{proof}
	We start by noting that $N_{p+1}$ is well defined using Remark \ref{remarknp1} as long as $N_{p}$ is big enough. We also take $N_{p}$ big enough to apply the lemmas proved previously. We will prove all the properties for $\lambda=1$, noting that the proof is the same for generic $\lambda$ since everything is linear in $\lambda$.
	We now note that, by the definition of $u^{k}$ in Corollary \ref{defuk} we have that, considering for now the second component of the velocity $u^{k}_{2}$ (the other one being analogous)
	\begin{align}\label{ukx1}
		&u_{2}^{k}(f_{N_{p+1}}(x)\sin(N_{p+1}lX(x,t)+\theta_{0})):=\\
		&u_{2}^{k}(f_{N_{p+1}}(x)\sin(N_{p+1}l\P_{1}(x,t)+\theta_{0})\cos(N_{p+1}lX(x,t)-N_{p+1}l\P_{1}X(x,t)))\nonumber\\
		&+u_{2}^{k}(f_{N_{p+1}}(x)\cos(N_{p+1}l\P_{1}(x,t)+\theta_{0})\sin(N_{p+1}lX(x,t)-N_{p+1}l\P_{1}X(x,t))).\nonumber
	\end{align}
	We will from now on take $\theta_{0}=0$ for simplicity, but the proof is exactly the same for generic $\theta_{0}$.
	Now, using the expression for $V^{k},V^{k-1}$ and the definition for $u^{k}$ in Corollary \ref{defuk} we can decompose 
	$$u_{2}^{k}(f_{N}(x)\sin(N_{p+1}lX(x,t)))=\sum_{i=0}^{K+1}\tilde{u}_{2}^{i}$$
	with
	\begin{align*}
		\tilde{u}_{2}^{i}:=
		&\sum_{j=0}^{i}\frac{\p_{x_1}^{i-j}\p_{x_2}^{j}[f_{N_{p+1}}(x)\sin(N_{p+1}lX(x,t)-N_{p+1}l\P_{1}X(x,t)+\frac{\pi}{2})]}{N_{p+1}^{i+1}}\\
		&\times\p_{x_{1}}(c^{s}_{i,j}\sin(N_{p+1}l\P_{1}X(x,t))+c^{c}_{i,j}\cos(N_{p+1}l\P_{1}X(x,t)))\\
		+&\sum_{j=0}^{i}\frac{\p_{x_1}^{i-j}\p_{x_2}^{j}[f_{N_{p+1}}(x)\sin(N_{p+1}lX(x,t)-N_{p+1}l\P_{1}X(x,t))]}{N_{p+1}^{i+1}}\\
		&\times\p_{x_{1}}(c^{s}_{i,j}\sin(N_{p+1}l\P_{1}X(x,t)+\frac{\pi}{2})+c^{c}_{i,j}\cos(N_{p+1}l\P_{1}X(x,t)+\frac{\pi}{2}))\\
		+&\sum_{j=0}^{i-1}\frac{\p_{x_1}^{i-j+1}\p_{x_2}^{j}[f_{N_{p+1}}(x)\sin(N_{p+1}lX(x,t)-N_{p+1}l\P_{1}X(x,t)+\frac{\pi}{2})]}{N_{p+1}^{i+1}}\\
		&\times(c^{s}_{i,j}\sin(N_{p+1}l\P_{1}X(x,t))+c^{c}_{i,j}\cos(N_{p+1}l\P_{1}X(x,t)))\\
		+&\sum_{j=0}^{i-1}\frac{\p_{x_1}^{i-j+1}\p_{x_2}^{j}[f_{N_{p+1}}(x)\sin(N_{p+1}lX(x,t)-N_{p+1}l\P_{1}X(x,t))]}{N_{p+1}^{i+1}}\\
		&\times(c^{s}_{i,j}\sin(N_{p+1}l\P_{1}X(x,t)+\frac{\pi}{2})+c^{c}_{i,j}\cos(N_{p+1}l\P_{1}X(x,t)+\frac{\pi}{2})).
	\end{align*}
	We will show that
	$$\tilde{u}^{i}_{2}=\bar{g}_{1,i,2}(x)\sin(N_{p+1}lX(x,t))+\bar{g}_{2,i,2}(x)\cos(N_{p+1}lX(x,t))$$
	with the desired bounds, which is equivalent to our lemma.
	Focusing first on the addends in the first two lines, we now note that, using lemma \ref{auxuk}  with $\beta=2\beta_{p+1}$, the bounds for $\|f_{N_{p+1}}\|_{C^{m}}$ and the bounds for $X(x,t)$ given by \ref{cotasx1np1}, we have
	\begin{align*}
		&\sum_{j=0}^{i}\p_{x_1}^{i-j}\p_{x_2}^{j}[f_{N_{p+1}}(x)\sin(N_{p+1}lX(x,t)-N_{p+1}l\P_{1}X(x,t)+C)]\\
	&=g_{1,i}(x)\sin(N_{p+1}lX(x,t)-N_{p+1}l\P_{1}X(x,t)+C)+g_{2,i}(x)\sin(N_{p+1}lX(x,t)-N_{p+1}l\P_{1}X(x,t)+C)
	\end{align*}
	with $g_{1,i},g_{2,i}$ having the desired bounds for the support and there are sequences $(\bar{a}_{j})_{j\in\mathds{N}},(\bar{b}_{j})_{j\in\mathds{N}}$ (depending on $\beta_{p},\beta_{p+1},i,L$, $(a_{j})_{j\in\mathds{N}},$ and $(b_{j})_{j\in\mathds{N}}$) such that
		$$\|g_{1,i}\|_{C^{j}},\|g_{2,i}\|_{C^{j}}\leq \bar{a}^{i}_{j}N_{p+1}^{i(1-2\beta_{p+1})+j2\beta_{p+1}}\ln(N)^{\bar{b}^{i}_{j}}.$$
			Using this, we can then write
			\begin{align*}
				&\sum_{j=0}^{i}\p_{x_1}^{i-j}\p_{x_2}^{j}[f_{N_{p+1}}(x)\sin(N_{p+1}lX(x,t)-N_{p+1}l\P_{1}X(x,t)+\frac{\pi}{2})]\\
				&\times\p_{x_{1}}(c^{s}_{i,j}\sin(N_{p+1}l\P_{1}X(x,t))+c^{c}_{i,j}\cos(N_{p+1}l\P_{1}X(x,t)))\\
				+&\sum_{j=0}^{i}\p_{x_1}^{i-j}\p_{x_2}^{j}[f_{N_{p+1}}(x)\sin(N_{p+1}lX(x,t)-N_{p+1}l\P_{1}X(x,t))]\\
				&\times\p_{x_{1}}(c^{s}_{i,j}\sin(N_{p+1}l\P_{1}X(x,t)+\frac{\pi}{2})+c^{c}_{i,j}\cos(N_{p+1}l\P_{1}X(x,t)+\frac{\pi}{2}))\\
				&=[g_{1,i}(x)\sin(N_{p+1}lX(x,t)-N_{p+1}l\P_{1}X(x,t)+\frac{\pi}{2})+g_{2,i}(x)\cos(N_{p+1}lX(x,t)-N_{p+1}l\P_{1}X(x,t)+\frac{\pi}{2})]\\
				&\times \p_{x_{1}}(c^{s}_{i,j}\sin(N_{p+1}l\P_{1}X(x,t))+c^{c}_{i,j}\cos(N_{p+1}l\P_{1}X(x,t)))\\
				&+[g_{1,i}(x)\sin(N_{p+1}lX(x,t)-N_{p+1}l\P_{1}X(x,t))+g_{2,i}(x)\cos(N_{p+1}lX(x,t)-N_{p+1}l\P_{1}X(x,t))]\\
				&\times \p_{x_{1}}(c^{s}_{i,j}\sin(N_{p+1}l\P_{1}X(x,t)+\frac{\pi}{2})+c^{c}_{i,j}\cos(N_{p+1}l\P_{1}X(x,t)+\frac{\pi}{2}))
				\end{align*}
				\begin{align*}
				&=g_{1,i}(x)c^{s}_{i,j}Nl[(\p_{x_{1}}X)(0,t)]\sin(N_{p+1}lX(x,t)-N_{p+1}l\P_{1}X(x,t)+\frac{\pi}{2})\cos(N_{p+1}l\P_{1}X(x,t))\\
				&+g_{1,i}(x)c^{s}_{i,j}Nl[(\p_{x_{1}}X)(0,t)]\sin(N_{p+1}lX(x,t)-N_{p+1}l\P_{1}X(x,t))\cos(N_{p+1}l\P_{1}X(x,t)+\frac{\pi}{2})\\
				&+g_{2,i}(x)c^{s}_{i,j}Nl[(\p_{x_{1}}X)(0,t)]\cos(N_{p+1}lX(x,t)-N_{p+1}l\P_{1}X(x,t)+\frac{\pi}{2})\cos(N_{p+1}l\P_{1}X(x,t))\\
				&+g_{2,i}(x)c^{s}_{i,j}Nl[(\p_{x_{1}}X)(0,t)]\cos(N_{p+1}lX(x,t)-N_{p+1}l\P_{1}X(x,t))\cos(N_{p+1}l\P_{1}X(x,t)+\frac{\pi}{2})\\
				&-g_{1,i}(x)c^{c}_{i,j}Nl[(\p_{x_{1}}X)(0,t)]\sin(N_{p+1}lX(x,t)-N_{p+1}l\P_{1}X(x,t)+\frac{\pi}{2})\sin(N_{p+1}l\P_{1}X(x,t))\\
				&-g_{1,i}(x)c^{c}_{i,j}Nl[(\p_{x_{1}}X)(0,t)]\sin(N_{p+1}lX(x,t)-N_{p+1}l\P_{1}X(x,t))\sin(N_{p+1}l\P_{1}X(x,t)+\frac{\pi}{2})\\
				&-g_{2,i}(x)c^{c}_{i,j}Nl[(\p_{x_{1}}X)(0,t)]\cos(N_{p+1}lX(x,t)-N_{p+1}l\P_{1}X(x,t)+\frac{\pi}{2})\sin(N_{p+1}l\P_{1}X(x,t))\\
				&-g_{2,i}(x)c^{c}_{i,j}Nl[(\p_{x_{1}}X)(0,t)]\cos(N_{p+1}lX(x,t)-N_{p+1}l\P_{1}X(x,t))\sin(N_{p+1}l\P_{1}X(x,t)+\frac{\pi}{2})
			\end{align*}
			\begin{align}\label{g1ig2i}
				&=g_{1,i}(x)c^{s}_{i,j}Nl[(\p_{x_{1}}X)(0,t)]\cos(NlX(x,t))-g_{2,i}(x)c^{s}_{i,j}Nl[(\p_{x_{1}}X)(0,t)]\sin(NlX(x,t))\\
				&-g_{1,i}(x)c^{c}_{i,j}Nl[(\p_{x_{1}}X)(0,t)]\sin(NlX(x,t))-g_{2,i}(x)c^{c}_{i,j}Nl[(\p_{x_{1}}X)(0,t)]\cos(NlX(x,t))\nonumber
			\end{align}
			but now, using the properties of $g_{1,i}$ and of $X(x,t)$ we get
			\begin{align*}
				&\|\frac{g_{1,i}(x)c^{s}_{i,j}N_{p+1}l[(\p_{x_{1}}X)(0,t)]}{N_{p+1}^{i+1}}\|_{C^{m}}\leq C_{m} \sum_{l=0}^{m}\|g_{1,i}(x)\|_{C^{l}}\|\frac{c^{s}_{i,j}N_{p+1}lX(x,t)}{N_{p+1}^{i+1}}\|_{C^{m-l+1}(\text{supp}g_{1,i})}\\
				&\leq\frac{C_{L,m,i,\beta_{p},\beta_{p+1}}}{N_{p+1}^{i}} \sum_{l=0}^{m}\bar{a}^{i}_{l}N_{p+1}^{i(1-2\beta_{p+1})+l2\beta_{p+`}}\ln(N_{p+1})^{\bar{b}^{i}_{l}}\ln(N_{p+1})^{C_{\beta_{p},\beta_{p+1}}}\\
				&\leq  C N_{p+1}^{(m-i)2\beta_{p+1}}\ln(N_{p+1})^{C}
			\end{align*}
			where in the last line we omit all the sub-indexes, but the constants depend on  $\beta_{p},\beta_{p+1},(a_{i})_{i\in\mathds{N}},(b_{i})_{i\in\mathds{N}}$, $i,m,L$ and $k$.  The exact same bounds can be obtained for the rest of the terms in \eqref{g1ig2i}, giving the desired bounds for $\tilde{u}^{i}$.
			
			To obtain \eqref{cotaukx1}, using \eqref{ukx1} we can apply Corollary \ref{defuk} with 
			$$\rho^{1}(x)=f^{1}(x)\cos(N_{p+1}l\P_{1}X(x,t));f^{1}(x)=f_{N_{p+1}}(x)\sin(N_{p+1}lX(x,t)-N_{p+1}l\P_{1}X(x,t))$$
			and with 
			$$\rho^{2}(x)=f^{2}(x)\sin(N_{p+1}l\P_{1}X(x,t));f^{2}(x)=f_{N_{p+1}}(x)\cos(N_{p+1}lX(x,t)-N_{p+1}l\P_{1}X(x,t))$$
			to obtain
			\begin{align}\label{equ2x1}
				&\|u_{2}(f_{N_{p+1}}(x)\sin(N_{p+1}lX(x,t)+\theta_{0}))-u^{k}_{2}(f_{N_{p+1}}(x)\sin(N_{p+1}lX(x,t)+\theta_{0}))\|_{C^{J}}\nonumber\\
				&\leq C_{k,\eps,J}(\sum_{i=0}^{k} \tilde{N}^{i\eps+J+1-k}(\|f^{1}\|_{C^{k-i+J+1}}+\|f^{2}\|_{C^{k-i+J+1}}),\\
                &+\tilde{N}^{J+1-(k+2)\eps} \|f_{1}\|_{C^{k+J+2}}+\|f_{2}\|_{C^{k+J+2}})\nonumber
			\end{align}
			with 
			$$\tilde{N}=N_{p+1}l[(\p_{x_{1}}X(x=0,t))^2+(\p_{x_{2}}X(x=0,t))^2]^{\frac{1}{2}}$$
			which fulfils
			\begin{equation}\label{ntilde}
				N_{p+1}\ln(N_{p+1})^{-C_{\beta_{p},\beta_{p+1}}}\leq \tilde{N}\leq N_{p+1}L\ln(N_{p+1})^{C_{\beta_{p},\beta_{p+1}}}
			\end{equation}
                by Lemma \ref{p1phi} combined with
                $$[(\p_{x_{1}}X(x=0,t))^2+(\p_{x_{2}}X(x=0,t))^2]^{\frac{1}{2}}\leq [(\p_{x_{1}}X(x=0,1))^2+(\p_{x_{2}}X(x=0,1))^2]^{\frac{1}{2}}\sum_{i,j=1,2}|\p_{x_{j}}\phi_{i}(x=0,t,1)|,$$
                $$[(\p_{x_{1}}X(x=0,1))^2+(\p_{x_{2}}X(x=0,1))^2]^{\frac{1}{2}}\leq [(\p_{x_{1}}X(x=0,t))^2+(\p_{x_{2}}X(x=0,t))^2]^{\frac{1}{2}}\sum_{i,j=1,2}|\p_{x_{j}}\phi_{i}(x=0,1,t)|.$$

			But then, we can apply Lemma \ref{auxuk} with $j=0$, $\rho=f^{1},f^{2}$, $f_{N}=f_{N_{p+1}}(x)$, $\beta=2\beta_{p+1}$ to get
			$$\|f^{1}\|_{C^{m}},\|f^{2}\|_{C^{m}}\leq C_{m,\beta_{p},\beta_{p+1},L}N_{p+1}^{m(1-2\beta_{p+1})}\ln(N)^{C_{m,\beta_{p},\beta_{p+1},L}},$$
			and thus, combining this with \eqref{equ2x1} and \eqref{ntilde} we get, by taking $\epsilon=1-\beta_{p+1}$ (omitting the sub-indexes of the constants, that depend on $m,\beta_{p},\beta_{p+1},L$ and $k$)
\begin{align*}
	&\|u_{2}(f_{N_{p+1}}(x)\sin(N_{p+1}lX(x,t)+\theta_{0}))-u^{K}_{2}(f_{N_{p+1}}(x)\sin(N_{p+1}lX(x,t)+\theta_{0}))\|_{C^{J}}\nonumber\\
	&\leq C \ln(N)^{C}(\sum_{i=0}^{k}N_{p+1}^{i(1-\beta_{p+1})+J+1-k}N_{p+1}^{(1-2\beta_{p+1})(k-i+J+1)}+N^{J+1-(k+2)(1-\beta_{p+1})+(1-2\beta_{p+1})k+J+2})\\
 &\leq C \ln(N)^{C}N_{p+1}^{2(J+2)-k\beta_{p+1}}.
\end{align*}

\end{proof}

\begin{lemma}\label{cuadpeque}
	Given a $p$-layered solution and $k\in\mathds{N}\cup 0$,
	$X(x,t)$ as in Definition \ref{soloperator}, $\frac14>\beta_{p+1}>0$ and $N_{p+1}$ given by Remark \ref{remarknp1}, $(a_{i})_{i\in\mathds{N}}, (b_{i})_{i\in\mathds{N}}$  and $f^{1}_{N}(x),f^{2}_{N}(x)\in C^{\infty}$ fulfilling, for $i=1,2$ $\text{supp}(f^{i}_{N_{p+1}}(x))\subset B_{N_{p+1}^{-2\beta_{p+1}}[\ln(N_{p+1})]^4}$ and 
	$$\|f^{i}_{N_{p+1}}\|_{C^{m}}\leq a_{m}N_{p+1}^{2m\beta_{p+1}}\ln(N_{p+1})^{b_{m}}$$
	then for $l_{2},l_{1}\geq 1$, $\lambda_{1},\lambda_{2}>0$, $\theta_{0}^{1},\theta_{0}^{2}\in\mathds{R}$ we have that
	\begin{align*}
		&u^{k}(\lambda_{1}f^{1}_{N_{p+1}}(x)\sin(N_{p+1}l_{1}X(x,t)+\theta^{1}_{0}))\cdot\nabla (\lambda_{2}f^{2}_{N_{p+1}}(x)\sin(N_{p+1}l_{2}X(x,t)+\theta^{2}_{0}))\\
		&=\lambda_{1}\lambda_{2}\big(\sin(N_{p+1}l_{1}X(x,t)+\theta^{1}_{0})[a_{1,1}(x)\sin(N_{p+1}l_{2}X(x,t))+a_{1,2}(x)\cos(N_{p+1}l_{2}X(x,t)+\theta^{2}_{0})]\\
		&+\cos(N_{p+1}l_{1}X(x,t)+\theta^{1}_{0})[a_{2,1}(x)\sin(N_{p+1}l_{2}X(x,t)+\theta^{2}_{0})+a_{2,2}(x)\cos(N_{p+1}l_{2}X(x,t)+\theta^{2}_{0})]\big)
	\end{align*}
	 and there exists $(\bar{a}_{m})_{m\in\mathds{N}}$, $(\bar{b}_{m})_{m\in\mathds{N}}$ (depending on  $\beta_{p},\beta_{p+1},(a_{j})_{j\in\mathds{N}},(b_{j})_{j\in\mathds{N}}$ and $k$)
	with, for $i,j=1,2$
	$$\|a_{i,j}\|_{C^{m}}\leq \bar{a}_{m}N_{p+1}^{2(m-1)\beta_{p+1}+1}\ln(N_{p+1})^{\bar{b}_{m}}.$$

\end{lemma}
\begin{proof}
	We will consider the case $\lambda_{1}=\lambda_{2}=1$ to simplify the notation, but note that the general result then follows since all the operations involved are linear with respect to $\lambda_{1},\lambda_{2}$. Similarly, to simplify the notation we will take $\theta^{1}_{0}=\theta^{2}_{0}=0$, but again, note that this does not significantly alter the proof. To show this, we will decompose $u^{k}=(u^{k}-u^{0})+u^{0}$ and we will show the desired bounds for each of the two terms obtained.
	 For this, if we define $\alpha_{p+1}(t)$ such that
	 \begin{align}\label{alphap1u0}
	 	\cos(\alpha_{p+1}(t))=\frac{\p_{x_{1}}X(x=0,t)}{[(\p_{x_{1}}X(x=0,t))^2+\p_{x_{1}}X(x=0,t)]^2]^{\frac12}},\nonumber\\ \sin(\alpha_{p+1}(t))=\frac{\p_{x_{2}}X(x=0,t)}{[(\p_{x_{1}}X(x=0,t))^2+\p_{x_{1}}X(x=0,t)]^2]^{\frac12}}
	 \end{align}

	then we have that
	\begin{align*}&u^{0}(f^{1}_{N_{p+1}}(x)\sin(N_{p+1}l_{1}X(x,t)))\\&=C_{0}\cos(\alpha_{p+1}(t))f^{1}_{N_{p+1}}(x)\sin(N_{p+1}l_{1}X(x,t))(\sin(\alpha_{p+1}(t)),-\cos(\alpha_{p+1}(t)))\end{align*}
	so
	\begin{align*}
		&u^{0}(f^{1}_{N_{p+1}}(x)\sin(N_{p+1}l_{1}X(x,t)))\cdot\nabla (f^{2}_{N_{p+1}}(x)\sin(N_{p+1}l_{2}X(x,t)))\\
		&=C_{0}f^{1}_{N_{p+1}}(x)\cos(\alpha_{p+1}(t))\sin(N_{p+1}l_{1}X(x,t))N_{p+1}l_{2}\cos(N_{p+1}l_{2}X(x,t))\\
		&\times f^{2}_{N_{p+1}}(x)[\sin(\alpha_{p+1}(t))\p_{x_{1}}X(x,t)-\cos(\alpha_{p+1}(t))\p_{x_{2}}X(x,t)]\\
		&+C_{0}f^{1}_{N_{p+1}}(x)\cos(\alpha_{p+1}(t))\sin(N_{p+1}l_{1}X(x,t))\sin(N_{p+1}l_{2}X(x,t))\\
		&\times[\sin(\alpha_{p+1}(t))\p_{x_{1}}f^{2}_{N_{p+1}}(x)-\cos(\alpha_{p+1}(t))\p_{x_{2}}f^{2}_{N_{p+1}}(x)]
	\end{align*}
	but
	\begin{align*}
		&\|C_{0}f^{1}_{N_{p+1}}(x)\cos(\alpha_{p+1}(t))[\sin(\alpha_{p+1}(t))\p_{x_{1}}f^{2}_{N_{p+1}}(x)-\cos(\alpha_{p+1}(t))\p_{x_{2}}f^{2}_{N_{p+1}}(x)]\|_{C^{m}}\\
		&\leq C_{m}\sum_{i=0}^{m}\|f^{1}_{N_{p+1}}(x)\|_{C^{i}} \|f^{2}_{N_{p+1}}(x)\|_{C^{m-i+1}}\\
		&\leq C \ln(N_{p+1})^{C}N_{p+1}^{2(m+1)\beta_{p+1}}\leq C \ln(N_{p+1})^{C}N_{p+1}^{2(m-1)\beta_{p+1}+1}
	\end{align*}
where we used $\beta_{p+1}\leq \frac14$ and the constants in the last line depend on depend on  $\beta_{p},\beta_{p+1},$ $(a_{j})_{j\in\mathds{N}}$ and $(b_{j})_{j\in\mathds{N}}$.
On the other hand, by the definition of $\alpha_{p+1}(t)$ \eqref{alphap1u0} we have that
$$[\sin(\alpha_{p+1}(t))\p_{x_{1}}X(x=0,t)-\cos(\alpha_{p+1}(t))\p_{x_{2}}X(x=0,t)]=0$$
so
\begin{align*}
	&=\sin(\alpha_{p+1}(t))\p_{x_{1}}X(x,t)-\cos(\alpha_{p+1}(t))\p_{x_{2}}X(x,t)\\
	&= \sin(\alpha_{p+1}(t))(\p_{x_{1}}X(x,t)-\p_{x_{1}}X(x=0,t))-\cos(\alpha_{p+1}(t))(\p_{x_{2}}X(x,t)-\p_{x_{2}}X(x=0,t))
\end{align*}
and since we have
\begin{align*}
	&\|\sin(\alpha_{p+1}(t))(\p_{x_{1}}X(x,t)-\p_{x_{1}}X(x=0,t))-\cos(\alpha_{p+1}(t))(\p_{x_{2}}X(x,t)-\p_{x_{2}}X(x=0,t))\|_{L^{\infty}(B_{N_{p+1}^{-2\beta_{p+1}}[\ln(N_{p+1})]^4})}\\
	&\leq CN_{p+1}^{-2\beta_{p+1}}\ln(N)^{C}
\end{align*}
\begin{align*}
	&\|\sin(\alpha_{p+1}(t))(\p_{x_{1}}X(x,t)-\p_{x_{1}}X(x=0,t))-\cos(\alpha_{p+1}(t))(\p_{x_{2}}X(x,t)-\p_{x_{2}}X(x=0,t))\|_{C^{m}(B_{N_{p+1}^{-2\beta_{p+1}}[\ln(N_{p+1})]^4})}\\
	&\leq C\ln(N)^{C}
\end{align*}
where the constants depend on  $\beta_{p},\beta_{p+1},(a_{j})_{j\in\mathds{N}},(b_{j})_{j\in\mathds{N}}$ and $m$.
We can then use this to prove
\begin{align*}
	&\|C_{0}f^{1}_{N_{p+1}}(x)\cos(\alpha_{p+1}(t))N_{p+1}l_{2}f^{2}_{N_{p+1}}(x)[\sin(\alpha_{p+1}(t))\p_{x_{1}}X(x,t)-\cos(\alpha_{p+1}(t))\p_{x_{2}}X(x,t)]\|_{C^{m}}\\
	&\leq C_{m}\sum_{i=0}^{m}\|f^{1}_{N_{p+1}}(x)f^{2}_{N_{p+1}}(x)\|_{C^{i}}\|\sin(\alpha_{p+1}(t))\p_{x_{1}}X(x,t)-\cos(\alpha_{p+1}(t))\p_{x_{2}}X(x,t)\|_{C^{m-i}}\\
	&\leq C \ln(N_{p+1})^{C}N_{p+1}^{2(m-1)\beta_{p+1}+1}
\end{align*}
where the constants on depend on  $\beta_{p},\beta_{p+1},(a_{j})_{j\in\mathds{N}},(b_{j})_{j\in\mathds{N}}$ and $m$. Therefore we only need to show the analogous bound with $u^{k}-u^{0}$ to finish the proof. But, by Lemma \ref{boundsukgen}
\begin{align*}
	&u^{k}(f^{1}_{N_{p+1}}(x)\sin(N_{p+1}l_{1}X(x,t)))-u^{0}(f^{1}_{N_{p+1}}(x)\sin(N_{p+1}l_{1}X(x,t)))\\
	&=a(x,t)\sin(N_{p+1}l_{1}X(x,t))+b(x,t)\cos(N_{p+1}l_{1}X(x,t))
\end{align*}
with 
$$\|a(x,t)\|_{C^{m}},\|b(x,t)\|_{C^{m}}\leq  C \ln(N_{p+1})^{C}N_{p+1}^{2(m-1)\beta_{p+1}}$$
where the constants depend on $\beta_{p},\beta_{p+1},(a_{j})_{j\in\mathds{N}},(b_{j})_{j\in\mathds{N}}$, $k$ and $m$. Therefore
\begin{align*}
	&\|a(x,t) N_{p+1}l_{2}\p_{x_{1}}X(x,t)f^{2}_{N_{p+1}}(x)\|_{C^{m}}\leq C_{m}\sum_{i=0}^{m} \|a(x,t) )\|_{C^{i}} \|N_{p+1}l_{2}\p_{x_{1}}X(x,t)f^{2}_{N_{p+1}}(x)\|_{C^{m-i}}\\
	&\leq C \ln(N_{p+1})^{C}N_{p+1}^{2(m-1)\beta_{p+1}+1}\\
	&\|a(x,t) \p_{x_{1}}f^{2}_{N_{p+1}}(x)\|_{C^{m}}\leq C_{m}\sum_{i=0}^{m} \|a(x,t) )\|_{C^{i}} \|f^{2}_{N_{p+1}}(x)\|_{C^{m-i+1}}\leq C N_{p+1}^{2(m-1)\beta_{p+1}+1}\\
	&\|b(x,t) N_{p+1}l_{2}\p_{x_{2}}X(x,t)f^{2}_{N_{p+1}}(x)\|_{C^{m}}\leq C_{m}\sum_{i=0}^{m} \|b(x,t) )\|_{C^{i}} \|N_{p+1}l_{2}\p_{x_{2}}X(x,t)f^{2}_{N_{p+1}}(x)\|_{C^{m-i}}\\
	&\leq C \ln(N_{p+1})^{C}N_{p+1}^{2(m-1)\beta_{p+1}+1}\\
	&\|b(x,t) \p_{x_{2}}f^{2}_{N_{p+1}}(x)\|_{C^{m}}\leq C_{m}\sum_{i=0}^{m} \|b(x,t) )\|_{C^{i}} \|f^{2}_{N_{p+1}}(x)\|_{C^{m-i+1}}\leq C N_{p+1}^{2(m-1)\beta_{p+1}+1}
\end{align*}
with the constants on depend on  $\beta_{p},\beta_{p+1},(a_{j})_{j\in\mathds{N}},(b_{j})_{j\in\mathds{N}}$, $k$ and $m$, and this gives the bound for the remaining terms we wanted to bound, finishing the proof.
\end{proof}
 \begin{remark}
 	What Lemma \ref{cuadpeque} is telling us is that the interactions between functions considered in the lemma (which are, in some sense, highly anisotropic) is of lower order than what the rough a priori bounds would give, and in particular
 	$$u^{k}(\lambda_{1}f^{i}_{N_{p+1}}(x)\sin(N_{p+1}l_{1}X(x,t)+\theta^{1}_{0}))\cdot\nabla (\lambda_{2}f^{2}_{N_{p+1}}(x)\sin(N_{p+1}l_{2}X(x,t)+\theta^{2}_{0}))$$
 	is smaller ($N_{p+1}^{-2\beta_{p+1}}$ times smaller) than what one expects on first sight. This will be crucial later on to construct our $p+1$-layered solution, since we will be considering many of this kind of functions, which will interact weakly with each other.
 \end{remark}

\subsection{Constructing the $p+1$ layer}\label{subsecfinconst}
We are now ready to construct the $p+1$-layered solution from a generic $p$-layered solution. To do this we will use an iterative argument: We will start with a very rough guess of the behavior for the new layer, and then we will improve it iteratively until we obtain a solution to \ref{eq:IPMsystem} with the desired properties. More precisely, our first guess would require a very big source term to be a solution to the \ref{eq:IPMsystem} system, and each correction will make the source required smoother and smaller.

To address how we can find these corrections we have the following definition.

\begin{definition}\label{induccion}
	Given $\sum_{l=0}^{p}\rho_{l}(x,t)$ a $p$-layered solution, $\beta_{p+1}\in (0,\frac14)$,  (and $N_{p+1}$ chosen according to Remark \ref{remarknp1}), $M=\lceil\frac{1}{\beta_{p+1}^2}\rceil$ and, for $i=1,...,j$, $\rho_{p+1,j} $ fulfilling
    \begin{equation}\label{rhopi}
        \rho_{p+1,i}=\sum_{l=1}^{2^{i}}g^{1}_{l,i}(x,t)\sin(N_{p+1}lX(x,t))+\sum_{l=1}^{2^{i}}g^{2}_{l,i}(x,t)\cos(N_{p+1}lX(x,t)),
    \end{equation}
	we will define
	\begin{align*}
		\tilde{F}_{p+1,j}:=&P_{>0}\big[ u^{M}(\rho_{p+1,j})\cdot\nabla (\sum_{l=1}^{j}\rho_{p+1,l})+u^{M}(\sum_{l=1}^{j-1}\rho_{p+1,l})\cdot\nabla \rho_{p+1,j}\big]\\
		&+u^{M}(\rho_{p+1,j})\cdot\nabla (\sum_{l=0}^{p}\P_{M}(\rho_{l}))-u^{0}(\rho_{p+1,j})\cdot\nabla (\sum_{l=0}^{p}\P_{1}(\rho_{l}))\\
		&+(\P_{M}(u(\sum_{l=0}^{p}\rho_{l}))-\P_{1}(u(\sum_{l=0}^{p}\rho_{l})))\cdot \big[\sum_{l=1}^{2^{j}}\sin(N_{p+1}lX(x,t))\nabla g^{1}_{l,j}(x,t)+\sum_{l=1}^{2^{j}}\cos(N_{p+1}lX(x,t))\nabla g^{2}_{l,j}(x,t)\big]
	\end{align*}
    and
	\begin{align*}
		F_{p+1,j}:=&P_{0}\big[ u^{M}(\rho_{p+1,j})\cdot\nabla (\sum_{l=1}^{j}\rho_{p+1,l})+u^{M}(\sum_{l=1}^{j-1}\rho_{p+1,l})\cdot\nabla \rho_{p+1,j}\big]\\
		&+\big[u(\rho_{p+1,j})-u^{M}(\rho_{p+1,j})]\cdot\nabla (\sum_{l=1}^{j}\rho_{p+1,l})+\big[u(\sum_{l=1}^{j-1}\rho_{p+1,l})-u^{M}(\sum_{l=1}^{j-1}\rho_{p+1,l})\big]\cdot\nabla \rho_{p+1,j}\\
		&+u(\rho_{p+1,j})\cdot\nabla (\sum_{l=0}^{p}\rho_{l})-u^{M}(\rho_{p+1,j})\cdot\nabla (\sum_{l=0}^{p}\P_{M}(\rho_{l}))\\
		&+\big[u(\sum_{l=0}^{p}\rho_{l})-\P_{M}(u(\sum_{l=0}^{p}\rho_{l}))\big]\cdot \big[\sum_{l=1}^{2^{j}}g^{1}_{l,j}(x,t)\nabla\sin(N_{p+1}lX(x,t))+\sum_{l=1}^{2^{j}}g^{2}_{l,j}(x,t)\nabla\cos(N_{p+1}lX(x,t))\big]\\
		&+\big[u(\sum_{l=0}^{p}\rho_{l})-\P_{M}(u(\sum_{l=0}^{p}\rho_{l}))\big]\cdot \big[\sum_{l=1}^{2^{j}}\sin(N_{p+1}lX(x,t))\nabla g^{1}_{l,j}(x,t)+\sum_{l=1}^{2^{j}}\cos(N_{p+1}lX(x,t))\nabla g^{2}_{l,j}(x,t)\big]
	\end{align*}
	where 
	$$P_{0}[f(x)\sin(aX(x,t))]=0,\quad P_{0}[f(x)\cos(aX(x,t))]=\begin{cases}
		f(x)\cos(aX(x,t))\text{ if } a=0,\\
		0\text{ otherwise. }
	\end{cases}$$
	$$P_{>0}:=Id-P_{0}.$$
	This allows us to write
	$$\tilde{F}_{p+1,j}=\sum_{l=1}^{2^{j+1}}h^{1}_{l,j}(x,t)\sin(N_{p+1}lX(x,t))+\sum_{l=1}^{2^{j+1}}h^{2}_{l,j}(x,t)\cos(N_{p+1}lX(x,t)).$$
	We can then define 
	$$\rho_{p+1,j+1}(x,t)=\int_{t_{0}}^{t}S^{M}_{s,t}[-\tilde{F}_{p+1,j}(x,s)]ds$$
	with $t_{0}$ the value obtained from Lemma \ref{girot0}
	.

\end{definition}

Since these definitions are a little obscure without context, let us explain the significance of some of these functions, as well as some of the properties they have. The ultimate goal here is to take
$$\rho_{p+1}(x,t)=\sum_{j=1}^{J}\rho_{p+1,j}(x,t)$$
for some value of $J$. Let us now remember that $\rho_{p}$ fulfills the evolution equation
$$\p_{t}(\sum_{j=0}^{p}\rho_{j}(x,t))+u(\sum_{j=0}^{p}\rho_{j}(x,t))\cdot\nabla(\sum_{j=0}^{p}\rho_{j}(x,t))=\sum_{j=0}^{p}F_{j}(x,t).$$
If we now assume that
\begin{align*}
	&\p_{t}(\sum_{j=0}^{p}\rho_{j}(x,t)+\sum_{j=1}^{j_{0}}\rho_{p+1,j}(x,t))+u(\sum_{j=0}^{p}\rho_{j}(x,t)+\sum_{j=1}^{j_{0}}\rho_{p+1,j}(x,t))\cdot\nabla(\sum_{j=0}^{p}\rho_{j}(x,t)+\sum_{j=1}^{j_{0}}\rho_{p+1,j}(x,t))\\
	&=\sum_{j=0}^{p}F_{j}(x,t)+\sum_{j=0}^{j_{0}}F_{p+1,j}(x,t)+\tilde{F}_{p+1,j_{0}}(x,t),
\end{align*}
then the given definitions actually give us that
\begin{align*}
	&\p_{t}(\sum_{j=0}^{p}\rho_{j}(x,t)+\sum_{j=1}^{j_{0}+1}\rho_{p+1,j}(x,t))+u(\sum_{j=0}^{p}\rho_{j}(x,t)+\sum_{j=1}^{j_{0}+1}\rho_{p+1,j}(x,t))\cdot\nabla(\sum_{j=0}^{p}\rho_{j}(x,t)+\sum_{j=1}^{j_{0}+1}\rho_{p+1,j}(x,t))\\
	&=\sum_{j=0}^{p}F_{j}(x,t)+\sum_{j=0}^{j_{0}+1}F_{p+1,j}(x,t)+\tilde{F}_{p+1,j_{0}+1}(x,t),
\end{align*}
which we will prove later. We would like to show that these solutions that we are constructing iteratively for the \ref{eq:IPMsystem} system eventually have the desired behaviour, i.e., that one of these solutions gives a $p+1$-layered solution. The main difficulty here is to show that the source terms that would be associated to the $p+1$-layer, that is to say,
$$\sum_{j=1}^{j_{0}}F_{p+1,j}(x,t)+\tilde{F}_{p+1,j_{0}}(x,t),$$
has the desired smoothness and smallness for big $j_{0}$. The idea here is to check that all the terms $F_{p+1,j}(x,t)$ are smooth and small, and that $\tilde{F}_{p+1,j_{0}}(x,t)$ becomes smoother and smaller as $j_{0}$ goes to infinity.
It is also useful to keep in mind the moral difference between $F_{p+1,j}$ and $\tilde{F}_{p+1,j}$: $\tilde{F}_{p+1,j}$ are very explicit, and this allows us to use them to construct a better approximation for our solution by applying the solution operator $S^{M}_{t_{1},t_{2}}$ to them, while $F_{p+1,j}$ has less structure, but they are small enough that we can include them directly in the source term.

It is also worth mentioning the reason why we include the operator $P_{0}$ and $P_{>0}$ in our definitions. Basically, the operator $P_{0}$ gives the low frequency part of our errors, while $P_{>0}$ gives the high frequency part. The reason to distinguish between both types is twofold: First, the low frequency terms are much smoother, which means we can include them directly in the smoother part of the error terms, $F_{p+1,j}$. Secondly, for our iterative process we need to compute $u^{k}(\rho_{p+1,j})$, and the operator $u^{k}$ is not defined for the low frequency terms, so we cannot include them in our iterative construction.

Before we go on to prove the next lemma, that will allow us to construct our $p$-layered solutions, it is important to keep in mind some of the main properties that we will use through the proof.
\begin{enumerate}
    \item Since $N_{p+1}\geq e^{K_{-}\beta_{p}N_{p}^{\frac{\beta_{p+1}}{8}}}$, we have that
    $$\ln(N_{p+1})\geq K_{-}\beta_{p}N_{p}^{\frac{\beta_{p+1}}{8}}$$
    and thus we can use that
    $$N_{p}^{C}\leq C_{\beta_{p}}\ln(N_{p+1})^{C_{\beta_{p+1}}}.$$
    \item We choose 
    $$M:=\lceil\frac{1}{\beta_{p+1}^2}\rceil\leq \lceil\frac{4}{\beta_{p}^2}\rceil\leq K_{p}-1$$
    so in particular we have that $M+1\leq K_{p}$ and thus, using that $\P_{M}\sum_{l=0}^{p}\rho_{l}$ is a polynomial of degree $M$
    $$\|\P_{M}\sum_{l=0}^{p}\rho_{l}\|_{C^{m}(B_{1}(0))}=\|\P_{M}\sum_{l=0}^{p}\rho_{l}\|_{C^{M}(B_{1}(0))}\leq C_{M}N_{p}^{M}\leq C_{\beta_{p},\beta_{p+1}}\ln(N_{p+1})^{C_{\beta_{p},\beta_{p+1}}} $$
    and similarly
    $$\|\P_{M}u(\sum_{l=0}^{p}\rho_{l})\|_{C^{m}(B_{1}(0))}\leq C_{\beta_{p},\beta_{p+1}}\ln(N_{p+1})^{C_{\beta_{p},\beta_{p+1}}}. $$
    \item We will be considering functions fulfilling
    $$\p_{t}f+(\P_{1}(u(\sum_{l=0}^{p}\rho_{l})))\cdot\nabla f=0$$
    for $t\in[t_{0},1]$. In particular, using Lemma \ref{p1phi}, this will imply that, if $\text{supp}(f(x,t_{0}))\subset B_{N_{p+1}^{-2\beta_{p+1}}}(0)$, then for $N_{p+1}$ big, for $t\in[t_{0},1]$
    $$\text{supp}(f(x,t))\subset B_{CN_{p+1}^{-2\beta_{p+1}}N_{p}^{\frac{\beta_{p}}{4}}}(0)\subset B_{N_{p+1}^{-2\beta_{p+1}}\ln(N_{p+1})^4}(0).$$
\end{enumerate}

\begin{lemma}\label{itera}
	Given $\sum_{i=0}^{p}\rho_{i}(x,t)$ a $p$-layered solution and $\beta_{p+1}\in (0,\frac14)$, and $f(x)\in C^{\infty}$ fulfilling
	$\text{supp}f(x)\subset B_{1}(0)$, $f(x)=1$ if $|x|\leq \frac{1}{2}$ and
	$$f(x_{1},x_{2})=f(-x_{1},x_{2})=f(x_{1},-x_{2})=f(-x_{1},-x_{2}),$$
	  we define  
	 $$M:=\lceil\frac{1}{\beta_{p+1}^2}\rceil$$
	 with $\mathcal{A}_{1}(t)$ as in Definition \ref{soloperator} and
	$$\rho_{p+1,1}(x,t)=w(t)S^{M}_{t_{0},t}[f(N_{p+1}^{2\beta_{p+1}}x)\frac{\mathcal{A}_{1}(t_{0})}{N_{p+1}^{1-\beta_{p+1}}}\sin(N_{p+1}X(x,t=t_0))]=g^{1}_{1,1}(x,t)\sin(N_{p+1}X(x,t))$$
    \begin{equation}\label{wt}
        w(t)=\begin{cases}
		0\text{ if }t\leq t_{0},\\
		N_{p}(t-t_{0})\text{ if }t_{0}\leq t\leq t_{0}+N_{p}^{-1},\\
		1\text{ if }t\geq t_{0}+N_{p}^{-1}
	\end{cases}
    \end{equation}
\begin{align*}
	F_{p+1,0}=&w'(t)S^{M}_{t_{0},t}[f(N_{p+1}^{2\beta_{p+1}}x)\sin(N_{p+1}x_{1})],
\end{align*}
with $X(x,t)$ as in Definition \ref{soloperator}.
Then defining $\rho_{p+1,j}$, $F_{p+1,j}$, $\tilde{F}_{p+1,j}$ inductively using Definition \ref{induccion}, we conclude that $\rho_{p+1,j}$ are odd and we have that, for $N_{p}$ big enough (depending on $\beta_{p}$ and $\beta_{p+1}$), for any $j_{0}\in\mathds{N}$, $t\in[t_{0},1]$
\begin{align}\label{evolucionj0}
	&\p_{t}(\sum_{j=0}^{p}\rho_{j}(x,t)+\sum_{j=1}^{j_{0}}\rho_{p+1,j}(x,t))+u(\sum_{j=0}^{p}\rho_{j}(x,t)+\sum_{j=1}^{j_{0}}\rho_{p+1,j}(x,t))\cdot\nabla(\sum_{j=0}^{p}\rho_{j}(x,t)+\sum_{j=1}^{j_{0}}\rho_{p+1,j}(x,t))\\ \nonumber
	&=\sum_{j=0}^{p}F_{j}(x,t)+\sum_{j=0}^{j_{0}}F_{p+1,j}(x,t)+\tilde{F}_{p+1,j_{0}}(x,t).
\end{align}
These $\rho_{p+1,j}$ can be decomposed as in \eqref{rhopi} in Definition \ref{soloperator} with
\begin{equation}\label{boundg}
	\|g^{1}_{l,j}(\cdot,t)\|_{C^{m}},\|g^{2}_{l,j}(\cdot,t)\|_{C^{m}}\leq \mathcal{A}_{1}(t)C_{m,j,\beta_{p},\beta_{p+1}}\ln(N_{p+1})^{C_{m,j,\beta_{p},\beta_{p+1}}}N_{p+1}^{-1+[2m+2-j]\beta_{p+1}},
\end{equation}
\begin{equation}\label{boundh}
	\|h^{1}_{l,j}(\cdot,t)\|_{C^{m}},\|h^{2}_{l,j}(\cdot,t)\|_{C^{m}}\leq \mathcal{A}_{1}(t)C_{m,j,\beta_{p},\beta_{p+1}}\ln(N_{p+1})^{C_{m,j,\beta_{p},\beta_{p+1}}}N_{p+1}^{-1+[2m+1-j]\beta_{p+1}},
\end{equation}
and
\begin{equation}\label{support}
	\text{supp}(g^{1}_{l,j}(x,t)),\text{supp}(g^{2}_{l,j}),\text{supp}(h^{1}_{l,j}),\text{supp}(h^{2}_{l,j})\subset \text{supp}(\rho_{p+1,1})\subset B_{N_{p+1}^{-2\beta_{p+1}}\ln(N_{p+1})^4}(0).
\end{equation}
\end{lemma}

\begin{proof}
	Let us start by showing \eqref{evolucionj0} by induction. First, to show the case $j_{0}=1$, we note that
\begin{align*}
	F_{p+1,0}+F_{p+1,1}+\tilde{F}_{p+1,1}&=w'(t)S^{M}_{t_{0},t}[f(N_{p+1}^{2\beta_{p+1}}x)\sin(N_{p+1}x_{1})]+ u(\rho_{p+1,1})\cdot\nabla \rho_{p+1,1}\\	
	&+u(\rho_{p+1,1})\cdot\nabla (\sum_{l=0}^{p}\rho_{l})-u^{0}(\rho_{p+1,1})\cdot\nabla (\sum_{l=0}^{p}\P_{1}(\rho_{l}))\\
	&+\big[u(\sum_{l=0}^{p}\rho_{l})-\P_{M}(u(\sum_{l=0}^{p}\rho_{l}))\big] w(t)g^{1}_{1,1}(x,t)\cdot\nabla \sin(N_{p+1}X(x,t))\\
	&+\big[u(\sum_{l=0}^{p}\rho_{l})-\P_{1}(u(\sum_{l=0}^{p}\rho_{l}))\big] w(t)\sin(N_{p+1}X(x,t))\cdot\nabla g^{1}_{1,1}(x,t).
\end{align*}
On the other hand for \eqref{evolucionj0} to be true for $j_{0}=1$, using that $\sum_{j=0}^{p}\rho_{j}(x,t)$ is a solution to \eqref{eq:ipm}, we would need
$$F_{p+1,0}+F_{p+1,1}+\tilde{F}_{p+1,1}=\frac{d}{dt}\rho_{p+1,1}+u(\rho_{p+1,1})\cdot\nabla(\rho_{p+1,1}+\sum_{l=0}^{p}\rho_{l})+u(\sum_{l=0}^{p}\rho_{l})\cdot\nabla \rho_{p+1,1}$$
which is equivalent to
\begin{align*}
	&w(t)\p_{t}S^{M}_{t_{0},t}[f(N_{p+1}^{2\beta_{p+1}}x)\sin(N_{p+1}x_{1})]=-u^{0}(\rho_{p+1,1})\cdot\nabla (\sum_{l=0}^{p}\P_{1}(\rho_{l}))\\
 &-\big[\P_{M}(u(\sum_{l=0}^{p}\rho_{l})) g^{1}_{1,1}(x,t)\cdot\nabla \sin(N_{p+1}X(x,t))+\P_{1}(u(\sum_{l=0}^{p}\rho_{l})) \sin(N_{p+1}X(x,t))\cdot\nabla g^{1}_{1,1}(x,t)\big]
\end{align*}
which is true by our definition of $S^{M}_{t_{0},t}$ and Remark \ref{remarkunfsol}.
Now, assuming that the result is true for $j_{0}\leq J_{0}$, for \eqref{evolucionj0} to be true for $j_{0}=J_{0}+1$ we would need
\begin{align*}
	&\p_{t}\rho_{p+1,J_{0}+1}+u(\rho_{p+1,J_{0}+1})\cdot\nabla (\sum_{j=1}^{J_{0}+1}u(\rho_{p+1,j})+\sum_{j=0}^{p}\rho_{j})+u(\sum_{j=1}^{J_{0}}u(\rho_{p+1,j})+\sum_{j=0}^{p}\rho_{j})\cdot\nabla \rho_{p+1,J_{0}+1}\\
	&=F_{p+1,J_{0}+1}+\tilde{F}_{p+1,J_{0}+1}-\tilde{F}_{p+1,J_{0}}.
\end{align*}
As before, we now compute $F_{p+1,J_{0}+1}+\tilde{F}_{p+1,J_{0}+1}$:

\begin{align*}
	&F_{p+1,J_{0}+1}+\tilde{F}_{p+1,J_{0}+1}\\
	&=u(\rho_{p+1,J_{0}+1})\cdot\nabla (\sum_{j=1}^{J_{0}+1}\rho_{p+1,j})+u(\sum_{j=1}^{J_{0}}\rho_{p+1,j})\cdot\nabla \rho_{p+1,J_{0}+1}\\
	&+u(\rho_{p+1,J_{0}+1})\cdot\nabla(\sum_{l=0}^{p}\rho_{l})- u^{0}(\rho_{p+1,J_{0}+1})\cdot\nabla(\sum_{l=0}^{p}\P_{1}\rho_{l})\\
	&+\big[u(\sum_{l=0}^{p}\rho_{l})-\P_{M}(u(\sum_{l=0}^{p}\rho_{l}))\big]\cdot \big[\sum_{l=1}^{2^{J_{0}+1}}g^{1}_{l,J_{0}+1}(x,t)\nabla\sin(N_{p+1}lX(x,t))+\sum_{l=1}^{2^{J_{0}+1}}g^{2}_{l,J_{0}+1}(x,t)\nabla\cos(N_{p+1}lX(x,t))\big]\\
	&+\big[u(\sum_{l=0}^{p}\rho_{l})-\P_{1}(u(\sum_{l=0}^{p}\rho_{l}))\big]\cdot \big[\sum_{l=1}^{2^{J_{0}+1}}\sin(N_{p+1}lX(x,t))\nabla g^{1}_{l,J_{0}+1}(x,t)+\sum_{l=1}^{2^{J_{0}+1}}\cos(N_{p+1}lX(x,t))\nabla g^{2}_{l,J_{0}+1}(x,t)\big]
\end{align*}
and therefore it is enough to show that
\begin{align*}
	&\p_{t}\rho_{p+1,J_{0}+1}\\
	&=-\tilde{F}_{p+1,J_{0}}-u^{0}(\rho_{p+1,J_{0}+1})\cdot\nabla(\sum_{l=0}^{p}\P_{1}\rho_{l})\\
	&-\P_{M}(u(\sum_{l=0}^{p}\rho_{l}))\cdot \big[\sum_{l=1}^{2^{J_{0}+1}}g^{1}_{l,J_{0}+1}(x,t)\nabla\sin(N_{p+1}lX(x,t))+\sum_{l=1}^{2^{J_{0}+1}}g^{2}_{l,J_{0}+1}(x,t)\nabla\cos(N_{p+1}lX(x,t))\big]\\
	&-\P_{1}(u(\sum_{l=0}^{p}\rho_{l}))\cdot \big[\sum_{l=1}^{2^{J_{0}+1}}\sin(N_{p+1}lX(x,t))\nabla g^{1}_{l,J_{0}+1}(x,t)+\sum_{l=1}^{2^{J_{0}+1}}\cos(N_{p+1}lX(x,t))\nabla g^{2}_{l,J_{0}+1}(x,t)\big],
\end{align*}
but we can now use Remark \ref{evolsol0} to get

\begin{align*}
	&\p_{t}\rho_{p+1,J_{0}+1}=\frac{d}{dt}\int_{t_{0}}^{t}S^{M}_{s,t}[-\tilde{F}_{p+1,J_{0}}(x,s)]ds\\
	&=-\tilde{F}_{p+1,J_{0}}-u^{0}(\rho_{p+1,J_{0}+1})\cdot\nabla(\sum_{l=0}^{p}\P_{1}\rho_{l})\\
	&-\P_{M}(u(\sum_{l=0}^{p}\rho_{l}))\cdot \big[\sum_{l=1}^{2^{J_{0}+1}}g^{1}_{l,J_{0}+1}(x,t)\nabla\sin(N_{p+1}lX(x,t))+\sum_{l=1}^{2^{J_{0}+1}}g^{2}_{l,J_{0}+1}(x,t)\nabla\cos(N_{p+1}lX(x,t))\big]\\
	&-\P_{1}(u(\sum_{l=0}^{p}\rho_{l}))\cdot \big[\sum_{l=1}^{2^{J_{0}+1}}\sin(N_{p+1}lX(x,t))\nabla g^{1}_{l,J_{0}+1}(x,t)+\sum_{l=1}^{2^{J_{0}+1}}\cos(N_{p+1}lX(x,t))\nabla g^{2}_{l,J_{0}+1}(x,t)\big]
\end{align*}
which is exactly what we wanted to prove.

Now, to obtain the the bounds \eqref{boundg} and \eqref{boundh}, we will also argue by induction. Throughout this proof, we will omit the sub-index $\beta_{p},\beta_{p+1}$ to have a less cumbersome notation, but one should keep in mind that the constants $C$ involved in this proof depend on $\beta_{p},\beta_{p+1}$. Let us assume that we know the bound for $\tilde{F}_{p+1,j}$ to be correct when $j=j_{0}$, so 
$$\tilde{F}_{p+1,j_{0}}=\sum_{l=1}^{2^{j_{0}+1}}h_{l,j_{0}}^{1}(x,t)\sin(N_{p+1}lX(x,t))+\sum_{l=1}^{2^{j_{0}+1}}h_{l,j_{0}}^{2}(x,t)\cos(N_{p+1}lX(x,t)),$$
with $h_{l,j_{0}}^{1}(x,t)$, $h_{l,j_{0}}^{2}(x,t)$ fulfilling the bounds from \eqref{boundh}.
We now have that
$$\rho_{p+1,j_{0}+1}=\sum_{l=1}^{2^{j_{0}+1}}\int_{t_{0}}^{t}S_{s,t}[h_{l,j_{0}}^{1}(x,s)\sin(N_{p+1}lX(x,s))]ds+\sum_{l=1}^{2^{j_{0}+1}}\int_{t_{0}}^{t}S_{s,t}[h_{l,j_{0}}^{2}(x,s)\cos(N_{p+1}lX(x,s))]ds.$$

Note that, by the definition of the first order solution operator, we have that

$S_{s,t}[h_{l,j_{0}}^{1}(x,s)\sin(N_{p+1}lX(x,s))]=\frac{\mathcal{A}_{1}(t)}{\mathcal{A}_{1}(s)}f_{s}(x,t)\sin(N_{p+1}lX(x,t))$
with
$$\p_{t}f_{s}(x,t)+\P_{1}[u(\sum_{l=0}^{p}\rho_{l})]\cdot\nabla f(x,t)=0$$
$$f_{s}(x,t=s)=h_{l,j_{0}}^{1}(x,s),$$
and the analogous expression for $S_{s,t}[h_{l,j_{0}}^{2}(x,s)\cos(N_{p+1}lX(x,s)]$.

But
$f_{s}(x,s)=f_{s}(\phi(x,s,t),t)$
with 
$$\p_{t}\phi(x,s,t)=\P_{1}(u(\sum_{l=0}^{p}\rho_{l}))$$ 
$$\phi(x,s,s)=x,$$
which in particular implies $f_{s}(\phi(x,t,s),s)=f_{s}(x,t)$.
Applying Lemma \ref{p1phi}, we have that, for $i,j=1,2$, for $t,s\in[t_{0},1]$ 
$|\p_{x_{j}}\phi_{i}(x,t,s)|\leq CN_{p}^{\frac{\beta_{p}}{4}}$, and this combined with the fact that $\phi_{i}(x,t,s)$ are linear functions gives us
$$\|f_{s}(x,t)\|_{C^{j}}\leq (CN_{p}^{\frac{\beta_{p}}{4}})^{j}\|f_{s}(x,t=s)\|_{C^{j}}=(CN_{p}^{\frac{\beta_{p}}{4}})^{j}\|h_{l,j_{0}}^{1}(x,s)\|_{C^{j}}.$$
Finally, using that
$$S_{s,t}[h_{l,j_{0}}^{1}(x,s)\sin(N_{p+1}lX(x,s))]=\frac{\mathcal{A}_{1}(t)}{\mathcal{A}_{1}(s)}f_{s}(x,t)\sin(N_{p+1}lX(x,t))$$
we get
\begin{align*}
	\sin(N_{p+1}lX(x,t))g^{1}_{l,j_{0}+1}(x,t)=&\int_{t_{0}}^{t}S_{s,t}[h_{l,j_{0}}^{1}(x,s)\sin(N_{p+1}lX(x,s)]ds\\&=\sin(N_{p+1}lX(x,t))\int_{t_{0}}^{t}\frac{\mathcal{A}_{1}(t)}{\mathcal{A}_{1}(s)}h^{1}_{l,j}(\phi(x,t,s),s)ds
\end{align*}
and thus, using the bounds for $\|h_{l,j_{0}}^{1}(x,s)\|_{C^{m}}$, we obtain
\begin{align*}
	&\|g^{1}_{l,j_{0}+1}(x,t)\|_{C^{m}}\leq \|\int_{t_{0}}^{t}\frac{\mathcal{A}_{1}(t)}{\mathcal{A}_{1}(s)}h^{1}_{l,j_{0}}(\phi(x,t,s),s)ds\|_{C^{m}}\\
	&\leq CN_{p}^{\frac{m\beta_{p}}{4}}\int_{t_{0}}^{t}\frac{\mathcal{A}_{1}(t)}{\mathcal{A}_{1}(s)}\|h_{l,j_{0}}^{1}(x,s)\|_{C^{m}}ds\leq C_{m,j_{0},\beta_{p},\beta_{p+1}}N_{p}^{\frac{m\beta_{p}}{4}} \mathcal{A}_{1}(t)\ln(N_{p+1})^{C_{m,j_{0},\beta_{p},\beta_{p+1}}}N_{p+1}^{-1+[2m+1-j_{0}]\beta_{p+1}}
\end{align*}
and the exact same bounds holds for $g^{2}_{l,j_{0}+1}(x,t)$, and therefore the bounds for $g^{1}_{l,j_{0}+1},g^{2}_{l,j_{0}+1}$ are proved. Note that, since
$$\rho_{p+1,1}(x,t)=w(t)S^{M}_{t_{0},t}[f(N_{p+1}^{2\beta_{p+1}}x)\frac{\mathcal{A}_{1}(s)}{N_{p+1}^{1-\beta_{p+1}}}\sin(N_{p+1}X(x,t=t_0))]=g^{1}_{1,1}(x,t)\sin(N_{p+1}X(x,t))$$
and $|w(t)|\leq 1$ we can use the same argument to show
\begin{equation}\label{g111}
    \|g^{1}_{1,1}(x,t)\|_{C^{m}}\leq C\mathcal{A}_{1}(t)N_{p}^{\frac{m\beta_{p}}{4}}N_{p+1}^{-1+(2m+1)\beta_{p+1}}.
\end{equation}

Before we go on to prove the bounds for $\tilde{F}_{p+1,j}$, it is best to prove first \eqref{support}. 
For this, we start by noting that the inequality is obviously fulfilled for $\rho_{p+1,1}$ by Lemma \ref{p1phi}, and that
$$x\in\text{supp}(\rho_{p+1,1}(x,t))\Leftrightarrow \phi(x,t,t_{0})\in\text{supp}(\rho_{p+1,1}(x,t_{0}))$$
since the support of $\rho_{p+1,1}$ is getting transported with velocity $\P_{1}(u(\sum_{l=0}^{p}\rho_{l}))$.

 Furthermore, we note that, for any $j$, $\text{supp}(\tilde{F}_{p+1,j})\subset \rho_{p+1,j}$, since all the terms that form $\tilde{F}_{p+1,j}$ include either $u^{M}(\rho_{p+1,j})$, $u^{0}(\rho_{p+1,j})$ or derivatives of $\rho_{p+1,j}$, and all these operators maintain the support.
Therefore, we only need to show that, if $\tilde{F}_{p+1,j}$ has the right bounds for the support, then so does $\rho_{p+1,j+1}$.

But for this we note that, from the definition of $S_{s,t}^{M}$, we have 
\begin{align*}
	&\text{supp}(S_{s,t}^{M}(h^{1}_{l,j}(x,s)\sin(N_{p+1}lX(x,s))))=\{y\in \R^2: y= \phi(x,s,t), x\in\text{supp}(h^{1}_{l,j}(x,s))\}\\
	&\subset\{y\in \R^2: y= \phi(x,s,t), x\in\text{supp}(\rho_{p+1,1}(x,s))\}\subset \{y\in \R^2: y= \phi(x,s,t), \phi(x,s,t_{0})\in\text{supp}(\rho_{p+1,1}(x,t_{0}))\}\\
	&=\{y\in \R^2: y= \phi(\phi(z,t_{0},s),s,t), z\in\text{supp}(\rho_{p+1,1}(x,t_{0}))\}=\{y\in \R^2: y= \phi(z,t_{0},t), z\in\text{supp}(\rho_{p+1,1}(x,t_{0}))\}\\
	&=\{y\in \R^2: \phi(y,t,t_{0})\in\text{supp}(\rho_{p+1,1}(x,t_{0}))\}=\text{supp}(\rho_{p+1,1}(x,t))
\end{align*}
as we wanted to prove, and the analogous computation works for $h^{2}_{l,j}(x,s)\cos(N_{p+1}lX(x,s))$.

We want now to obtain the bounds for $h^{1}_{l,j},h^{2}_{l,j}$, assuming that \eqref{boundg} is correct for the desired value $j$. Note also that the bounds for $g^{1}_{1,1},g^{2}_{1,1}$ are trivially fulfilled from \eqref{g111} and $g^{2}_{1,1}=0$.

We start first with the terms coming from
\begin{align}\label{Fumbg}
	&u^{M}(\rho_{p+1,j})\cdot\nabla (\sum_{l=0}^{p}\P_{M}(\rho_{l}))-u^{0}(\rho_{p+1,1})\cdot\nabla (\sum_{l=0}^{p}\P_{1}(\rho_{l}))\nonumber\\
	&=u^{M}(\rho_{p+1,j})\cdot\nabla [(\sum_{l=0}^{p}\P_{M}(\rho_{l}))-(\sum_{l=0}^{p}\P_{1}(\rho_{l}))]+[u^{M}(\rho_{p+1,j})-u^{0}(\rho_{p+1,1})]\cdot\nabla (\sum_{l=0}^{p}\P_{1}(\rho_{l})).
\end{align}
Thus, we can use Lemma \ref{boundsukgen} (more precisely \eqref{decompuk} and \eqref{decompukbd}) and the bounds for $g^{1}_{l,j},g^{2}_{l,j}$ to get, for $q=1,2$
\begin{align*}
	&u_{q}^{M}(g^{1}_{l,j}(x,t)\sin(N_{p+1}lX(x,t)))-u_{q}^{0}(g^{1}_{l,j}(x,t)\sin(N_{p+1}lX(x,t)))\\
	&=a_{q}(x,t)\sin(N_{p+1}lX(x,t))+b_{q}(x,t)\cos(N_{p+1}lX(x,t))
\end{align*}
with
\begin{align*}
		&\|a_{q}(x,t)\|_{C^{m}},\|b_{q}(x,t)\|_{C^{m}}\leq C_{j,m}\ln(N_{p+1})^{C_{j,m}}\mathcal{A}_{1}(t)N_{p+1}^{-1+[2m-j]\beta_{p+1}}
\end{align*}
and thus, 
\begin{align*}
	&\|a_{q}(x)\p_{x_{q}}(\sum_{l=0}^{p}\P_{1}(\rho_{l}))\|_{C^{m}}\leq \sum_{i=0}^{m}\|a_{q}(x)\|_{C^{i}}\|\sum_{l=0}^{p}\P_{1}(\rho_{l})\|_{C^{m-i+1}}\\
	&\leq C_{j,m}\ln(N_{p+1})^{C_{j,m}}\mathcal{A}_{1}(t)N_{p+1}^{-1+[2m-j]\beta_{p+1}}
\end{align*}
and the same inequalities follow for the term including $b(x,t)$, so we get the desired bounds. For the other term in \eqref{Fumbg}, as before we can use Lemma \ref{boundsukgen} to get for $g=1,2$
\begin{align*}
	&u_{q}^{M}(g^{1}_{l,j}(x,t)\sin(N_{p+1}lX(x,t)))\\
	&=a_{q}(x,t)\sin(N_{p+1}lX(x,t))+b_{q}(x,t)\cos(N_{p+1}lX(x,t))
\end{align*}
and using the bounds for $a_{i}(x)$ obtained from Lemma \ref{boundsukgen} and the hypothesis for $g^{1}_{l,j},g^{1}_{l,j}$ we get
\begin{align*}
	&\|a_{q}(x,t)\cdot\p_{x_{q}} [(\sum_{l=0}^{p}\P_{M}(\rho_{l}))-(\sum_{l=0}^{p}\P_{1}(\rho_{l}))]\|_{C^{m}}\\
	&\leq C_{m}\sum_{i=0}^{m}\|a(x,t)\|_{C^{i}} \|\sum_{l=0}^{p}\P_{M}(\rho_{l})-\sum_{l=0}^{p}\P_{1}(\rho_{l})\|_{C^{m+1-i}(\text{supp}(\rho_{p+1,j}))}\\
	&\leq C_{m}\|a(x,t)\|_{C^{m-1}}\|\sum_{l=0}^{p}\P_{M}(\rho_{l})-\sum_{l=0}^{p}\P_{1}(\rho_{l})\|_{C^{m+1}}\\
	&+C_{m}\|a(x,t)\|_{C^{m}} \|\sum_{l=0}^{p}\P_{M}(\rho_{l})-\sum_{l=0}^{p}\P_{1}(\rho_{l})\|_{C^{1}(\text{supp}(\rho_{p+1,j}))}\\
	&\leq C_{j,m}\ln(N_{p+1})^{C_{j,m}}\mathcal{A}_{1}(t)N_{p+1}^{-1+[2(m-1)+2-j]\beta_{p+1}}+C_{j,m}\ln(N_{p+1})^{C_{j,m}}\mathcal{A}_{1}(t)N_{p+1}^{-1+[2m+2-j]\beta_{p+1}}N_{p+1}^{-2\beta_{p+1}}\\
	&\leq C_{j,m}\ln(N_{p+1})^{C_{j,m}}\mathcal{A}_{1}(t)N_{p+1}^{-1+[2m-j]\beta_{p+1}}.
\end{align*}
Next, we consider the term
\begin{align*}
	P_{>0}\big[ u^{M}(\rho_{p+1,j})\cdot\nabla (\sum_{l=1}^{j}\rho_{p+1,l})+u^{M}(\sum_{l=1}^{j-1}\rho_{p+1,l})\cdot\nabla \rho_{p+1,j}\big].
\end{align*}
But we can directly apply Lemma \ref{cuadpeque}  and the bounds for $g^{1}_{l,s},g^{2}_{l,s}$, $s=1,...,j$ to obtain
\begin{align*}
&u^{M}(\rho_{p+1,j})\cdot\nabla (\sum_{l=1}^{j}\rho_{p+1,l})+u^{M}(\sum_{l=1}^{j-1}\rho_{p+1,l})\cdot\nabla \rho_{p+1,j}\\
&=\sum_{l_{1}=1}^{2^{j}}\sum_{l_{2}=1}^{2^{j}}a_{l_{1},l_{2},j}(x,t)\sin(Nl_{1}X(x,t))\sin(Nl_{2}X(x,t))+\sum_{l_{1}=1}^{2^{j}}\sum_{l_{2}=1}^{2^{j}}b_{l_{1},l_{2},j}(x,t)\cos(Nl_{1}X(x,t))\sin(Nl_{2}X(x,t))\\
&+\sum_{l_{1}=1}^{2^{j}}\sum_{l_{2}=1}^{2^{j}}c_{l_{1},l_{2},j}(x,t)\sin(Nl_{1}X(x,t))\cos(Nl_{2}X(x,t))+\sum_{l_{1}=1}^{2^{j}}\sum_{l_{2}=1}^{2^{j}}d_{l_{1},l_{2},j}(x,t)\cos(Nl_{1}X(x,t))\cos(Nl_{2}X(x,t))
\end{align*}
with
\begin{align}\label{boundsabcd}
	&\|a_{l_{1},l_{2},j}\|_{C^{m}},\|b_{l_{1},l_{2},j}\|_{C^{m}},\|c_{l_{1},l_{2},j}\|_{C^{m}},\|d_{l_{1},l_{2},j}\|_{C^{m}}\leq C_{m,j}\mathcal{A}_{1}(t)^2\ln(N_{p+1})^{C_{m,j}}N_{p+1}^{-1+[2m+1-j]\beta_{p+1}}\nonumber\\
	&\leq C_{m,j}\mathcal{A}_{1}(t)\ln(N_{p+1})^{C_{m,j}}N_{p+1}^{-1+[2m+1-j]\beta_{p+1}}
\end{align}
where we used $\mathcal{A}_{1}(t)=\frac{\mathcal{A}_{1}(t)}{\mathcal{A}_{1}(1)}\leq  e^2$ by Lemma \ref{crecimientoAp}. To finish the estimates for this term, we just need to notice that using trigonometric identities, we can decompose this as
\begin{align}\label{cotasNl}
	\sum_{l=0}^{2^{j+1}}a_{l,j}(x,t)\sin(NlX(x,t))+b_{l,j}(x,t)\cos(NlX(x,t))
\end{align}

with $a_{l,j},b_{l,j}$ fulfilling the same bounds as in \eqref{boundsabcd}. Note that this includes a term of "zero frequency" that $P_{>0}$ eliminates.

Now we only need to  obtain bounds for the last term in $\tilde{F}_{p+1,j}$
$$(\P_{M}(u(\sum_{l=0}^{p}\rho_{l})-\P_{1}(u(\sum_{l=0}^{p}\rho_{l})))\cdot \big[\sum_{l=1}^{2^{j}}\sin(N_{p+1}lX(x,t))\nabla g^{1}_{l,j}(x,t)+\sum_{l=1}^{2^{j}}\cos(N_{p+1}lX(x,t))\nabla g^{2}_{l,j}(x,t)\big].$$

In this case, we only need to bound

$$\|(\P_{M}(u(\sum_{l=0}^{p}\rho_{l}))-\P_{1}(u(\sum_{l=0}^{p}\rho_{l})))\cdot  \nabla g^{i}_{l,j}(x,t)\|_{C^{m}}$$
for $i=1,2$, $l=1,2,...,2^{j}$.
But we can use the bounds for $u(\sum_{l=0}^{p}\rho_{l})$ to get, for $i=0,1$
$$\|(\P_{M}(u(\sum_{l=0}^{p}\rho_{l}))-\P_{1}(u(\sum_{l=0}^{p}\rho_{l})))\|_{C^{i}(B_{N_{p+1}^{-2\beta_{p+1}}\ln(N_{p+1})^4}(0))}\leq CN_{p+1}^{-(2-i)2\beta_{p+1}}\ln(N_{p+1})^{C}$$
and, for $m\geq 2$
$$\|(\P_{M}(u(\sum_{l=0}^{p}\rho_{l}))-\P_{1}(u(\sum_{l=0}^{p}\rho_{l})))\|_{C^{m}(B_{N_{p+1}^{-2\beta_{p+1}}\ln(N_{p+1})^4}(0))}\leq C_{m}\ln(N_{p+1})^{C_{m}}$$
so
\begin{align*}
	&\|(\P_{M}(u(\sum_{l=0}^{p}\rho_{l}))-\P_{1}(u(\sum_{l=0}^{p}\rho_{l})))\cdot  \nabla g^{i}_{l,j}(x,t)\|_{C^{m}}\\
	&\leq C_{m,j} \ln(N_{p+1})^{C_{m,j}}N_{p+1}^{-1+\beta_{p+1}(2m-j)}
\end{align*}
which is the last bound we needed and thus the proof is finished.
\end{proof}

We will use Lemma \ref{itera} to construct our $p+1$-layered solution, but we still need to show that this process actually gives us the desired properties.

\begin{lemma}\label{forces}
	Given $\sum_{l=0}^{p}\rho_{l}(x,t)$ a $p$-layered solution, $\beta_{p+1}\in (0,\frac14)$, and $f(x)\in C^{\infty}$ fulfilling
	$\text{supp}f(x)\subset B_{1}(0)$, $f(x)=1$ if $|x|\leq \frac{1}{2}$ and
	$$f(x_{1},x_{2})=f(-x_{1},x_{2})=f(x_{1},-x_{2})=f(-x_{1},-x_{2}),$$
	let $\rho_{p+1,j}(x,t),F_{p+1,j}(x,t),\tilde{F}_{p+1,j}(x,t)$ be as given by Lemma \ref{itera}. Then we have that, if $N_{p}$ is big enough (depending on $\beta_{p},\beta_{p+1}$), for 
	$$m_{0}:=\lfloor\frac{1}{4\beta_{p+1}}\rfloor-10$$ 
	$$\|F_{p+1,j}(x,t)\|_{C^{m_{0}}}\leq  N_{p+1}^{-\frac12},$$
	$$\|\tilde{F}_{p+1,M}(x,t)\|_{C^{m_{0}}}\leq N_{p+1}^{-1}.$$
\end{lemma}
\begin{proof}
	As in other proofs, the constants in this proof will depend on $\beta_{p},\beta_{p+1}$, but we will omit adding the sub-indexes to obtain a more compact notation. Note that this also means that the dependence on $M$ and $m_{0}$ will be implicit.
	Let us start by bounding $\|F_{p+1,0}\|_{C^{m_{0}}}$, since it has a different structure than the rest of the terms. We have that
	$$F_{p+1,0}=N_{p}S^{M}_{t_{0},t}[f(N_{p+1}^{2\beta_{p+1}}x)\frac{\mathcal{A}_{1}(t_{0})}{N_{p+1}^{1-\beta_{p+1}}}\sin(N_{p+1}x_{1})]$$
	if $t\in[t_{0},t_{0}+N_{p}^{-1}]=N_{p}$, and $0$ otherwise.
	Note that, by Remark \ref{remarknp1} we have that
	$$\frac{\mathcal{A}_{1}(t_{0})}{N_{p+1}^{1-\beta_{p+1}}}=N_{p+1}^{-1+\beta_{p+1}-\frac{1}{\beta_{p+1}}}.$$
	By the definition of $S^{M}_{t_{0},t}$, we have that
	$$N_{p}S^{M}_{t_{0},t}[f(N_{p+1}^{2\beta_{p+1}}x)\frac{\mathcal{A}_{1}(t_{0})}{N_{p+1}^{1-\beta_{p+1}}}\sin(N_{p+1}x_{1})]=N_{p}f(x,t)\frac{\mathcal{A}_{1}(t)}{N_{p+1}^{1-\beta_{p+1}}}\sin(N_{p+1}X(x,t))$$
	with
	$$\p_{t}f(x,t)+\P_{1}[u(\sum_{l=0}^{p}\rho_{l})]\cdot\nabla f(x,t)=0$$
	$$f(x,t=t_{0})=f(N_{p+1}^{2\beta_{p+1}}x).$$
	But
	$f(x,t)=f(\phi(x,t,t_{0}),t_{0})$
	with 
	$$\p_{t}\phi(x,s,t)=\P_{1}(u(\sum_{l=0}^{p}\rho_{l}))$$ 
	$$\phi(x,s,s)=x.$$
	Applying Lemma \ref{p1phi}, we have that, for $i,j=1,2$, for $t\in[t_{0},1]$ 
	$|\p_{x_{j}}\phi_{i}(x,t,t_{0})|\leq CN_{p}^{\frac{\beta_{p}}{4}}$, and this combined with the fact that $\phi_{i}(x,t,t_{0})$ are linear functions gives us
	$$\|f(x,t)\|_{C^{j}}\leq (CN_{p}^{\frac{\beta_{p}}{4}})^{j}\|f(x,t_{0})\|_{C^{j}}=(CN_{p}^{\frac{\beta_{p}}{4}})^{j}\|f(N_{p+1}^{2\beta_{p+1}}x)\|_{C^{j}}\leq CN_{p}^{j\frac{\beta_{p}}{4}}N_{p+1}^{2j\beta_{p+1}}.$$
	We can also bound $\mathcal{A}_{1}(t)$ by using that, from Definition \ref{soloperator}
	$$|\p_{t}\mathcal{A}_{1}(t)|\leq |\mathcal{A}_{1}(t)3N_{p}^{\beta_{p}}|$$
	and thus, for $t\in[t_{0},t_{0}+N_{p}^{-1}]$
	$$|\frac{\mathcal{A}_{1}(t)}{\mathcal{A}_{1}(t_{0})}|\leq C.$$
	We can then apply Lemma \ref{boundsfinalproduct} to get
	$$\|F_{p+1,0}(x,t)\|_{C^{m_{0}}}\leq C\ln(N_{p+1})^{C}N_{p+1}^{m_{0}-1+\beta_{p+1}-\frac{1}{\beta_{p+1}}}\leq N_{p+1}^{-1}.$$
	Next, we obtain the bounds for $F_{p+1,j}$, $j>0$. First, to bound
	$$\|P_{0}\big[ u^{M}(\rho_{p+1,j})\cdot\nabla (\sum_{l=1}^{j}\rho_{p+1,l})+u^{M}(\sum_{l=1}^{j-1}\rho_{p+1,l})\cdot\nabla \rho_{p+1,j}\big]\|_{C^{m_{0}}}$$
	we notice that in \eqref{cotasNl} we already showed that
	$$P_{0}\big[ u^{M}(\rho_{p+1,j})\cdot\nabla (\sum_{l=1}^{j}\rho_{p+1,l})+u^{M}(\sum_{l=1}^{j-1}\rho_{p+1,l})\cdot\nabla \rho_{p+1,j}\big]=b_{0,j}(x,t)$$
	with
	\begin{align*}
		&\|b_{0,j}(x,t)\|_{C^{m_{0}}}\leq C\mathcal{A}_{1}(t)\ln(N_{p+1})^{C}N_{p+1}^{-1+2m_{0}\beta_{p+1}}\\
		&\leq C\ln(N_{p+1})^{C}N_{p+1}^{-1+2m_{0}\beta_{p+1}}\leq \frac{1}{10}N_{p+1}^{-\frac12}
	\end{align*}
	which finishes the bounds of this term.
	Next, we study the term
	\begin{align*}
		&\|\big[u(\rho_{p+1,j})-u^{M}(\rho_{p+1,j})]\cdot\nabla (\sum_{l=1}^{j}\rho_{p+1,l})\|_{C^{m_{0}}}\\
		&\leq C\sum_{i=0}^{m_{0}}\|u(\rho_{p+1,j})-u^{M}(\rho_{p+1,j})\|_{C^{m-i}}\|(\sum_{l=1}^{j}\rho_{p+1,l})\|_{C^{i+1}}
	\end{align*}
but we can then  use Lemma \ref{boundsukgen} to obtain
	
	\begin{align*}
		&\|u(\rho_{p+1,j})-u^{M}(\rho_{p+1,j})\|_{C^{m-i}}\leq CN_{p+1}^{2(m-i+2)-M\beta_{p+1}}\ln(N)^{C},
	\end{align*}
	and the bounds obtained in Lemma \ref{itera} and Lemma \ref{boundsfinalproduct} give us
	\begin{align*}
		\|(\sum_{l=1}^{j}\rho_{p+1,l})\|_{C^{i+1}}\leq C_{i}\ln(N_{p+1})^{C_{i}} N_{p+1}^{i+\beta_{p+1}}
	\end{align*}
	so
	\begin{align*}
		&\|\big[u(\rho_{p+1,j})-u^{M}(\rho_{p+1,j})]\cdot\nabla (\sum_{l=1}^{j}\rho_{p+1,l})\|_{C^{m_{0}}}\\
		&\leq C\ln(N_{p+1})^{C}N_{p+1}^{2(m_{0}+2)-M\beta_{p+1}} N_{p+1}\\
		&\leq C\ln(N_{p+1})^{C}N_{p+1}^{2m_{0}+5-M\beta_{p+1}} \leq \frac{1}{10}N_{p+1}^{-1}.
	\end{align*}
	In a completely analogous way we also obtain
	\begin{align*}
		&\|\big[u(\sum_{l=1}^{j-1}\rho_{p+1,l})-u^{M}(\sum_{l=1}^{j-1}\rho_{p+1,l})\big]\cdot\nabla \rho_{p+1,j}\|_{C^{m_{0}}}\leq \frac{1}{10}N_{p+1}^{-1}.
	\end{align*}
	Next we have
	\begin{align*}
		&\|u(\rho_{p+1,j})\cdot\nabla (\sum_{l=0}^{p}\rho_{l})-u^{M}(\rho_{p+1,j})\cdot\nabla (\sum_{l=0}^{p}\P_{M}(\rho_{l}))\|_{C^{m_{0}}}\\
		&\leq\|[u(\rho_{p+1,j})-u^{M}(\rho_{p+1,j})]\cdot\nabla (\sum_{l=0}^{p}\rho_{l})\|_{C^{m_{0}}}+\|u^{M}(\rho_{p+1,j})\cdot\nabla [(\sum_{l=0}^{p}\rho_{l})-(\sum_{l=0}^{p}\P_{M}(\rho_{l}))]\|_{C^{m_{0}}}.
	\end{align*}
We can use the bound 
$$\|(\sum_{l=0}^{p}\rho_{l})\|_{C^{m_{0}}}\leq C\ln(N_{p+1})^{C}$$
to bound
\begin{align*}
	&\|[u(\rho_{p+1,j})-u^{M}(\rho_{p+1,j})]\cdot\nabla (\sum_{l=0}^{p}\rho_{l})\|_{C^{m_{0}}}\leq C\ln(N_{p+1})^{C}N_{p+1}^{2(m_{0}+2)-M\beta_{p+1}}\\
	&\leq \frac{1}{10}N_{p+1}^{-1}
\end{align*}
as we did before, while for the other term we have
\begin{align*}
	&\|u^{M}(\rho_{p+1,j})\cdot\nabla [(\sum_{l=0}^{p}\rho_{l})-(\sum_{l=0}^{p}\P_{M}(\rho_{l}))]\|_{C^{m_{0}}}\\
	&\leq C\ln(N_{p+1})^{C}(N_{p+1})^{m_{0}}\|(\sum_{l=0}^{p}\rho_{l})-(\sum_{l=0}^{p}\P_{M}(\rho_{l}))\|_{C^{m_{0}+1}(B_{N_{p+1}^{-2\beta_{p+1}}\ln(N_{p+1})^4(0)})}\\
	&\leq C\ln(N_{p+1})^{C}N_{p+1}^{m_{0}-2(M-m_{0})\beta_{p+1}}\leq \frac{1}{10}N_{p+1}^{-1}.
\end{align*}
Finally we have
\begin{align*}
	A:=&\\
	&\big[u(\sum_{l=0}^{p}\rho_{l})-\P_{M}(u(\sum_{l=0}^{p}\rho_{l}))\big]\cdot \big[\sum_{l=1}^{2^{j}}g^{1}_{l,j}(x,t)\nabla\sin(N_{p+1}lX(x,t))+\sum_{l=1}^{2^{j}}g^{2}_{l,j}(x,t)\nabla\cos(N_{p+1}lX(x,t))\big]\\
	&+\big[u(\sum_{l=0}^{p}\rho_{l})-\P_{M}(u(\sum_{l=0}^{p}\rho_{l}))\big]\cdot \big[\sum_{l=1}^{2^{j}}\sin(N_{p+1}lX(x,t))\nabla g^{1}_{l,j}(x,t)+\sum_{l=1}^{2^{j}}\cos(N_{p+1}lX(x,t))\nabla g^{2}_{l,j}(x,t)\big].
\end{align*}
We note that
\begin{align*}
	&\|A\|_{C^{m_{0}}}\leq \|u(\sum_{l=0}^{p}\rho_{l})-\P_{M}(u(\sum_{l=0}^{p}\rho_{l}))\|_{C^{m_{0}}(B_{N_{p+1}^{-2\beta_{p+1}}\ln(N_{p+1})^4(0)})}\|\rho_{p+1,j}\|_{C^{m_{0}+1}}\\
	&\leq C\ln(N_{p+1})^{C}N_{p+1}^{-2\beta_{p+1}(M+1-m_{0})}N_{p+1}^{m_{0}+\beta_{p+1}}\leq \frac{1}{10} N_{p+1}^{-1}.
\end{align*}

To obtain the bounds for $\tilde{F}_{p+1,m_{0}}$ we just need to use \eqref{boundh} and Lemma \ref{boundsfinalproduct} to get

$$\|\tilde{F}_{p+1,M}\|_{C^{m_{0}}}\leq C\ln(N_{p+1})^{C}N_{p+1}^{-1+m_{0}+[1-M]\beta_{p+1}}\leq N_{p+1}^{-1}.$$

\end{proof}

We are now ready to prove the last lemma in this paper, which will show that we can iteratively construct our $p+1$-layered solution. For the sake of clarity, we will include in our statement all the previous choices explicitly.

\begin{lemma}\label{finalitera}
	Given $\sum_{i=0}^{p}\rho_{i}(x,t)$ a $p$-layered solution with parameters $(N_{i},\beta_{i},k_{i},K_{i})$ ($i=0,1,...,p$), $N_{p+1}$, $\beta_{p+1}\in (\frac{\beta_{p}}{2},\frac14)$ (chosen according to Remark \ref{remarknp1}), if $K_{p}\geq \lceil\frac{4}{\beta_{p}^2}\rceil+1$ , then given $k_{p+1}\geq \lceil\frac{4}{\beta_{p}^2}\rceil+1$, and $f(x)\in C^{\infty}$ fulfilling
	$\text{supp}f(x)\subset B_{1}(0)$, $f(x)=1$ if $|x|\leq \frac{1}{2}$ and
	$$f(x_{1},x_{2})=f(-x_{1},x_{2})=f(x_{1},-x_{2})=f(-x_{1},-x_{2}),$$
if $N_{p}$ is big enough (depending only on $\beta_{p},\beta_{p+1}$, $k_{p+1}$ and the choice of $f(x)$), then we have that, for $M:=\lceil\frac{1}{\beta_{p+1}^2}\rceil$,
$$\sum_{i=0}^{p}\rho_{i}(x,t)+\sum_{j=1}^{M}\rho_{p+1,j},$$
with $\rho_{p+1,j}$ as in Lemma \ref{itera}, is a $p+1$-layered solution with parameters $(N_{i},\beta_{i},k_{i},K_{i})$ for $i=0,1,2,...,p+1$ with $\rho_{p+1}=\sum_{j=1}^{M}\rho_{p+1,j}$, $F_{p+1}=\sum_{j=0}^{M}F_{p+1,j}+\tilde{F}_{p+1,M}$.
\end{lemma}
\begin{proof}
	$\rho_{p+1},F_{p+1}\in C^{\infty}$ is straightforward by construction, since $\rho_{p+1,1}\in C^{\infty}$, and all the terms in the construction are obtained by taking derivatives, multiplying and composing with smooth functions, while the support condition is a direct consequence of \eqref{support} in Lemma \ref{itera}. Property 2 in Definition \ref{player} is also immediate from the Definition of $N_{p+1}$, namely from Remark \ref{remarknp1}. 
	Property 3 is also trivially true by applying Lemma \ref{itera}. Property 4 follows directly from Lemmas \ref{girot0} and \ref{itera}. Properties 5 and 7 follow from Lemmas \ref{itera} and \ref{boundsfinalproduct}, as well as 
 $k_{p+1}, K_{p}\geq \lceil\frac{4}{\beta_{p}^2}\rceil\geq \lceil\frac{1}{\beta_{p+1}^2}\rceil$ 
 in our choices. Property 6 follows directly from Lemma \ref{forces}. Finally, property 8 follows from our choice of $\beta_{p+1}$ and the bounds for $N_{p+1}$ from Remark \ref{remarknp1}.
	
Bound \eqref{lowfrapriori} follows from Property 5 and Property 2 combined with $k_{i}\geq 2$.

Now, to prove \eqref{derivadasrhoi}, we define
\begin{align*}
    \cos(\alpha_{p+1}(t))&=\frac{\p_{x_{1}}X(x=0,t)}{[(\p_{x_{1}}X(x=0,t))^2+\p_{x_{1}}X(x=0,t)]^2]^{\frac12}},\\ \sin(\alpha_{p+1}(t))&=\frac{\p_{x_{2}}X(x=0,t)}{[(\p_{x_{1}}X(x=0,t))^2+\p_{x_{1}}X(x=0,t)]^2]^{\frac12}}. \end{align*}
By Lemma \ref{itera}, we have that
\begin{align}\label{decomprhop}
	\rho_{p+1}(x,t)=\rho_{p+1,1}(x,t)+\sum_{l=0}^{C_{\beta_{p},\beta_{p+1}}}a_{l}(x,t)\sin(NlX(x,t))+\sum_{l=0}^{C_{\beta_{p},\beta_{p+1}}}b_{l}(x,t)\cos(NlX(x,t))
\end{align}

with
$$\|a_{l}(x,t)\|_{C^{m}},\|b_{l}(x,t)\|_{C^{m}}\leq C_{m,l,\beta_{p},\beta_{p+1}}\ln(N_{p+1})^{C_{m,l,\beta_{p},\beta_{p+1}}}N_{p+1}^{-1+2m\beta_{p+1}},$$
and $\rho_{p+1,1}(x,t)=g^{1}_{1,1}(x,t)\sin(NX(x,t))$.
This already gives us, for $i=1,2$
\begin{align*}
	&|\p_{x_{i}}[\sum_{l=0}^{C_{\beta_{p},\beta_{p+1}}}a_{l}(x,t)\sin(NlX(x,t))+\sum_{l=0}^{C_{\beta_{p},\beta_{p+1}}}b_{l}(x,t)\sin(NlX(x,t))]|(x=0)\\
	&\leq C_{m,l,\beta_{p},\beta_{p+1}}\ln(N_{p+1})^{C_{m,l,\beta_{p},\beta_{p+1}}}\leq N_{p+1}^{\frac{\beta_{p+1}}{8}}.
\end{align*}

On the other hand, we have that
$$\p_{x_{1}}\rho_{p+1,1}(x=0,t)=[\p_{x_{1}}X(x=0,t)^2+\p_{x_{2}}X(x=0,t)^2]^{0.5}(N_{p+1}\cos(\alpha_{p+1}(t)))g^{1}_{1,1}(x=0,t),$$
$$\p_{x_{2}}\rho_{p+1,1}(x=0,t)=[\p_{x_{1}}X(x=0,t)^2+\p_{x_{2}}X(x=0,t)^2]^{0.5}(N_{p+1}\sin(\alpha_{p+1}(t)))g^{1}_{1,1}(x=0,t).$$
Note that, by the definition of $X(x,t)$, we have that
$$\p_{t}\P_{1}(X(x,t))=-\P_{1}(u(\sum_{l=0}^{p}\rho_{l}))\cdot\nabla[\P_{1}(X(x,t))]$$
which, combined with $u(\sum_{l=0}^{p}\rho_{l})(x=0,t)=0$ implies
$$|\p_{t}[\p_{x_{1}}X(x=0,t)^2+\p_{x_{2}}X(x=0,t)^2]^{0.5}|\leq C\|u(\sum_{l=0}^{p}\rho_{l})\|_{C^1}[\p_{x_{1}}X(x=0,t)^2+\p_{x_{2}}X(x=0,t)^2]^{0.5},$$
and using the bounds for $\|u(\sum_{l=0}^{p}\rho_{l})\|_{C^1}$ and $[\p_{x_{1}}X(x=0,t=1)^2+\p_{x_{2}}X(x=0,t=1)^2]^{0.5}=1$ we get, for $N_{p}$ big, for $t\in [1-\tilde{C}N_{p+1}^{\frac{-3\beta_{p+1}}{4}},1]$
$$[\p_{x_{1}}X(x=0,t)^2+\p_{x_{2}}X(x=0,t)^2]^{0.5}\in [\frac{1}{\sqrt{2}},\sqrt{2}].$$
By using the definition of $g^{1}_{1,1}(x,t)$ we can check that
$$g^{1}_{1,1}(x=0,t)=\frac{\mathcal{A}_{1}(t)}{N_{p+1}^{1-\beta_{p+1}}}.$$
But since 
$$|\p_{t}\mathcal{A}_{1}(t)|\leq |\mathcal{A}_{1}(t)3N_{p}^{\beta_{p}}|$$
for $t\in [1-\tilde{C}N_{p+1}^{\frac{-3\beta_{p+1}}{4}},1]$, and using Property 2, for $N_{p}$ big depending only on $\beta_{p},\beta_{p+1}$
$$\mathcal{A}_{1}(t)\in[\frac{1}{\sqrt{2}},\sqrt{2}],$$
which finishes the proof of \eqref{derivadasrhoi} since we have, for the desired times,
$$\p_{x_{1}}\rho_{p+1,1}(x=0,t)\in \cos(\alpha_{p+1}(t))[\frac{N_{p+1}^{\beta_{p+1}}}{2},2N_{p+1}^{\beta_{p+1}}],$$
$$\p_{x_{2}}\rho_{p+1,1}(x=0,t)\in \sin(\alpha_{p+1}(t))[\frac{N_{p+1}^{\beta_{p+1}}}{2},2N_{p+1}^{\beta_{p+1}}].$$

To obtain \eqref{anguloup}, we use again \eqref{decomprhop}, and note that
$$\|u(\sum_{l=0}^{C_{\beta_{p},\beta_{p+1}}}a_{l}(x,t)\sin(NlX(x,t))+\sum_{l=0}^{C_{\beta_{p},\beta_{p+1}}}b_{l}(x,t)\sin(NlX(x,t)))\|_{C^{1}}\leq\ \frac{1}{2} N_{p+1}^{\frac{\beta_{p+1}}{8}}.$$
Furthermore, we have, using Lemma \ref{boundsukgen} that
\begin{align*}
	&\|u(\rho_{p+1,1})-u^{0}(\rho_{p+1,1})\|_{C^{1}}\\
	&\leq \|u^{M}(\rho_{p+1,1})-u^{0}(\rho_{p+1,1})\|_{C^{1}}-\|u(\rho_{p+1,1})-u^{M}(\rho_{p+1,1})\|_{C^{1}}\leq \frac{1}{2} N_{p+1}^{\frac{\beta_{p+1}}{8}}.
\end{align*}
And finally, using the definition of $u^{0}$ in Corollary \ref{defuk}

\begin{align*}
	&\p_{x_{i}}u_{1}^{0}(\rho_{p+1,1})(x=0,t)=C_{0}g^{1}_{1,1}(0,t)(\p_{x_{i}}\sin(NX(x,t))(x=0))\sin(\alpha_{p+1}(t))\cos(\alpha_{p+1}(t))\\
	&=C_{0}\frac{\mathcal{A}_{1}(t)}{N_{p+1}^{1-\beta_{p+1}}}(\p_{x_{i}}\sin(NX(x,t))(x=0))\sin(\alpha_{p+1}(t))\cos(\alpha_{p+1}(t))
\end{align*}

	\begin{align*}
		&\p_{x_{i}}u_{2}^{0}(\rho_{p+1,1})(x=0,t)=-C_{0}g^{1}_{1,1}(0,t)(\p_{x_{i}}\sin(NX(x,t))(x=0))\cos(\alpha_{p+1}(t))^2\\
		&=-C_{0}\frac{\mathcal{A}_{1}(t)}{N_{p+1}^{1-\beta_{p+1}}}(\p_{x_{i}}\sin(NX(x,t))(x=0))\cos(\alpha_{p+1}(t))^2
	\end{align*}
which finishes the proof.

\end{proof}

\section{Proof of the blow-up}\label{secblowup}

We are now ready to prove our final theorem.

\begin{theorem}
    There exists smooth initial conditions $\rho(x,0)\in C^{\infty}_{c}$ and a smooth source $F(x,t)\in C^{\infty}_{c}$ such that the only smooth solution to the IPM equation $\rho(x,t)$ blows up in finite time, more precisely
    $$\text{lim}_{t\rightarrow1}\|\rho(x,t)\|_{C^1}=\infty.$$
\end{theorem}
\begin{proof}
    We will show this using a constructive argument, by considering a sequence of $p$-layered solutions
    $$(\sum_{l=0}^{n}\rho_{l})_{n\in\mathds{N}}$$
    and we will show that
    $\rho_{\infty}(x,t)=\sum_{l=0}^{\infty}\rho_{l}$
    is a smooth, compactly supported solution to the \ref{eq:IPMsystem} system that blows up at $t=1$.
    We start by fixing
    $f(x)\in C^{\infty}$ fulfilling
	$\text{supp}f(x)\subset B_{1}(0)$, $f(x)=1$ if $|x|\leq \frac{1}{2}$ and
	$$f(x_{1},x_{2})=f(-x_{1},x_{2})=f(x_{1},-x_{2})=f(-x_{1},-x_{2}),$$
 and then we define
    $$\alpha_{0}(t)=N_{0}^{-\frac{1}{64}}+\frac{\pi}{2}$$
    $$\rho_{0}(x,t)=f(N_{0}^{\frac{1}{4}}x)\frac{\sin(N_{0}(\cos(\alpha_{0}(t)x_{1}+\sin(\alpha_{0}(t))x_{2})}{N_{0}^{1-\frac{1}{8}}}.$$
    We will choose $N_{0}$ big enough so that $\rho_{0}(x,t)$  with $\beta_{0}=\frac18$, $k_{0},K_{0}\geq 2$ is a $0$-layered solution and also big enough so that Lemma \ref{finalitera} applies with $\beta_{p}=\beta_{p+1}=\frac18$.
    Now, given a $p$-layered solution with $\beta_{p}=\tilde{\beta}$ we construct our $p+1$-layered solution the following way:
    \begin{enumerate}
        \item If $K_{p}\geq\lceil\frac{4}{\tilde{\beta}^2}\rceil +1$ and $N_{p}$ is big enough to apply Lemma \ref{finalitera} with both $\beta_{p}=\tilde{\beta},\beta_{p+1}=\frac{\tilde{\beta}}{2}$ and $\beta_{p}=\frac{\tilde{\beta}}{2},\beta_{p+1}=\frac{\tilde{\beta}}{2}$,  with $k_{p+1}\geq \lceil\frac{4}{\tilde{\beta}^2}\rceil +1$, then we choose our $p+1-$layered solution with $\beta_{p+1}=\frac{\tilde{\beta}}{2}$.
        \item Otherwise, we choose our $p+1-$layered solution with $\beta_{p+1}=\tilde{\beta}$, $k_{p+1}$ the biggest choice of $k_{p+1}$ that our value of $N_{p}$ allows.
    \end{enumerate}
    Note that, since $N_{i+1}\geq N_{i}$, the way our construction works we know option 2 is always a real option: We have $\tilde{\beta}=\beta_{i_{0}}$, where $i_{0}$ is the last step when we used option 1, and we already know that $N_{i_{0}}$ was big enough to apply Lemma \ref{finalitera} with $\beta_{i_{0}}=\beta_{i_{0}+1}=\tilde{\beta}$ and with $k_{i_{0}+1}\geq \lceil\frac{4}{\tilde{\beta}^2}\rceil+1$. The only other possible option is that $\tilde{\beta}=\beta_{0}$, but note that we have chosen already $\beta_{0},N_{0}$ so that step 2 is an option. Note also that, every time we apply step 1, since  $N_{p}$ increases, the maximum value of $k_{p+1}$ that step 2 allows must remain the same or grow.

    Therefore, we can define the sequence $(\sum_{l=0}^{n}\rho_{l})_{n\in\mathds{N}}$. Furthermore, since $\rho_{p+1},F_{p+1}=0$ if $t\in[0,1-\tilde{C}N_{p}^{-\frac{3\beta_{p}}{4}}]$ and $2N_{p+1}^{-\frac{3\beta_{p+1}}{4}}\leq N_{p}^{-\frac{3\beta_{p}}{4}}$, we have that for any given $\bar{t}\in[0,1)$, 
    $$\sum_{l=0}^{\infty}\rho_{l}(x,\bar{t})=\sum_{l=0}^{C_{\bar{t}}}\rho_{l}(x,\bar{t})$$
    $$\sum_{l=0}^{\infty}F_{l}(x,\bar{t})=\sum_{l=0}^{C_{\bar{t}}}F_{l}(x,\bar{t})$$
    so the limits $\theta_{\infty}=\sum_{i=0}^{\infty}\theta_{p}$, $F_{\infty}=\sum_{i=0}^{\infty}F_{p}$ are well defined and give a solution to the \ref{eq:IPMsystem} system for $t\in[0,1)$. Note also that, by construction, $\rho_{\infty},F_{\infty}$ are smooth and compactly supported.
    Next, we show that
    $$\text{lim}_{p\rightarrow\infty}\beta_{p}=0.$$
    For this we argue by contradiction: If $\beta_{p}\geq \epsilon$, then there is a last time we apply step 1 in our construction, which means that either $N_{p}$ or $K_{p}$ is bounded from above. Since we know $N_{p+1}\geq 4N_{p}$, it must be that $K_{p}$ is bounded, say by $M$. However, as $N_{p}$ grows (monotonously and without bound) we know from Lemma \ref{finalitera} that $k_{p+1}$ must also tend to infinity. Then, consider $p_{0}$ such that for any $P\geq p_{0}$, we have $k_{P}>M+1$. By assumption we then have
    $$\|\sum_{l=0}^{P}\rho_{l}\|_{C^{M+1}}> 2N_{P}^{M+1}.$$
    But, using $N_{i+1}\geq 4N_{i}$ and property 5 in Definition \ref{player}
    \begin{align*}
        &\|\sum_{l=0}^{P}\rho_{l}\|_{C^{M+1}}\leq \|\sum_{l=0}^{p_{0}-1}\rho_{l}\|_{C^{M+1}}+ \sum_{l=p_{0}}^{P}N_{l}^{M+1} \leq C+ N_{P}^{M+1}\sum_{p=p_{0}}^{P}4^{p-P}\leq C+\frac{3}{2}N_{P}^{M+1}
    \end{align*}
    which then gives us, for any $P$
    $$C+\frac{3}{2}N_{P}^{M+1}\geq \|\sum_{l=0}^{P}\rho_{l}\|_{C^{M+1}}> 2N_{P}^{M+1}$$
    and taking $P$ big gives a contradiction.
    Therefore, $\beta_{p}\rightarrow0$. In particular, this immediately implies that $F_{\infty}=\sum_{l=0}^{\infty}F_{l}\in C^{\infty}$ since $\beta_{p}$ is decreasing and tends to zero, so for any value of $m$ there exists $C_{m}$ such that $p\geq C_{m}$ implies $m\leq \frac{1}{4\beta_{p}^2}-10 $, and thus, using property 6 in Definition \ref{player} we have that

    $$\|\sum_{l=0}^{\infty}F_{l}\|_{C^{m}}\leq C+\|\sum_{l=C_{m}}^{\infty}F_{l}\|_{C^{m}}\leq C+\sum_{l=C_{m}}^{\infty}N_{l}^{-\frac14}\leq C$$
    where in the last step we used $N_{p+1}\geq 4N_{p}$ to show that the sum converges.
    The only thing left to show is that
    $$\text{lim}_{t\rightarrow 1}\|\rho_{\infty}\|_{C^1}=\infty.$$
    But for $t=1-\tilde{C}N_{p}^{-\frac{3\beta_{p}}{4}}$, using property 4 in \ref{player} we have that
    $$\rho_{\infty}(x,t=1-\tilde{C}N_{p}^{-\frac{3\beta_{p}}{4}})=\sum_{l=0}^{p}\rho_{l}(x,t=1-\tilde{C}N_{p}^{-\frac{3\beta_{p}}{4}})$$
    and we have that, if $z=x_{1}\cos(\alpha_{p}(t))+x_{2}\sin(\alpha_{p}(t))$
    \begin{align*}
       &|(\frac{\p}{\p z}\rho_{\infty})(x=0,t=1-N_{p}^{-\frac{3\beta_{p}}{4}})|\geq |(\frac{\p}{\p z}\rho_{p})(x=0,t=1-N_{p}^{-\frac{3\beta_{p}}{4}})|-|(\frac{\p}{\p z}\sum_{l=0}^{p-1}\rho_{l})(x=0,t=1-N_{p}^{-\frac{3\beta_{p}}{4}})|\\
       &\geq \frac{N_{p}^{\beta_{p}}}{2}- 2N_{p}^{\frac{\beta_{p}}{8}}\geq \frac{N_{p}^{\beta_{p}}}{4}
    \end{align*}
    where we used \eqref{lowfrapriori} and \eqref{cotasnb}, and since $N_{p}^{\beta_{p}}$ is unbounded this finishes the proof.

\end{proof}

\section*{Acknowledgment}
 This work is supported in part by the Spanish Ministry of Science
and Innovation, through the “Severo Ochoa Programme for Centres of Excellence in R$\&$D (CEX2019-000904-S \& CEX2023-001347-S)” and 152878NB-I00. We were also partially supported by the ERC Advanced Grant 788250, and by the SNF grant FLUTURA: Fluids, Turbulence, Advection No. 212573.

\bibliographystyle{siam}


\end{document}